\documentclass[11pt]{article}

\usepackage[margin=1in,letterpaper]{geometry}
\usepackage{authblk} % Add the authblk package for author details

\usepackage[
    backend=biber,
    style=numeric, % Changed from 'alphabetic' to 'numeric'
    maxnames = 99,
    url=true,      % Ensure URLs are displayed
    doi=true,      % Include DOIs (if needed)
    isbn=false,    % Hide ISBN (optional)
]{biblatex}

\newif\iftranreview % Hide identifying material only in the review version.
\usepackage{microtype}
\usepackage[normalem]{ulem} % Visible deletion marks in tracked revisions.
\usepackage{amssymb}
\usepackage{amsmath}
\usepackage{mathtools}
\usepackage{tikz-cd}
\usepackage{amsthm}
\usepackage{thmtools, thm-restate}
\usepackage{accents}
\usepackage{enumitem} % Include the enumitem package
\usepackage{adjustbox} % Fit figures within the available line width.
\usepackage{graphicx} % Required for inserting images
\usepackage{tikz}
\usepackage{pgfplots}
\pgfplotsset{compat=newest}
\usepackage{float}
\usepackage{needspace}
\usepackage{array}

\usepackage{hyperref} % Add the hyperref package for clickable links
\hypersetup{
    hypertexnames=false, % Unique PDF destinations for shared and restated counters.
    colorlinks=true,
    linkcolor=blue,  % Change to a preferred color
    urlcolor=blue,
    citecolor=blue
}
\usepackage[capitalise]{cleveref}
\let\cref\Cref

\newtheorem{theorem}{Theorem}[section]
\newtheorem{lemma}[theorem]{Lemma}

\newtheorem{proposition}[theorem]{Proposition}
\newtheorem{corollary}[theorem]{Corollary}
\theoremstyle{definition}
\newtheorem{definition}[theorem]{Definition}
\newtheorem{example}[theorem]{Example}

\newtheorem{remark}[theorem]{Remark}

\definecolor{darkblue}{rgb}{0.0, 0.0, 0.8}
\definecolor{darkred}{rgb}{0.8, 0.0, 0.0}
\definecolor{darkgreen}{rgb}{0.0, 0.8, 0.0}

\newcommand{\rmH}{\mathrm{H}}

\newcommand{\bbF}{\mathbb{F}}

\newcommand{\bbN}{\mathbb{N}}

\newcommand{\bbR}{\mathbb{R}}
\newcommand{\bbS}{\mathbb{S}}
\newcommand{\bbT}{\mathbb{T}}

\newcommand{\caB}{\mathcal{B}}
\newcommand{\caC}{\mathcal{C}}
\newcommand{\caD}{\mathcal{D}}

\newcommand{\dgh}{d_\mathrm{GH}}
\newcommand{\dhi}{d_\mathrm{HI}}
\newcommand{\di}{d_\mathrm{I}}

\newcommand{\db}{d_\mathrm{B}}
\newcommand{\de}{d_\mathrm{E}}
\newcommand{\dih}{d_\mathrm{IH}}

\newcommand{\kernel}{\operatorname{ker}}
\newcommand{\VR}{\operatorname{VR}}
\newcommand{\im}{\operatorname{im}}
\newcommand{\id}{\operatorname{id}}

\newcommand{\dis}{\operatorname{dis}}
\newcommand{\codis}{\operatorname{codis}}

\newcommand{\diam}{\operatorname{diam}}

\newcommand{\field}{\Bbbk}

\newcommand{\TC}{\mathbf{TC}}
\newcommand{\cat}{\mathbf{cat}}

\newcommand{\secat}{\mathbf{secat}}

\newcommand{\cl}{\mathbf{cl}}
\newcommand{\zcup}{\mathbf{zcl}}
\newcommand{\invariant}{\mathbf{I}}
\newcommand{\nil}{\mathbf{nil}}

\usepackage{dutchcal}
\newcommand{\topo}{\mathcal{T\!o\!p}}
\newcommand{\Vect}{\mathcal{V\!e\!c}}
\newcommand{\Int}{\mathcal{I\!n\!t}}

\newcommand{\mor}{\operatorname{Mor}}
\newcommand{\ob}{\operatorname{Ob}}
\newcommand{\rp}{\mathbb{R}\mathrm{P}}
\newcommand{\wedgeS}{\vee^n \mathbb{S}}
\newcommand{\pathspace}{\operatorname{Path}}

\definecolor{dgmcolor}{RGB}{255,20,147}
\colorlet{dgmvalue1}{dgmcolor!20!white}
\colorlet{dgmvalue2}{dgmcolor!40!white}
\colorlet{dgmvalue3}{dgmcolor!60!white}
\colorlet{dgmvalue4}{dgmcolor!80!white}
\tikzset{dgm region/.style={draw=dgmvalue#1,fill=dgmvalue#1,fill opacity=0.45}}

\newcommand{\introheading}[1]{\paragraph{#1}}

\AtEveryBibitem{%
  \iffieldundef{fjournal}
    {}% do nothing if fjournal undefined
    {\clearfield{journal}%
     \fieldalias{fjournal}{journal}}%
}

\AtEveryBibitem{%
  \ifboolexpr{
    test {\ifentrytype{online}}
    or
    test {\ifentrytype{misc}}
  }
    {}
    {
      \iffieldundef{doi}
        {}
        {\clearfield{url}}%
    }%
}

\bibliography{ptc}

\title{Persistence and Topological Complexity}

\author[1]{Facundo Mémoli}
\author[2]{Ling Zhou}
\affil[1]{Department of Mathematics, Rutgers University\\\texttt{facundo.memoli@rutgers.edu}}
\affil[2]{Department of Mathematics, University of Oklahoma\\\texttt{ling.zhou@ou.edu}}

\date{\today}

\begin{document}

\maketitle

\begin{abstract}
Topological complexity is a homotopy invariant that measures the minimal number of continuous rules required for motion planning in a space. In this work, we introduce persistent analogs of topological complexity and its cohomological lower bound, the zero-divisor-cup-length, for persistent topological spaces, and establish their stability. For Vietoris--Rips filtrations of metric spaces, we show that the erosion distances between these persistent invariants are bounded above by twice the Gromov--Hausdorff distance. We develop metric models of classical topological spaces based on cube boundaries with the \(\ell_\infty\) metric and their products and wedges, allowing explicit computations of persistent invariants and their distances. These examples show that, in certain cases, persistent topological complexity and zero-divisor-cup-length distinguish spaces that persistent homology, persistent LS-category, and persistent cup-length fail to distinguish.
\end{abstract}

\tableofcontents

\section{Introduction}

Topological complexity, introduced by Farber~\cite{farber2003,farber2004instabilities}, is a homotopy invariant that quantifies the minimal number of continuous rules needed to specify a motion plan between any two points in a topological space.
We study a persistent version of topological complexity,
defined on the structure maps of a filtration, and use it
to obtain Gromov--Hausdorff lower bounds.

\introheading{Motion Planning Problem.}
Let \( X \) be a nonempty topological space.
Define \( \pathspace(X) \) as the space of all continuous paths \(\gamma: [0,1] \to X\), equipped with the compact-open topology, and let \( p_X: \pathspace(X) \to X \times X \) be the map assigning to each path its pair of endpoints \( (\gamma(0), \gamma(1)) \).

The \emph{motion planning problem} seeks a function (not necessarily continuous) \( s: X \times X \to \pathspace(X) \), called a \emph{motion planner}, such that \( p_X \circ s = \id_{X \times X} \).
In other words, a motion planner assigns to each ordered pair
of points in \(X\) a continuous path joining them.
When \( X \) is path-connected, such a function \( s \) always exists when no further constraints are imposed.
However, Farber showed in \cite[Theorem 1]{farber2003} that the existence of a \emph{continuous} motion planner implies that \( X \) is contractible.

\introheading{Topological Complexity.}
Requiring $X$ to be contractible, equivalently requiring a continuous global section of $p_X$, is restrictive in many practical and theoretical contexts.
This motivates studying continuous local motion planners.
Farber introduced the \emph{topological complexity} \(\TC(X)\) as the minimal integer \( n \geq 0 \) such that \( X \times X \) admits a cover by \( n+1 \) open sets \( U_1,\dots, U_{n+1} \), each having a continuous local section \( s_i: U_i \to \pathspace(X) \) with \( p_X \circ s_i = \iota_{U_i} \), where \(\iota_{U_i}:U_i\hookrightarrow X\times X\) is the inclusion.
For instance, \(\TC(X)=0\) when \(X\) is contractible, and \(\TC(X)=\infty\) when \(X\) has more than one path component.
We use the reduced topological complexity, one less than Farber's original invariant~\cite{farber2003}, consistent with the standard convention for sectional category recalled below.

Topological complexity has also been developed for maps
\cite{murillo2021topological,scott2022ls,scott2022topological}, and in
relative forms by Farber \cite{farber2008invitation} and Short
\cite{short2018relative}.
For a map \(f:X\to Y\), the invariant \(\TC(f)\) used here
measures the complexity of choosing paths in \(Y\) from
\(f(x_0)\) to \(f(x_1)\) by continuous rules on local domains in
\(X\times X\); see \cref{def:TC of f}.
This viewpoint is essential here, since our persistent topological complexity
is defined by evaluating $\TC$ on the structure maps of a filtration.

Rudyak \cite{rudyak2010higher} introduced higher, or sequential, topological complexity, measuring the instability of motion planners that must guide a system through several prescribed intermediate stages; their basic properties and computations for familiar spaces are developed in \cite{basabe2014higher}.
Higher topological complexity has since been extended to maps, first via relative LS-category methods~\cite{rudyak2022relative} and later through a homotopy-invariant treatment~\cite{ipanaque2023higher}.
Extending our persistent framework to these higher invariants is a natural direction for future work.

\introheading{Connections to Sectional Category and LS-category.}

The definition of \(\TC(X)\) builds on the classical notion of the \emph{Schwarz\footnote{The name Schwarz is also translated as \v{S}varc.} genus} of a fibration \cite{schwarz1961genus}, also known as the \emph{sectional category} \cite{james1978category}.
Farber’s formulation was inspired by the work of Smale~\cite{smale1987topology} and Vassil\'{i}ev~\cite{Vassilev1993}, who applied the sectional category to study the complexity of algorithms for solving polynomial equations, as discussed in Farber’s survey~\cite{farber2006topology}.
Given a fibration \( p: E \to B \), the sectional category of $p$ is the smallest integer \( n \) (or \( +\infty \)) such that \( B \) can be covered by \( n+1 \) open subsets \( U_1,\dots, U_{n+1} \), each admitting a continuous section \( s_i: U_i \to E \) satisfying \( p \circ s_i = \iota_{U_i} \), where \(\iota_{U_i}:U_i\hookrightarrow B\) is the inclusion \cite{schwarz1961genus}.

Historically, the sectional category was introduced as a generalization of the \emph{Lyusternik-Schnirelmann category (LS-category)} \cite{lusternik1934methodes}.
In particular, the LS-category \(\cat(X)\) of a space $X$ coincides with the sectional category of the based path space fibration $p_{X,0}:\pathspace_0(X)\to X$ where $X$ is a based path-connected space, $\pathspace_0(X)$ consists of all paths in \(X\) starting from the base point, and $p_{X,0}$ maps each path $\gamma$ to $\gamma(1)$; see \cite[Section 6.1]{schwarz1961genus} or \cref{rmk:cat as secat}.

Farber identified the topological complexity \(\TC(X)\) of a space \( X \) with the sectional category of the path fibration \( p_X: \pathspace(X) \to X \times X \) \cite{farber2003}.
Thus, both the topological complexity and the LS-category emerge as special cases of the more general sectional category.

\introheading{Cohomological Lower Bound for Sectional Category.}

The \emph{cup-length} \(\cl(X)\) of a space \(X\) is a well-known cohomological lower bound for the LS-category, with the origins of this lower bound being traceable to \cite{FroloffElsholz1935}. While \cite{FroloffElsholz1935} framed their results in terms of cycle intersections, these can be reinterpreted using cup products of cocycles, assuming a manifold setting and invoking Poincaré duality.
%https://math.berkeley.edu/~hutching/teach/215b-2011/cup.pdf

In analogy to the above idea, Schwarz proved in \cite[Theorem 4]{schwarz1961genus} that the sectional category similarly admits a cohomological lower bound. Specifically, the sectional category of a fibration is bounded below by the largest integer \( n \) for which there exist \( n \) cohomology classes \(\xi_1, \dots, \xi_n \in \ker(p^*: \rmH^*(B) \to \rmH^*(E))\) such that their cup product \(\xi_1 \smile \cdots \smile \xi_n\) is non-zero.
Here, \(\rmH^*(\cdot)\) denotes the cohomology ring of a space, and \(p^*\) is the induced map on cohomology.

In the context of topological complexity, the analogous cohomological lower bound is referred to as the \emph{zero-divisor-cup-length}, denoted \(\zcup(X)\) \cite{farber2003}.
In other words, $\zcup(X) \leq \TC(X)$, and $\zcup(X)$ provides a powerful tool for estimating \(\TC(X)\).

\introheading{Persistent Invariants.}

Topological invariants play a central role in topological data analysis (TDA), where \emph{persistent homology} captures how homology evolves across a filtration of spaces~\cite{frosini1990distance,frosini1992measuring,robins1999towards,zomorodian2005computing,cohen2007stability,edelsbrunner2008persistent,carlsson2009topology,carlsson2020persistent}.

More broadly, \emph{persistent invariants} generalize this idea by tracking the evolution of topological invariants across scale.
A fundamental example is the \emph{rank invariant} of a persistence module, in particular, of a persistent homology module~\cite{carlsson2007theory}.
General frameworks for persistent invariants have been developed in~\cite{puuska2017erosion,bergomi2019rank,giunti2021amplitudes,memoli2024PersistentCup}.

Following~\cite{memoli2024PersistentCup}, a \emph{categorical invariant} on a category $\mathcal{C}$ is a function $\invariant: \operatorname{Ob}(\mathcal{C}) \sqcup \operatorname{Mor}(\mathcal{C}) \to \mathbb{N} \cup \{\infty\}$ satisfying two conditions: it assigns to each identity morphism the same value as its corresponding object, and it is non-increasing under composition.
This notion encompasses a broad class of topological invariants, including the rank invariant~\cite%[Definition 11]
{carlsson2007theory}, cup-length \cite{contessoto2022persistent}, and LS-category \cite{memoli2024PersistentCup}.

Let \(\Int\) denote the poset of closed intervals
\([t,s]\subset\bbR\), with \(t\leq s\), ordered by inclusion.
We regard \(\Int\) as a category whose morphisms are
inclusions of intervals.
Given a functor \(X_\bullet: (\mathbb{R}, \leq) \to \mathcal{C}\), referred to as a persistent object, and a categorical invariant \(\invariant\), one obtains a \emph{persistent invariant} by assigning to each interval \([t,s] \in \Int\) the value \(\invariant(X_t \to X_s)\), where \(X_t \to X_s\) is the structure map in the filtration. This assignment defines a functor
\(\invariant(X_\bullet):(\Int,\subseteq)\to
(\mathbb{N}\cup\{\infty\},\geq)\).

A fundamental result of \cite{memoli2024PersistentCup} shows that these persistent lifts are stable: the erosion distance (see \cref{def:de}) between the persistent invariants $\invariant(X_\bullet)$ and $\invariant(Y_\bullet)$ is bounded above by the interleaving distance (see \cref{def:interleaving}) between $X_\bullet$ and $Y_\bullet$.

\introheading{Gromov--Hausdorff distance.}

The Gromov--Hausdorff distance is a fundamental and well-studied notion of distance between compact metric spaces \cite{edwards1975structure,burago2001course,gromov2007metric}.
While it provides a natural framework for comparing spaces up to metric isometry, its exact value is known only in a limited number of cases \cite{memoli2012some,ji2021gromov,adams2022gromov,talipov2022gromov,lim2021gromov,harrison2023quantitative,saul2024gromov,martin2024some}.
As a result, much of the literature has focused on deriving effective lower and upper bounds; see \cite{lim2021gromov} for an overview.

Classical lower bounds for the Gromov--Hausdorff distance between
compact metric spaces $X$ and $Y$ include the \emph{diameter lower
bound} $\tfrac12|\diam(X)-\diam(Y)|$
(\cite[page 255]{burago2001course}) and bounds derived from
persistent homology \cite{chazal2009gromov}, among others
\cite{memoli2012some}.
The persistent homology bounds use the Vietoris--Rips filtrations
\(\VR_\bullet(X)\) and \(\VR_\bullet(Y)\).
For \(t>0\), \(\VR_t(X)\) is the geometric realization of the
simplicial complex whose simplices are the finite subsets of \(X\)
of diameter less than \(t\).
There are cases where spaces have identical diameters but differ in persistent homology, so persistent homology gives a stronger lower bound.
For example, this occurs when comparing round spheres of different dimensions equipped with their geodesic metrics; see \cite[Proposition 9.38]{lim2024vietoris}.

More recently, persistent invariants beyond persistent homology have been employed to obtain sharper lower bounds on the Gromov--Hausdorff distance \cite{zhou2023beyond,zhou2023persistent,memoli2024persistenthomotopy,memoli2025ephemeral,medina2025persistent}.
These invariants are also stable under perturbations and capture finer topological or algebraic information than persistent homology in certain cases.

\introheading{Our Contributions.}

In this work, we introduce persistent analogues of topological
complexity and zero-divisor-cup-length and establish their stability.
To investigate their distinguishing power, we develop
\(\ell_\infty\) metric realizations of classical topological spaces
whose Vietoris--Rips filtrations can be described explicitly up to
homotopy. These constructions provide examples in which persistent
invariants and their distances can be computed explicitly, allowing
direct comparisons of their distinguishing power and the
Gromov--Hausdorff lower bounds they yield.

A \emph{persistent (topological) space} is a functor
\(X_\bullet\colon(\mathbb{R},\leq)\to\topo\), where
\(\topo\) denotes the category of topological spaces.
If each \(X_t\) is nonempty and path-connected, we call
\(X_\bullet\) a \emph{persistent path-connected space}.
Our construction builds on existing extensions of \(\TC\)
and \(\zcup\) from spaces to maps~\cite{murillo2021topological,
scott2022ls}.
For an inclusion \(X_t\hookrightarrow X_s\), the value
\(\TC(X_t\hookrightarrow X_s)\) measures the complexity of
choosing paths in \(X_s\) between points of \(X_t\) by
continuous local rules.
Following the persistent-invariant framework of
\cite{contessoto2022persistent,memoli2024PersistentCup},
we define the \emph{persistent topological complexity}
and \emph{persistent zero-divisor-cup-length} as functors
\[
\TC(X_\bullet),\zcup(X_\bullet)\colon
(\Int,\subseteq)\longrightarrow
(\mathbb{N}\cup\{\infty\},\geq).
\]
Their values on \([t,s]\) are obtained by evaluating
\(\TC\) and \(\zcup\), respectively, on the structure map
\(X_t\to X_s\), using a fixed coefficient field for \(\zcup\).
See \cref{def:p_TC,def:p_zcup} for details.

The classical lower bound for topological complexity given by zero-divisor-cup-length~\cite{farber2003, scott2022ls} extends to the persistent setting.
If each $X_t$ is an absolute neighborhood retract (ANR),
\cref{prop:persistent-tc-bounds} gives
\[
\zcup(X_\bullet) \leq \TC(X_\bullet).
\]

We then establish the erosion stability of these
persistent invariants. 
Earlier erosion-stability results for persistent cup-length and
related persistent invariants
\cite{contessoto2022persistent,memoli2024PersistentCup}
were stated in terms of the Blumberg--Lesnick homotopy
interleaving distance $\dhi$~\cite{blumberg2023universality}.
We, however, work with the interleaving-up-to-homotopy distance $\dih$ (see \cref{def:d_IH}), which is a more direct approach for our purposes.
Our \cref{thm:stability} shows that every homotopy-invariant
categorical invariant is stable with respect to $\dih$, and that its
restriction to Vietoris--Rips filtrations is $2$-Lipschitz with respect
to the Gromov--Hausdorff distance.

\begin{restatable}[Stability of persistent invariants]{theorem}{stabilitytcnew}
\label{thm:stability}
Let $\invariant$ be a categorical invariant of maps whose value depends only on the homotopy class of the map; that is, $\invariant(f)=\invariant(g)$ whenever $f\simeq g$.
Then, the following results hold:
\begin{itemize}
\item For any two persistent spaces $X_\bullet, Y_\bullet :(\bbR,\leq) \to \topo$,
\begin{equation}\label{eq:stability1}
    d_{\mathrm{E}}(\invariant (X_\bullet), \invariant (Y_\bullet)) \leq \dih(X_\bullet, Y_\bullet).
\end{equation}
\item     For the Vietoris-Rips filtrations $\VR_\bullet(X)$ and $\VR_\bullet(Y)$ of compact metric spaces $X$ and $Y$,
\begin{equation}\label{eq:dE-dGH}
d_{\mathrm{E}}\left(\invariant \left(\VR_\bullet (X)\right),\invariant \left(\VR_\bullet (Y)\right)\right)\leq 2\cdot d_{\mathrm{GH}}(X,Y).
\end{equation}
\end{itemize}
\end{restatable}

In terms of distinguishing power, neither persistent topological complexity nor persistent zero-divisor-cup-length dominates the other, as demonstrated in \cref{ex:tc-zcl-cat-cl-main}. This behavior parallels the case of persistent LS-category and persistent cup-length, where neither invariant dominates the other; see \cite[Example 2.34]{memoli2024PersistentCup}.

\subsubsection*{Main quantitative comparisons.}

We study several metric realizations of classical spaces in
\cref{sec:examples-main} to compare the distinguishing power of
persistent \(\TC\), \(\zcup\), \(\cat\), \(\cl\), and persistent homology.
These include constructions from cube boundaries with the
\(\ell_\infty\) metric and metrics arising from geodesic spheres;
see \cref{table:persistent-distinguishing-main}.

\begin{figure}[!t]
\centering
\begin{adjustbox}{max width=\linewidth}
\begin{tikzpicture}[font=\small,line join=round]
% Each torus is schematic; the corner (0,0) is used for the wedge.
\tikzset{pics/tc-square-torus/.style={code={
\filldraw[fill=black!18,draw=black!65]
(-3.584,-1.216)--(-1.024,-1.216)--(-1.024,-1.671)--(-3.584,-1.671)--cycle;
\filldraw[fill=black!28,draw=black!65]
(-1.024,-1.216)--(0,0)--(0,-0.455)--(-1.024,-1.671)--cycle;
\filldraw[fill=black!22,draw=black!65]
(-2.104,-0.361)--(-1.064,-0.361)--(-1.064,-0.816)--(-2.104,-0.816)--cycle;
\filldraw[fill=black!12,draw=black!65]
(-2.52,-0.855)--(-2.104,-0.361)--(-2.104,-0.816)--(-2.52,-1.31)--cycle;
\filldraw[even odd rule,fill=black!9,draw=black!65]
(-3.584,-1.216)--(-1.024,-1.216)--(0,0)--(-2.56,0)--cycle
(-2.52,-0.855)--(-1.48,-0.855)--(-1.064,-0.361)--(-2.104,-0.361)--cycle;
}}}
\node[align=center,anchor=south,text width=3.85cm] at (1.9,9.1) {(a) Torus};
\node[align=center,anchor=south,text width=3.85cm] at (6,9.1) {(b) Mickey Mouse};
\node[align=center,anchor=south,text width=3.85cm] at (10.1,9.1) {(c) Two tori};
\node[align=center,anchor=south,text width=3.85cm] at (14.2,9.1) {(d) Torus and Mickey Mouse};
\node[align=center,anchor=north,font=\footnotesize] at (1.9,8.95)
{\(T:=(\square^1\times\square^1,d_{\max})\)};
\node[align=center,anchor=north,font=\footnotesize] at (6,8.95)
{\(M:=(\square^1\vee\square^2\vee\square^1,d_\vee)\)};
\node[align=center,anchor=north,font=\footnotesize] at (10.1,8.95)
{\(\begin{gathered}T\vee T=\bigl((\square^1\times\square^1)\vee\\
(\square^1\times\square^1),d_\vee\bigr)\end{gathered}\)};
\node[align=center,anchor=north,font=\footnotesize] at (14.2,8.95)
{\(\begin{gathered}T\vee M=\bigl((\square^1\times\square^1)\vee\\
\square^1\vee\square^2\vee\square^1,d_\vee\bigr)\end{gathered}\)};
% Torus.
\begin{scope}[shift={(3.065,7.25)},scale=0.65]
\pic[transform shape] {tc-square-torus};
\end{scope}
% Mickey Mouse: a cube boundary and two loops.
\begin{scope}[shift={(6,6.17)},scale=0.62]
\filldraw[fill=black!9,draw=black!65]
(0,1.4)--(-1,0.9)--(0,0.4)--(1,0.9)--cycle;
\filldraw[fill=black!18,draw=black!65]
(-1,0.9)--(0,0.4)--(0,-1)--(-1,-0.5)--cycle;
\filldraw[fill=black!28,draw=black!65]
(0,0.4)--(1,0.9)--(1,-0.5)--(0,-1)--cycle;
\draw[black!45,dashed] (-1,-0.5)--(0,0)--(1,-0.5) (0,0)--(0,1.4);
\draw[thick,black!75] (0,1.4)--(-1,1.7)--(-1.3,2.7)--(-0.3,2.4)--cycle;
\draw[thick,black!75] (0,1.4)--(1,1.7)--(1.3,2.7)--(0.3,2.4)--cycle;
\fill[black] (0,1.4) circle (2pt);
\end{scope}
% Two tori meet only at a single corner.
\begin{scope}[shift={(10.1,6.7)},scale=0.47]
\pic[transform shape] {tc-square-torus};
\begin{scope}[shift={(3.584,1.671)}]
\pic[transform shape] {tc-square-torus};
\end{scope}
\end{scope}
\fill[black] (10.1,6.7) circle (1.7pt);
% A torus, two loops, and a cube boundary meet at one corner.
\begin{scope}[shift={(14.5,6.7)},scale=0.5]
\pic[transform shape] {tc-square-torus};
\begin{scope}[rotate=15]
\filldraw[fill=black!9,draw=black!65]
(0,0)--(1,0.5)--(2,0)--(1,-0.5)--cycle;
\filldraw[fill=black!18,draw=black!65]
(0,0)--(1,-0.5)--(1,-1.9)--(0,-1.4)--cycle;
\filldraw[fill=black!28,draw=black!65]
(1,-0.5)--(2,0)--(2,-1.4)--(1,-1.9)--cycle;
\draw[black!45,dashed] (0,-1.4)--(1,-0.9)--(2,-1.4) (1,-0.9)--(1,0.5);
\end{scope}
\draw[thick,black!75] (0,0)--(-0.2,1.15)--(-1.05,1.95)--(-0.85,0.8)--cycle;
\draw[thick,black!75,rotate=8] (0,0)--(0.2,1.15)--(1.05,1.95)--(0.85,0.8)--cycle;
\end{scope}
\fill[black] (14.5,6.7) circle (1.7pt);
% Persistent invariant panels: each point (s,t) represents [s,t].
\node[align=center] at (1.9,4.6)
{\(\cat(\VR_\bullet(T))\)\\or \(\cl(\VR_\bullet(T))\)};
\node[align=center] at (6,4.6)
{\(\cat(\VR_\bullet(M))\)\\or \(\cl(\VR_\bullet(M))\)};
\node[align=center] at (10.1,4.6)
{\(\TC(\VR_\bullet(T\vee T))\)\\or \(\zcup(\VR_\bullet(T\vee T))\)};
\node[align=center] at (14.2,4.6)
{\(\TC(\VR_\bullet(T\vee M))\)\\or \(\zcup(\VR_\bullet(T\vee M))\)};
\foreach \center/\value in {1.9/2,6/1,10.1/4,14.2/3}{
\begin{scope}[shift={(\center-1.45,0.7)}]
\filldraw[dgm region=\value]
(0,0)--(0,2)--(2,2)--cycle;
\draw[black!65] (2,2)--(3,3);
\draw[->] (0,0)--(3.15,0) node[right] {$s$};
\draw[->] (0,0)--(0,3.15) node[above] {$t$};
\draw (2,0.04)--(2,-0.04) node[below] {$2$};
\draw (0.04,2)--(-0.04,2) node[left] {$2$};
\node[font=\normalsize] at (0.7,1.15) {$\value$};
\node[font=\normalsize] at (0.7,2.6) {$0$};
\end{scope}
}
\end{tikzpicture}
\end{adjustbox}
\caption[\protect\(\ell_\infty\) models and their persistent invariants]{%
Four $\ell_\infty$ models and their persistent invariants.
Here \(\square^n=\partial[-1,1]^{n+1}\) carries the \(\ell_\infty\) metric;
products carry the maximum metric and wedges the gluing metric.
The top row shows schematic drawings of the spaces, with black dots marking the common basepoints.
For all four spaces, $\VR_t$ is homotopy equivalent to the underlying space for $0<t\leq2$ and is contractible for $t>2$.
In each bottom panel, $(s,t)$ represents the interval $[s,t]$. The displayed invariant has the indicated value on $0<s\leq t\leq2$ and vanishes elsewhere.
The pair $T,M$ has identical persistent homology but different persistent LS-category and cup-length. The pair $T\vee T,T\vee M$ has identical persistent homology, LS-category, and cup-length but different persistent $\TC$ and $\zcup$, over every coefficient field.
See \protect\cref{prop:linfty-persistent-values-main,ex:tc-zcl-persistent-main}.
}
\label{fig:cubic-comparison-intro}
\end{figure}

Our constructions simplify the torus versus Mickey Mouse comparison of
\cite{memoli2024PersistentCup}. Write \(\bbS^k\) for the
\(k\)-dimensional sphere and \(\vee\) for the wedge sum, formed by
identifying chosen basepoints. In that work, the torus
\(\bbT^2:=\bbS^1\times\bbS^1\) and the Mickey Mouse space
\(\bbS^1\vee\bbS^2\vee\bbS^1\), with metrics built from geodesic
spheres, show that persistent LS-category and cup-length distinguish
spaces more effectively than persistent homology. The analysis is
more involved because only partial information about their
Vietoris--Rips filtrations is available.

We replace each sphere \(\bbS^k\) by the homeomorphic cube boundary
\(\square^k:=\partial[-1,1]^{k+1}\).
On this boundary, we use the \(\ell_\infty\)
metric, given by \(\|x-y\|_\infty=\max_i|x_i-y_i|\).
The maximum product metric \(d_{\max}\) takes the maximum of the
factor distances; the gluing metric \(d_\vee\) preserves each
summand's metric and sums distances to the common basepoint for
points in different summands. This gives the metric spaces
\[
T:=(\square^1\times\square^1,d_{\max}),
\qquad
M:=(\square^1\vee\square^2\vee\square^1,d_\vee)
\]
in \cref{fig:cubic-comparison-intro}(a)--(b).
We choose corners as basepoints of the cube boundaries, use the
corresponding product basepoint on \(T\), and take the common wedge point
as the basepoint of \(M\).
A short, direct computation shows that the spaces \(T\) and \(M\) have isomorphic persistent
homology modules in every degree and identical persistent \(\TC\) and
\(\zcup\), whereas persistent \(\cat\) and \(\cl\) distinguish them.

Moreover, the wedges \(T\vee T\) and \(T\vee M\), equipped with the
gluing metric and shown in \cref{fig:cubic-comparison-intro}(c)--(d),
have isomorphic persistent homology modules in every degree and identical
persistent \(\cat\) and \(\cl\), whereas persistent \(\TC\) and
\(\zcup\) distinguish them. Both
\(\TC\) and \(\zcup\) equal \(4\) for \(T\vee T\) and \(3\) for
\(T\vee M\) on intervals \([s,t]\) with \(0<s\leq t\leq2\), and
vanish on all other intervals; see \cref{ex:tc-zcl-persistent-main} for more details.

For \(T\) and \(M\), the Gromov--Hausdorff lower bound
\(1/2\) from persistent \(\cat\) and \(\cl\) improves on the bound from
persistent homology but is weaker than the diameter lower bound \(1\), since
\(\diam(T)=2\) and \(\diam(M)=4\).
By contrast, \(T\vee T\) and \(T\vee M\) both have diameter \(4\), so the
diameter lower bound vanishes.
Thus, for this pair, persistent \(\TC\) and \(\zcup\) improve not only on
persistent homology but also on the diameter lower bound.

\subsubsection*{Additional quantitative comparisons.}
We next consider a family in which all four persistent invariants
\(\TC\), \(\zcup\), \(\cat\), and \(\cl\) yield stronger
Gromov--Hausdorff lower bounds than persistent homology.
Motivated by the comparison in \cite{medina2025persistent},
we consider real projective space \(\rp^n\) and the wedge of
spheres \(\wedgeS:=\bigvee_{i=1}^n\bbS^i\), for \(n\geq2\).
We use two metric constructions, based on cube boundaries with
the \(\ell_\infty\) metric and on geodesic spheres.
For each construction, we consider versions with and without
truncation, giving four comparisons in total.
Throughout these comparisons, homology and cohomology are
taken with coefficients in \(\bbF_2\).

Suitable truncations and rescalings give pairs with isomorphic
persistent homology modules but positive erosion distances for
persistent \(\TC\), \(\zcup\), \(\cat\), and \(\cl\).
We treat these comparisons in
\cref{subsubsec:explicit-comparisons-main} and summarize the
distances in \cref{table:persistent-distinguishing-main}.

For the untruncated metrics, only partial information about the
Vietoris--Rips filtrations is available. Nevertheless, this
information suffices to show that, for suitable rescalings of
the projective-space metric, all four invariants give
Gromov--Hausdorff lower bounds stronger than both the bound
from persistent homology and the diameter bound;
see \cref{prop:untruncated-cubical-comparison,prop:scaled-geodesic-comparison}.

For an untruncated \(\ell_\infty\) model, let \(d_\square\)
be the antipodal quotient metric on
\(\rp^n\cong\square^n/(\pm1)\), induced by the
\(\ell_\infty\) metric on \(\square^n\). Set
\[
X=\left(\rp^n,\tfrac94 d_\square\right),
\qquad
Y=\left(\bigvee_{i=1}^n\square^i,d_\vee\right),
\]
where \(d_\vee\) is the gluing metric.
By \cref{prop:untruncated-cubical-comparison,thm:stability},
each of the four persistent invariants gives the
Gromov--Hausdorff lower bound \(3/8\), whereas positive-degree
persistent homology gives at most \(1/4\) and the diameter
bound equals \(1/4\).

We next consider the same underlying spaces with untruncated
metrics arising from geodesic spheres.
Equip \(\rp^n\) with the antipodal quotient metric induced by
the round geodesic sphere of radius \(2\), and \(\wedgeS\)
with the gluing metric on unit round geodesic
spheres~\cite{adamaszek2020homotopy}.
This choice of metrics allows us to compare our lower bounds
with those obtained from persistent cohomology operations.
For this pair, %Medina--Mardones and the second author
\cite[Theorem D]{medina2025persistent} showed that positive-degree
persistent homology yields a Gromov--Hausdorff lower bound of
at most \(\pi/8\), whereas persistent cohomology operations
based on Steenrod squares give the stronger bound \(\pi/6\).
The associated persistence modules were introduced in
\cite{medina2022per_st}.

We recover the bound \(\pi/6\) using each of the four persistent
invariants. However, the diameter lower bound
is \(\tfrac12|\pi-2\pi|=\pi/2\), which is larger.
More generally, \cref{prop:scaled-geodesic-comparison} treats
rescaling the projective-space metric by any factor
\(\lambda>0\) and shows that suitable choices make the bounds
from all four invariants stronger than both the bound from
persistent homology and the diameter lower bound.
Taking \(\lambda=1\) and applying \cref{thm:stability}
gives \cref{prop:erosion-comp}.

\begin{restatable}{corollary}{rpnvswedge}
\label{prop:erosion-comp}
Let \(n\geq2\). Equip \(\rp^n\) with the antipodal quotient
metric induced by the round geodesic sphere of radius \(2\),
and \(\wedgeS\) with the gluing metric on unit round geodesic
spheres. For \(\zcup\) and \(\cl\), use coefficients in
\(\bbF_2\). Then, for each
\(\invariant\in\{\TC,\zcup,\cat,\cl\}\),
\begin{equation}\label{eq:erosion-comp}
\tfrac{\pi}{3}
\leq d_{\mathrm{E}}\left(
\invariant(\VR_\bullet(\rp^n)),
\invariant(\VR_\bullet(\wedgeS))
\right)
\leq 2\cdot d_{\mathrm{GH}}\left(\rp^n,\wedgeS\right).
\end{equation}
\end{restatable}

\begin{figure}[!htbp]
\centering
\makebox[0.5\linewidth][c]{%
\begin{tikzpicture}[font=\small]
\begin{axis} [
scale only axis,width=3.15cm,height=3.15cm,
xlabel={$s$}, ylabel={$t$}, label style={font=\small},
title style={font=\small,align=center},
title = {$\TC(\VR_\bullet(\rp^n))$ or\\ $\zcup(\VR_\bullet(\rp^n))$},
ticklabel style = {font=\small},
axis y line=middle,
axis x line=middle,
ytick={0.67,0.95},
yticklabels={$2\pi/3$,$\pi$},
xtick={0.67,0.95},
xticklabels={$2\pi/3$,$\pi$},
xmin=0, xmax=1.1,
ymin=0, ymax=1.1,]
\addplot [mark=none,color=black] coordinates {(0,0) (1,1)};
\addplot [thick,dgm region=3]coordinates {
(0,.67)
(0,0)
(.67,.67)
(0,.67)};
\addplot [thick,color=black!10!white,fill=black!10!white,
        fill opacity=0.4]coordinates {
(0,0.95)
(0,0.67)
(0.67,0.67)
(0.95,0.95)
};
\node[mark=none] at (axis cs:.25,.45){\normalsize{$>2$}};
\end{axis}
\end{tikzpicture}%
}%
\makebox[0.5\linewidth][c]{%
\begin{tikzpicture}[font=\small]
\begin{axis} [
scale only axis,width=3.15cm,height=3.15cm,
xlabel={$s$}, ylabel={$t$}, label style={font=\small},
title style={font=\small,align=center},
title = {$\TC(\VR_\bullet(\wedgeS))$ or\\ $\zcup(\VR_\bullet(\wedgeS))$},
ticklabel style = {font=\small},
axis y line=middle,
axis x line=middle,
ytick={0.6,0.95},
yticklabels={$\zeta_n$,$\pi$},
xtick={0.6,0.95},
xticklabels={$\zeta_n$,$\pi$},
xmin=0, xmax=1.1,
ymin=0, ymax=1.1,]
\addplot [mark=none] coordinates {(0,0) (1,1)};
\addplot [thick,dgm region=2]coordinates {
(0,0.6)
(0,0)
(0.6,0.6)
(0,0.6)};
\addplot [thick,color=black!10!white,fill=black!10!white,
        fill opacity=0.4]coordinates {
(0,0.95)
(0,0.6)
(0.6,0.6)
(0.95,0.95)
};
\node[mark=none] at (axis cs:.25,.45){\normalsize{$2$}};
\end{axis}
\end{tikzpicture}%
}
\caption{
The persistent invariants $\invariant(\VR_\bullet(\rp^n))$ (left) and $\invariant(\VR_\bullet(\wedgeS))$ (right), where $\invariant = \TC$ or $\zcup$, $n \geq 2$, and $\zeta_n := \arccos\big(-\tfrac{1}{n+1}\big)$.
Here \(\rp^n\) carries the quotient metric from the round geodesic
sphere of radius \(2\), and \(\wedgeS\) the gluing metric on unit round
geodesic spheres; neither metric is truncated. Cohomology has coefficients
in \(\bbF_2\). Each point \((s,t)\), with \(0<s\leq t\), represents the interval \([s,t]\). The region labeled \(>2\) uses the shade for its integer lower bound \(3\). Gray regions indicate undetermined values.
See \cref{lem:TC-zcup-RPn-wedge-main} for details.
}
\label{fig:tc_zcl}
\end{figure}

To obtain these bounds, we first estimate persistent
\(\TC\), \(\zcup\), \(\cat\), and \(\cl\) from the known
parts of the Vietoris--Rips filtrations
\cite{adams2022metric,lim2024vietoris,adamaszek2020homotopy};
see \cref{lem:TC-zcup-RPn-wedge-main,prop:scaled-geodesic-comparison}.
\Cref{fig:tc_zcl} illustrates these estimates for persistent
\(\TC\) and \(\zcup\).
We then apply the erosion-distance estimate in \cref{prop:de-value-main}, which generalizes the argument of
\cite[Proposition~3.7]{memoli2024PersistentCup} for comparing
persistent LS-category and cup-length for a specific pair
of metric spaces.

\introheading{Organization.}

In \cref{sec:topological invariants}, we review LS-category,
sectional category, topological complexity, and zero-divisor-cup-length,
including their extensions to maps.
In \cref{sec:prelim}, we introduce the persistence framework, establish
erosion stability for homotopy-invariant categorical invariants, and
define persistent topological complexity and zero-divisor-cup-length.
In \cref{sec:examples-main}, we construct \(\ell_\infty\) and geodesic
metric spaces, analyze their Vietoris--Rips filtrations, and compute
persistent invariants and bounds on their distances. These comparisons
illustrate the distinguishing power of the invariants and yield
Gromov--Hausdorff lower bounds.

\introheading{Acknowledgement.}
\iftranreview\else
This work was partially supported by NSF DMS \#2301359, NSF CCF \#2310412 and NSF RI \#1901360.
L.Z. gratefully acknowledges the support of the AMS-Simons Travel Grants.
\fi
OpenAI's ChatGPT and Codex were used to assist with manuscript
editing and formatting, LaTeX source organization, literature searches, checking proofs and examples, and figure and table layout.
The authors take full responsibility for the manuscript and have
checked its mathematical results, proofs, and references.

\section{Topological Invariants}
\label{sec:topological invariants}

In this section, we review the classical topological invariants that form the foundation of our persistent constructions.
Our main focus is on \emph{topological complexity} and its cohomological lower bound, the \emph{zero-divisor-cup-length}.
Both invariants arise naturally from the broader framework of sectional category and are closely related to the Lyusternik--Schnirelmann (LS) category and cup-length.

Two aspects will be essential later.
First, each of these four mentioned invariants is defined not only for spaces but also for continuous maps.
This map-based perspective is crucial for defining their persistent analogues, since persistent invariants are obtained by evaluating classical invariants on the structure maps of a filtration.
Second, these invariants fit into a common categorical framework: their map-level versions are non-increasing under composition and
depend only on the homotopy class of the map.
This categorical viewpoint will allow us to lift them systematically to persistent invariants and establish their stability.

We begin with LS-category and sectional category, which provide the conceptual background for topological complexity.
We then review topological complexity and zero-divisor-cup-length, together with their fundamental properties and examples that will serve as test cases later in the paper.

We use nonnegative reduced conventions. For
\(\TC,\zcup,\cat,\cl\), we assign value zero to the empty
space and to every map with empty source.

\subsection{Lyusternik--Schnirelmann Category}

The notion of LS-category of a space was introduced by Lyusternik and Schnirelmann as a topological invariant to establish lower bounds on the number of critical points of smooth functions on a manifold \cite{lusternik1934methodes}. We begin by reviewing the definition and fundamental properties of LS-category following \cite{cornea2003lusternik}.

Beyond spaces, the LS-category can also be defined for continuous maps. This extension was first introduced by Fox \cite{fox1941lusternik} and further developed by Berstein and Ganea \cite{berstein1962category}. We recall the definitions of LS-category for both spaces and maps below.

\begin{definition}[%
{\cite[Definition 1.1]{cornea2003lusternik}}] \label{def:cat(X)}
    The \textbf{Lyusternik-Schnirelmann category (LS-category)} of $X$, denoted by $\cat(X)$, is the least nonnegative integer $n$ (or $+\infty$) such that $X$ has an open cover $U_1,\dots,U_{n+1}$ such that each inclusion $U_i\hookrightarrow  X$ is null-homotopic (i.e. $U_i$ is contractible to a point in $ X$).
\end{definition}

Let \( X \) and \( Y \) be topological spaces with basepoints \( x_0 \in X \) and \( y_0 \in Y \). The \emph{wedge sum} $X\vee_{x_0\sim y_0} Y$ (or simply $X\vee Y$) is the quotient of the disjoint union \( X \amalg Y \) obtained by identifying the basepoints:
 ${\displaystyle X\vee Y=(X\amalg Y)\;/{x_0 \sim y_0}}$.

\begin{proposition}[\cite{cornea2003lusternik}]
    \label{prop:properties of cat}
Let $X$ and $Y$ be two path-connected finite CW complexes with basepoints chosen among their vertices. Then:
    \begin{enumerate}[label=(\alph*)]
        \item \label{prop-item:cat product} $\cat(X \times Y) \leq \cat(X) + \cat(Y)$.
        \item \label{prop-item:cat wedge} $\cat(X\vee Y) = \max\{\cat(X), \cat(Y)\}$.
        \item \label{prop-item:cat-product-spheres} $\cat\left((\bbS^d)^{\times n}\right) = n$ for any $d\geq 1$ and $n\geq 1$.
    \end{enumerate}
\end{proposition}

\begin{proof}
These statements are well-known results in the theory of LS-category. See \cite[Theorem~1.37]{cornea2003lusternik},
\cite[Proposition~1.27(2)]{cornea2003lusternik} and
\cite[Proposition~1.5]{cornea2003lusternik} for proofs, respectively.
\end{proof}

\begin{definition}[{\cite[Definition 1.1]{berstein1962category}}]

\label{def:cat(f)}
The \textbf{LS-category of a map $f: X\to  Y$}, denoted by $\cat(f)$, is the least number $n$ (or $+\infty$) such that $ X$ can be covered by open sets $U_1,\dots, U_{n+1}$ such that each $f\vert_{U_i}$ is null-homotopic (i.e. $f\vert_{U_i}$ is homotopic to a constant map from $U_i$ to $Y$).
\end{definition}

Shortly after the introduction of the LS-category, Froloff and Elsholz \cite{FroloffElsholz1935} proposed an alternative method for estimating the number of critical values of a function on a manifold.
Their approach relied on the maximal number of cycles having a non-trivial intersection product, which nowadays is known to be equal to the cup-length.
Below, we review the cohomology ring, the concept of cup-length, and its connection to the LS-category.

From \cite[Section 3.2]{hatcher2002algebraic}, we recall that cohomology with coefficients in a ring is naturally endowed with a graded ring structure via the cup product operation.
The \emph{cup product} operation is defined at the cochain level using the Alexander–Whitney map, and induces a product on cohomology:
\[\smile: \rmH^*(X) \otimes \rmH^*(X) \to \rmH^*(X),\]
sending $\alpha \in \rmH^k(X)$ and $\beta \in \rmH^l(X)$ to $\alpha \smile \beta \in \rmH^{k+l}(X)$; see \cref{eq:cup-product-diagonal} for an interpretation via the diagonal map.
In this work, we consider cohomology over field coefficients, in which case the cohomology becomes a graded algebra over the coefficient field.
For simplicity, we will often omit the coefficients and write the cohomology of a topological space $X$ as $\rmH^*(X)$.

\begin{definition}
    The \textbf{nilpotency (or index)} of an ideal $I$, denoted by $\nil (I)$, of a ring $R$ is the smallest nonnegative integer $n$ (or $\infty$) such that $I^{n+1} = 0$.
    We use the same definition for a possibly nonunital subalgebra, and in particular $\nil(0)=0$.
\end{definition}

\begin{definition}[{\cite[Definition 1.4]{cornea2003lusternik}}]
    The \textbf{cup-length} of \( X \), denoted by \( \cl(X) \), is defined as the nilpotency of the ideal generated by positive-degree elements in \( \rmH^*(X) \).
\end{definition}

\begin{proposition}[{\cite[Proposition 1.5]{cornea2003lusternik}}]
\label{prop:cl<=cat}
    $\cl(X) \leq \cat(X)$.
\end{proposition}

\subsection{Sectional Category}

In \cite{schwarz1961genus}, Schwarz introduced a variant of LS-category, referred to as the \emph{genus} of a fibration.
Following the terminology of \cite{james1978category}, this notion is also known as the \emph{sectional category} of a fibration.
We provide a review of this concept following \cite[Section 9.3]{cornea2003lusternik}.

\begin{definition}
A map $p:E\to B$ has the \textbf{homotopy lifting property} with respect to a space $X$ if given a homotopy $g_\bullet:X\times [0,1]\to B$ and a map $\tilde{g}_0:X\to E$ lifting $g_0$,
there exists a homotopy $\tilde{g}_\bullet:X\times [0,1]\to E$ lifting $g_\bullet$ and satisfying $\tilde{g}_\bullet(x,0)=\tilde{g}_0(x)$ for every $x\in X$. The lifting condition is expressed by the commutative diagram:
\begin{center}
\begin{tikzcd}
&
E
\ar[d, "p" right]
\\
X\times [0,1]
\ar[r, "g_\bullet" below]
\ar[ur,"\tilde{g}_\bullet" above, dotted]
&
B,
\end{tikzcd}
\end{center}

A \textbf{fibration} is a map $p:E\to B$ satisfying the homotopy lifting property for all spaces $X$.
\end{definition}

\begin{definition}[Sectional category]
    Given a fibration $p:E\to B$, the \textbf{sectional category of $p$}, denoted by $\secat(p)$, is the least number $n$ (or $+\infty$) such that $B$ can be covered by $n+1$ open subsets $U_1,\dots, U_{n+1}$, each admitting a section $s_i:U_i\to E$ satisfying $p \circ s_i = \iota_{U_i}$, where $\iota_{U_i}:U_i\hookrightarrow B$ is the inclusion, i.e.,
    \begin{center}
    \begin{tikzcd}
    &
    E
    \ar[d, "p"]
    \\
    U_i
    \ar[r, hook]
    \ar[ur,"s_i", dotted]
    &
    B.
    \end{tikzcd}
    \end{center}
\end{definition}

For an arbitrary continuous map $p:E\to B$, we define $\secat(p)$ by allowing homotopy sections: on each open set $U_i$, a map $s_i:U_i\to E$ must satisfy $p\circ s_i\simeq\iota_{U_i}$, where $\iota_{U_i}:U_i\hookrightarrow B$ is the inclusion.
For a fibration, the homotopy lifting property turns a homotopy section into a strict section, so this agrees with the preceding definition.

\begin{remark}\label{rmk:cat as secat}
    Let $\pathspace_0(X)$ be the set of all paths in a based path-connected space \(X\) starting from the base point, and let $p_{X,0}:\pathspace_0(X)\to X$ be the map sending each based path $\gamma$ to $\gamma(1)$.
    Then
    \[\cat(X) = \secat(p_{X,0}:\pathspace_0(X)\to X).\]
    This is because an open subset $U\subset X$ admits a section $s:U \to \pathspace_0(X)$ if and only if the inclusion $U\hookrightarrow X$ is null-homotopic.
\end{remark}

\begin{proposition}[{\cite[Proposition 9.4]{cornea2003lusternik}}]
\label{prop:coho lower bound secat}
Let $p:E\to B$ be a fibration. Then:
\begin{itemize}
    \item If every fiber of $p$ is nonempty, then $\secat(p) \leq \cat(B)$. If, in addition, $E$ is contractible, then $\secat(p) = \cat(B)$.
    \item $\nil (\kernel (p^*)) \leq \secat (p)$, where $p^*$ is the induced morphism on cohomology.
\end{itemize}
\end{proposition}

In~\cite{murillo2021topological}, the authors introduce the notion of \(f\)-sectional category as a generalization of the classical sectional category.
Based on this, they define the topological complexity of a map.
Although their formulation differs from that in~\cite{scott2022ls}, the resulting invariant is shown to be numerically equivalent. See \cref{subset: tc zcup of maps} for further discussion.

\begin{definition}[{\cite[Definition 1.1]{murillo2021topological}}]
    Given maps $E\xrightarrow{p} B\xrightarrow{f} Z$, the \textbf{$f$-sectional category of $p$}, denoted by $\secat_f(p)$, is the least number $n$ (or $+\infty$) such that $B$ can be covered by $n+1$ open subsets $U_1,\dots, U_{n+1}$ each admitting a map $s_i:U_i\to E$ (called an \textbf{$f$-section}) such that $f\circ p \circ s_i \simeq f|_{U_i}$.
\end{definition}

\begin{proposition}[{\cite[Section 1]{murillo2021topological}}]
Let $p:E\to B$ and $f:B\to Z$ be two continuous maps between nonempty path-connected spaces.
    \begin{itemize}
        \item $\secat_{\id_B}(p) = \secat(p)$.
        \item If $p\simeq \tilde p$ and $f\simeq \tilde f$, then $\secat_f(p) = \secat_{\tilde f}(\tilde p)$.
        \item $\secat_f(p)\leq \min\{\secat(p), \cat(f)\}$.
        \item $\nil (\kernel (p^*|_{\im f^*})) \leq \secat_f(p)$.
    \end{itemize}
\end{proposition}

\subsection{Topological Complexity (TC) and Zero-Divisor-Cup-Length}
\label{sec:TC}

The motion planning problem seeks a function assigning a continuous path between any pair of points in a nonempty path-connected space $X$. A \emph{continuous} solution exists if and only if $X$ is contractible \cite[Theorem 1]{farber2003}.
To quantify the discontinuity of such motion planners, Farber introduced the \emph{topological complexity} of topological spaces.
We review the topological complexity of spaces \cite{farber2003}
and maps \cite{murillo2021topological,scott2022ls,scott2022topological}, along with cohomological lower bounds for topological complexity.

\subsubsection{Topological Complexity of Topological Spaces}

Recall that \( \pathspace(X) \) is the space of all continuous paths \(\gamma: [0,1] \to X\), and \( p_X: \pathspace(X) \to X \times X \) is the map assigning to each path its pair of endpoints \( (\gamma(0), \gamma(1)) \).

\begin{definition}[Topological complexity,
{\cite{farber2003}}%
]\label{def:TC}
Let $X$ be a topological space.
\begin{itemize}
    \item A \textbf{continuous motion planner} on a subset $U\subset X\times X$ is a continuous map $s:U\to \pathspace(X)$ such that $p_X\circ s = \iota_U$, where $\iota_U:U\hookrightarrow X\times X$ is the inclusion.
    \item The \textbf{topological complexity}
    $\TC(X)$ of $X$ is the least number $n$ (or $+\infty$) such that $X\times X$ can be covered by open subsets $U_1,\dots, U_{n+1}$, each admitting a continuous motion planner.\footnote{The motion planners need not be compatible with each other.}
    In other words, $s_i:U_i\to \pathspace(X)$ is a continuous map satisfying $p_X\circ s_i=\iota_{U_i}$, i.e.,
    \[
    \begin{tikzcd}[column sep=small]
        &
        \pathspace(X) \ar[d, "p_X" right] \ar[r, phantom, "\ni"] &
        \gamma \ar[d, mapsto]
        \\
        U_i \ar[r, hook] \ar[ur, "s_i" above, dotted] &
        X \times X \ar[r, phantom, "\ni"] &
        (\gamma(0), \gamma(1))
    \end{tikzcd}
    \]
\end{itemize}
\end{definition}

Here are some properties of $\TC(X)$.

\begin{proposition}[\cite{farber2003,farber2004instabilities,farber2006topology}]
\label{prop:tc of a space}
Let $X$ and $Y$ be path-connected spaces. Then:
\begin{enumerate}[label=(\alph*)]
    \item If $X$ and $Y$ are homotopy equivalent, then $\TC(X)=\TC(Y)$.
    \item If $X$ is paracompact\footnote{A topological space is paracompact if every open cover of it has a locally finite open refinement.}, then $\cat(X)\leq \TC(X)\leq \cat(X\times X)$.
    \item If $X$ and $Y$ are metrizable, then $\TC(X\times Y) \leq \TC(X) + \TC(Y).$
    \item \label{prop-item:wedge}
    $\max\{\TC(X), \TC(Y)\} \leq \TC(X\vee Y).$
    \item \label{prop-item:product of spheres} Let $X=(\bbS^d)^{\times n}$ be the product of $n$ copies of $d$-spheres, for $d \geq 1$ and $n \geq 1$. Then $\TC(X)$ equals $n$ if $d$ is odd, and $2n$ if $d$ is even.
    \item \label{prop-item:bouquets} Let $X=\bigvee^n\bbS^d$ be the bouquet of $n$ copies of $d$-spheres, for $d \geq 1$ and $n \geq 1$. Then $\TC(X)$ equals $1$ if $n=1$ and $d$ is odd, and $2$ if either $n>1$ or $d$ is even.
    \item $\TC(X) = \secat(p_X: \pathspace(X) \to X \times X)$.
\end{enumerate}
\end{proposition}

\begin{remark}\label{rmk:wedge}
The lower bound in \cref{prop:tc of a space}\,\ref{prop-item:wedge} was improved
in \cite[Theorem~3.6]{dranishnikov2014topological} to
\(\max\{\TC(X),\,\TC(Y),\,\cat(X\times Y)\}\) for \emph{absolute neighborhood
retracts} (ANRs), i.e., metrizable spaces \(X\) such that whenever \(X\) is embedded
as a closed subset of a metrizable space \(Z\), there is an open
neighborhood of \(X\) in \(Z\) that retracts onto \(X\);
see \cite{borsuk1967retracts}.
In particular, all finite CW complexes are ANRs.
\end{remark}

We include detailed discussions of the following examples, which will serve as key test cases in  \cref{sec:examples-main} for comparing the distinguishing power of persistent topological complexity with that of persistent homology and other persistent invariants.

\begin{example}[Topological complexity of real projective spaces]
\label{ex:tc of rpn}
    For \( n \geq 1 \), let \( \rp^n \) denote the real projective \( n \)-space.
    Throughout this example, the cohomology used to define \(\zcup\) is
taken with coefficients in \(\bbF_2\).
    The topological complexity of \( \rp^n \) has been studied extensively; see \cite{FTY2003}.
    In particular, it follows from \cite[Theorem 4.5]{FTY2003} and its proof that, for any positive integer \( r \), if \( n \geq 2^{r-1} \), then
    \[
    \TC(\rp^n) \geq \zcup(\rp^n) \geq 2^r - 1.
    \]
    Thus, for $n = 2^{r-1},\, 2^{r-1}+1,\, \dots,\, 2^r - 2$,
    \[
    \TC(\rp^n) \geq \zcup(\rp^n) > n = \cat(\rp^n).
    \]
    Moreover, it is known that \( \TC(\rp^n) = n \) for \( n = 1, 3, 7 \) and \( \TC(\rp^{2}) = 3 \).
    Explicit values of \( \TC(\rp^n) \) for all \( n \leq 23 \) are provided in \cite{FTY2003}.
\end{example}

\begin{example}[Topological complexity of bouquets of spheres]
\label{ex:tc of bouquets}
    \cref{prop:tc of a space} \ref{prop-item:bouquets} can be extended to wedges of spheres of varying dimensions. Let \(n\geq 1\), let \(d_1,\dots,d_n\geq 1\), and set \( X := \bbS^{d_1} \vee \dots \vee \bbS^{d_n} \). We claim that
    \begin{equation}\label{eq:tc of bounquets}
        \TC(\bbS^{d_1}\vee \dots \vee \bbS^{d_n}) =
            \begin{cases}
                1, \quad &\mbox{if $n=1$ and $d_1$ is odd;}\\
                2, \quad &\mbox{otherwise}.
            \end{cases}
    \end{equation}

    This is precisely \cite[Lemma~4.1]{ipanaque2019category}, after translating from the unreduced conventions used there to our reduced conventions.
    We provide a brief justification to demonstrate how the inequalities in \cref{prop:tc of a space,prop:properties of cat} yield the result.

The case \(n=1\) is the classical computation of
\(\TC(\bbS^{d_1})\) recalled in \cref{prop:tc of a space} \ref{prop-item:product of spheres}.
Suppose now that \(n\geq2\), and write
\[
Y:=\bbS^{d_2}\vee\cdots\vee\bbS^{d_n},
\qquad
X=\bbS^{d_1}\vee Y.
\]
By \cref{prop:properties of cat} \ref{prop-item:cat wedge},
\(\cat(X)=\cat(Y)=1\). Moreover,
\(
\cat(\bbS^{d_1}\times Y)=2.
\)
Indeed, the product inequality \cref{prop:properties of cat} \ref{prop-item:cat product} gives the upper bound, while nonzero
positive-degree cohomology classes from \(\bbS^{d_1}\) and \(Y\) have a
nonzero external product, giving the matching cup-length lower bound.
Since \(\bbS^{d_1}\) and \(Y\) are finite CW complexes and hence ANRs,
\cref{rmk:wedge} applied to \(X=\bbS^{d_1}\vee Y\) yields
\[
2=\cat(\bbS^{d_1}\times Y)
\leq \TC(X)
\leq \cat(X\times X)
\leq 2\cat(X)=2.
\]
Therefore, \(\TC(X)=2\) whenever \(n\geq2\), completing the computation.

\end{example}

\subsubsection{Zero-divisor-cup-length of Topological Spaces}

As illustrated in \cref{ex:tc of bouquets}, computation of topological complexity often relies on auxiliary invariants such as the LS-category and cup-length. Motivated by this, Farber defines the \emph{zero-divisor-cup-length}, a cohomological invariant derived from the lower bound for sectional category (see \cref{prop:coho lower bound secat}), to provide more accessible lower bounds for topological complexity \cite{farber2003}.

The cup product operation is closely related to the map induced on cohomology by the diagonal map $\Delta_X : X\to X\times X$, defined as $\Delta_X(x) = (x,x)$.
Specifically, let $\pi_1,\pi_2:X\times X\to X$ be the projections onto the first and second factors, respectively.
Given a map $f:X\to Y$, let $f^*:\rmH^*(Y) \to \rmH^*(X)$ denote the induced map in cohomology.
The cohomology classes $\alpha\in \rmH^k(X)$ and $\beta\in \rmH^l(X)$ can be pulled back to $X\times X$ as $\pi_1^*(\alpha)$ and $ \pi_2^*(\beta)$, respectively.
The \emph{cross product} (or \emph{external cup product}) is defined as
\[
\rmH^*(X) \otimes \rmH^*(X) \xrightarrow{\times} \rmH^*(X\times X) \text{ with }
\alpha \times \beta := \pi_1^*(\alpha) \smile \pi_2^*(\beta).
\]

The \emph{internal cup product} operation $\smile:\rmH^*(X) \otimes \rmH^*(X) \xrightarrow{\times} \rmH^*(X\times X) \xrightarrow{\Delta_X^*} \rmH^*(X)$ can also be defined as the composition of $\Delta_X^*$ and the cross product: \label{para:cup product}
\begin{equation}\label{eq:cup-product-diagonal}
\alpha \smile \beta
:= \Delta_X^*(\alpha\times \beta)
= \Delta_X^*(\pi_1^*(\alpha) \smile \pi_2^*(\beta)).
\end{equation}
We will use the notation $\smile_X$ when we want to specify its dependence on the space $X$.

    Under the usual Künneth hypotheses, the cross product induces a ring
isomorphism
\[
\rmH^*(X)\otimes\rmH^*(X)\cong\rmH^*(X\times X).
\]
For example, this holds when \(X\) is a CW complex and
\(\rmH^k(X;\field)\) is finite-dimensional for every \(k\);
see \cite[Theorem~3.15]{hatcher2002algebraic}.
For homogeneous classes \(a,b,c,d\in\rmH^*(X)\), the
tensor-product multiplication is
\[
(a\otimes b)(c\otimes d)
=(-1)^{|b||c|}(a\smile c)\otimes(b\smile d),
\]
where \(|b|\) and \(|c|\) denote the cohomological degrees of
\(b\) and \(c\).
Under this identification, Farber's formulations in terms of
\(\ker(\smile_X)\) in \cite{farber2003} and
\(\ker(\Delta_X^*)\) in \cite{farber2008invitation} agree.
We adopt the latter formulation.

\begin{definition}[%
{\cite[Definition 6]{farber2003}}]
    Let $X$ be any topological space.
    The kernel of $\Delta_X^*: \rmH^*(X \times X) \to \rmH^*(X)$ is called the \textbf{ideal of the zero-divisors} of $\rmH^*(X)$ and its elements are called \textbf{zero-divisors}.

The \textbf{zero-divisor-cup-length}, denoted by $\zcup(X)$, is the nilpotency of $\kernel (\Delta_X^*) $, i.e., the length of the longest non-zero product in $\kernel(\Delta_X^*)$.
\end{definition}

\begin{proposition}[{\cite[Theorem 7]{farber2003}}]
\label{prop:zcup<=tc}
    $\zcup(X) = \nil \big(\kernel (p_X^*: \rmH^*(X\times X) \to \rmH^*(\pathspace(X)))\big)$.
    Thus, $\zcup(X) \leq \TC(X)$.
\end{proposition}

\begin{proof}
    Let $c:X\to \pathspace(X)$ be the map assigning to each $x\in X$ the constant path at $x$.
    Then $c$ is a homotopy equivalence, implying that $c^*$ defines an isomorphism $\rmH^*(\pathspace(X))\cong \rmH^*(X)$.
    Note that $\Delta_X = p_X\circ c$.
    Thus,
    \[\zcup(X) = \nil (\kernel (\Delta_X^*)) = \nil (\kernel (p_X^*)) \leq \secat(p_X) = \TC(X).\qedhere\]
\end{proof}

The following example illustrates the use of \cref{prop:zcup<=tc} to derive lower bounds on topological complexity.

\begin{example}[\cite{farber2003}]
\label{ex:zcup of Sn}

    Let \(\bbS^n\) be the \(n\)-dimensional sphere for \(n \geq 1\), and choose a coefficient field of characteristic different from 2. Let \(u \in \rmH^n(\bbS^n)\) denote the fundamental class and \(1 \in \rmH^0(\bbS^n)\) the unit. Consider the elements
    \[
    \alpha := 1 \otimes u - u \otimes 1, \qquad \beta := u \otimes u
    \]
    in \(\rmH^*(\bbS^n \times \bbS^n) \cong \rmH^*(\bbS^n) \otimes \rmH^*(\bbS^n)\). Since \(\Delta^*(x \otimes y) = x \smile y\), we have
    \(
    \Delta^*(\alpha) = u - u = 0\) and \(\Delta^*(\beta) = u \smile u = 0.
    \)
    Thus, both \(\alpha\) and \(\beta\) lie in \(\ker (\Delta^*)\).
    Furthermore,
    \[
    \alpha \smile \alpha = (1 \otimes u - u \otimes 1)^2 =
    \begin{cases}
    0   & \text{if } n \text{ is odd}\\
    -2\beta & \text{if } n \text{ is even}.
    \end{cases}
    \]
    By \cref{prop:zcup<=tc}, we have
    \[
    \TC(\bbS^n) \geq \zcup(\bbS^n) \geq
    \begin{cases}
    1 & \text{if } n \text{ is odd}, \\
    2 & \text{if } n \text{ is even}.
    \end{cases}
    \]
    For a full computation of \(\TC(\bbS^n)\) and verification that equalities in the above formula hold, see \cite{farber2003}.

If instead the coefficient field has characteristic \(2\), then
\(\alpha^2=0\) for every \(n\geq1\).
Since \(\ker(\Delta^*)\) is spanned by \(\alpha\) and \(\beta\),
and \(\alpha\beta=\beta^2=0\), every product of two
zero-divisors vanishes. As \(\alpha\neq0\), it follows that
\(\zcup(\bbS^n)=1\).
\end{example}

\subsubsection{Topological Complexity and Zero-divisor-cup-length of Maps}

\label{subset: tc zcup of maps}

Several extensions of topological complexity from spaces to maps have been proposed.
In \cite{pavevsic2017complexity}, Pave{\v{s}}i{\'c}, following a suggestion by Dranishnikov, proposed defining the topological complexity of a map $f:X\to Y$ between Euclidean neighborhood retracts (ENRs) to be the minimal number $n$ such that $X\times Y$ can be covered by ENR subsets $U_1,\dots,U_{n+1}$, where for each \( i = 1, \dots, n+1 \), there exists a continuous map \( s_i: U_i \to \pathspace(X) \) satisfying \( p_{X,Y} \circ s_i = \iota_{U_i} \), where \(\iota_{U_i}:U_i\hookrightarrow X\times Y\) is the inclusion.
Here, $p_{X,Y}: \pathspace(X) \to X \times Y$ is defined by $p_{X,Y}(\gamma) = (\gamma(0), f(\gamma(1)))$, sending each path \(\gamma\) to the pair consisting of its initial point in \(X\) and the image under \(f\) of its terminal point in \(Y\).
Pave{\v{s}}i{\'c} further developed this approach in \cite{pavesic2018topologist, pavesic2019topological}.
The question of defining a homotopy-invariant notion of topological
complexity for maps remained open.
This issue was recently resolved in works by Murillo and Wu
\cite{murillo2021topological} and Scott
\cite{scott2022ls, scott2022topological}, who developed
homotopy-invariant notions of topological complexity for maps.
In Murillo and Wu's formulation
\cite[Definition~2.1]{murillo2021topological},
\[
\TC(f)=\secat_{f\times f}(p_X),
\]
where \(p_X:\pathspace(X)\to X\times X\) is the endpoint
fibration.
By \cite[Theorem~3.4]{scott2022topological}, this agrees with
Scott's formulation in terms of \(f\)-motion planners,
which we recall in \cref{def:TC of f}.

\begin{definition}[Topological complexity of a map, {\cite[Definition 3.1.1]{scott2022ls}}
]\label{def:TC of f}
Let $f:X\to Y$ be a continuous map between topological spaces.
\begin{itemize}
    \item A \textbf{$f$-motion planner} on a subset $U\subset X\times X$ is a continuous map $s^f:U\to \pathspace(Y)$ such that $p_Y\circ s^f = (f\times f)|_U$.
    \item The \textbf{(pullback) topological complexity} $\TC(f)$ of $f$ is the least number $n$ (or $+\infty$) such that $X\times X$ can be covered by open subsets $U_1,\dots,U_{n+1}$, each admitting an $f$-motion planner. Equivalently, for each $i$, there is a continuous map
    \(
        s_i^f:U_i\to\pathspace(Y)
    \)
    satisfying
    \(
        p_Y\circ s_i^f=(f\times f)|_{U_i},
    \)
    i.e.,
    \[
    \begin{tikzcd}
        & &
        \pathspace(Y) \ar[d, "p_Y" right]
        \\
        U_i
        \ar[r, hook]
        \ar[urr, "s_i^f" above, dotted]
        &
        X\times X
        \ar[r, "f\times f"]
        &
        Y\times Y.
    \end{tikzcd}
    \]
\end{itemize}
\end{definition}

    \begin{remark}[Relative topological complexity]
Farber defined relative topological complexity for a subset
\(B\subseteq X\times X\) by restricting the map
\(p_X\colon\pathspace(X)\to X\times X\) to \(B\)
\cite[Definition~4.20]{farber2008invitation}. After translating to the
reduced convention, the topological complexity of an inclusion
\(\iota\colon A\hookrightarrow X\) agrees with Farber's invariant for
\(B=A\times A\): both endpoints lie in \(A\), while paths may traverse \(X\).

Short \cite{short2018relative} instead considered \(B=X\times A\), so that
only the terminal endpoint is required to lie in \(A\). Thus, Short's
invariant differs in general from \(\TC(\iota)\).
\end{remark}

In~\cite[Theorem 3.1.11]{scott2022ls}, the author introduces the following notion as a generalization of the classical zero-divisor-cup-length, which provides a cohomological lower bound for the topological complexity of a space, to the setting of continuous maps.

\begin{definition}[Zero-divisor-cup-length of a map, {\cite[Theorem 3.1.11]{scott2022ls}}]
\label{def:zcl of f}
Let $f:X\to Y$ be a continuous map between arbitrary topological spaces, and
$(f\times f)^*:\rmH^*(Y\times Y)\to \rmH^*(X\times X)$ the ring homomorphism
induced by $f\times f$.
The \textbf{zero-divisor-cup-length} of $f$ is defined as
\begin{align}
    \zcup(f)
    :=& \nil\left(\ker(\Delta_X^*) \cap \im((f\times f)^*) \right) \label{eq:zcl-def1} \\
    =& \nil\Big( \im\big((f\times f)^*: \ker(\Delta_Y^*) \to \ker(\Delta_X^*)\big) \Big).
        \label{eq:zcl-def2}
\end{align}
\end{definition}

The equivalence of the expressions \cref{eq:zcl-def1,eq:zcl-def2} is explained in \cref{rmk:zcl-map-formulas}.

\begin{remark}\label{rmk:zcl-map-formulas}
To illustrate the definition of $\zcup(f)$, we include the following diagram, which shows how zero-divisors pull back along the map \( f \):
\[
\begin{tikzcd}
    \ker (\Delta_Y^*)
    \ar[r,hook]
    \ar[d, "(f\times f)^*" left]
    & \rmH^*(Y\times Y)
    \ar[r, "\Delta_Y^*"]
    \ar[d, "(f\times f)^*" left]
    & \rmH^*(Y)
    \ar[d, "f^*" right] \\
    \ker (\Delta_X^*)
    \ar[r,hook]
    & \rmH^*(X\times X)
    \ar[r, "\Delta_X^*"]
    & \rmH^*(X).
\end{tikzcd}
\]
The right-hand square commutes because \(\Delta_Y \circ f = (f \times f) \circ \Delta_X\), which implies
\[
f^* \circ \Delta_Y^* = \Delta_X^* \circ (f \times f)^*,
\quad \text{so} \quad (f \times f)^*(\ker (\Delta_Y^*)) \subseteq \ker (\Delta_X^*).
\]
Thus, \((f \times f)^*\) sends elements of \(\ker(\Delta_Y^*)\) to elements of \(\ker(\Delta_X^*)\).
Conversely, let $a\in\ker(\Delta_X^*)\cap\im((f\times f)^*)$ and choose $b\in\rmH^*(Y\times Y)$ such that $a=(f\times f)^*b$, and let $\pi_{1,X}:X\times X\to X$ and $\pi_{1,Y}:Y\times Y\to Y$ be the first projections.
Define $b':=b-\pi_{1,Y}^*\Delta_Y^*b$.
Since $\Delta_Y^*\pi_{1,Y}^*=\id$, we have $\Delta_Y^*b'=0$, while naturality gives
\[
(f\times f)^*b'
=a-\pi_{1,X}^*f^*\Delta_Y^*b
=a-\pi_{1,X}^*\Delta_X^*a=a.
\]
Thus $\ker(\Delta_X^*)\cap\im((f\times f)^*)=(f\times f)^*(\ker(\Delta_Y^*))$, proving the equality of \cref{eq:zcl-def1,eq:zcl-def2}.
\end{remark}

Below, we recall some properties of $\TC(f)$, where the
\emph{cup-length of $f$}, denoted by $\cl(f)$, is the largest
nonnegative integer \(n\) (or \(+\infty\)) for which there exist
positive-degree homogeneous elements
\(
u_1,\dots,u_n\in
\im\bigl(f^*:\rmH^*(Y)\to\rmH^*(X)\bigr)
\)
such that \(u_1\smile\cdots\smile u_n\neq0\).

\begin{proposition}\label{prop:TC of a map}
Let $f, f_1, f_2:X \to Y$ and $g:Y \to Z$ be continuous maps between nonempty spaces.
Then:
\begin{enumerate}[label=(\alph*)]
    \item \label{prop:indentity} $\TC(\id_X)=\TC(X)$.
    \item \label{prop:map v.s. space} $\TC(f)\leq \min\{\TC(X),\TC(Y)\}.$
    \item \label{prop:map homotopy invariance}
    If \(f_1\) and \(f_2\) are homotopic, then
    \(\TC(f_1)=\TC(f_2)\).
    \item \label{prop:map composition}
    \(\TC(g\circ f)\leq \min\{\TC(g),\TC(f)\}\).

    \item \label{prop:map composition right h.i.}
    If \(f\) has a right homotopy inverse, then
    \(\TC(g\circ f)=\TC(g)\).

    \item \label{prop:map composition left h.i.}
    If \(g\) has a left homotopy inverse, then
    \(\TC(g\circ f)=\TC(f)\).

     \item \label{prop:cup cat TC} $\cl(f)\leq \cat(f)\leq \TC(f)$. If $f(X)$ lies in one path component of $Y$, then also $\TC(f)\leq\cat(f\times f)$.
     \item \label{prop:cohomology estimate} Let $X$ be an ANR. Then $\zcup(f) \leq \TC(f)$.
\end{enumerate}
\end{proposition}

\begin{proof}
The assertions in \cref{prop:TC of a map}\,\ref{prop:indentity}--\ref{prop:map composition left h.i.} follow from \cite[Example~3.3(3) and Propositions~3.8(1), 3.12--3.13, 3.15--3.16]{scott2022topological}.
In \cref{prop:TC of a map}\,\ref{prop:cup cat TC}, the inequality $\cl(f)\leq\cat(f)$ follows from \cite[Section~3.2]{hatcher2002algebraic}; the remaining inequalities are \cite[Proposition~3.8(2)]{scott2022topological}.
\Cref{prop:TC of a map}\,\ref{prop:cohomology estimate} follows from the proof of \cite[Theorem~3.11]{scott2022topological}, using \cref{eq:zcl-def2}.
Scott states these results for path-connected spaces; the cited proofs also apply to disconnected spaces, allowing $\TC(f)=\infty$ when $f(X)$ meets more than one path component of $Y$. For the upper bound $\TC(f)\leq\cat(f\times f)$ in \cref{prop:TC of a map}\,\ref{prop:cup cat TC}, we require $f(X)$ to lie in one path component of $Y$.
\end{proof}

\section{Persistent Topological Complexity and Other Invariants}
\label{sec:prelim}

In this section, we review the necessary background on persistence theory.
This includes the categorical framework for defining and comparing persistent invariants, as well as the relevant notions of interleaving distance and stability.
Our presentation follows \cite{bubenik2015metrics, blumberg2023universality, memoli2024PersistentCup, cornea2003lusternik}.

\subsection{Persistence}

Let \(\caC\) be a category.
A \emph{persistent object} in \(\caC\) is defined as a functor \(F_\bullet : (\mathbb{R}, \leq) \to \caC\), where \((\mathbb{R}, \leq)\) is the poset category of real numbers ordered by the usual total order. In other words, a persistent object assigns to each real number \(t\) an object \(F_t \in \caC\), and to each pair \(t \leq s\), a morphism \(f_t^s: F_t \to F_s\), such that
\[
f_t^t = \id_{F_t} \quad \text{and} \quad f_s^r \circ f_t^s = f_t^r \quad \text{for all } t \leq s \leq r.
\]

In this work, we focus on the following choices for the category $\caC$. Let $\field$ be a field. We denote by $\Vect$ the category of vector spaces over $\field$, in which persistent objects are referred to as \emph{persistence vector spaces} or \emph{persistence modules}.
When $\caC = \topo$, the category of topological spaces, persistent objects are called \emph{persistent (topological) spaces}.
We also consider certain full subcategories of $\topo$, such as the category of path-connected topological spaces, in which persistent objects are called \emph{persistent path-connected spaces}.

\subsubsection*{Interleaving Distance.}

Let \(F_\bullet\) and \(G_\bullet\) be two persistent objects in \(\caC\). A \emph{natural transformation} \(\varphi: F_\bullet \Rightarrow G_\bullet\) (also called a homomorphism) consists of morphisms \(\{\varphi_t: F_t \to G_t\}_{t \in \mathbb{R}}\) satisfying the following commutative diagrams for all $t \leq s$:
\[
\begin{tikzcd}
F_t \ar[r, "f_t^s"] \ar[d, "\varphi_t"'] & F_s \ar[d, "\varphi_s"] \\
G_t \ar[r, "g_t^s"] & G_s
\end{tikzcd}
\]
If each \(\varphi_t\) is an isomorphism, we say the two persistent objects are \emph{naturally isomorphic}, denoted \(F_\bullet \cong G_\bullet\).

For any $\delta\geq 0$, a \emph{\(\delta\)-interleaving} between two persistent objects \(F_\bullet\) and \(G_\bullet\) consists of two families of morphisms:
\[
\{\varphi_t: F_t \to G_{t+\delta}\}_{t\in \mathbb{R}}, \quad \{\psi_t: G_t \to F_{t+\delta}\}_{t\in \mathbb{R}},
\]
such that the following diagrams commute for all $t\leq s$:
	\begin{center}
		\begin{tikzcd}[column sep={6em,between origins}]
			F_t \ar[dr, "\varphi_t" below left]
			\ar[r,"f_t^s"]
			&
			F_s
			\ar[dr,"\varphi_s " ]&
			\\
			& G_{t+\delta}
			\ar[r,"g_{t+\delta}^{s+\delta}" below]
			&
			 G_{s+\delta}
		\end{tikzcd}
\hspace{0.6cm}
	\begin{tikzcd}[column sep={6em,between origins}]
			&  F_{t+\delta}
			\ar[r,"f_{t+\delta}^{s+\delta}"]
			&
			 F_{s+\delta}
			\\
			G_t
			\ar[ur, "\psi_t"]
			\ar[r,"g_t^s" below]
			&
			 G_{s}
			\ar[ur,"\psi_s" below right]&
		\end{tikzcd}
	\end{center}
and
\begin{center}
		\begin{tikzcd}
			F_t
			\ar[dr, "\varphi_t" below left ]
			\ar[rr,"f_{t}^{t+2\delta}"]
			&&
			 F_{t+2\delta}
			\\
			&  G_{t+\delta}
			\ar[ur,"\psi_{t+\delta}" below right ]
			&
		\end{tikzcd}
\hspace{0.6cm}
		\begin{tikzcd}
			&  F_{t+\delta}
			\ar[dr,"\varphi_{t+\delta}"]
			&
			\\
			G_t
			\ar[ur, "\psi_t"]
			\ar[rr,"g_{t}^{t+2\delta}" below]
			&&
			 G_{t+2\delta}.
		\end{tikzcd}
	\end{center}

\begin{definition} \label{def:interleaving}
    The \textbf{interleaving distance} between \(F_\bullet\) and \(G_\bullet\) is defined as:
    \[
    \di(F_\bullet, G_\bullet) := \inf\{ \delta \geq 0 \mid F_\bullet \text{ and } G_\bullet \text{ are } \delta\text{-interleaved} \},
    \]
    with the convention \(\inf \emptyset = \infty\).
    When needed, we write \(\di^{\caC}\) to emphasize the underlying category.
\end{definition}

The interleaving distance is non-increasing under post-composition with functors, a property that will be used in the stability results we will discuss later:

\begin{theorem}[{\cite[Section 2.3]{bubenik2015metrics}}]
\label{thm:Bubenik}
For functors $ F_\bullet, G_\bullet:(\bbR,\leq)\to\caC$ and $\rmH:\caC\to \caD$,
\[ \di^{\caD}(\rmH \circ F_\bullet,\rmH \circ G_\bullet)\leq \di^{\caC}( F_\bullet,G_\bullet).\]
\end{theorem}

\subsubsection*{Barcodes and the bottleneck distance.}

A persistence module \(V_\bullet\) over a field is \textbf{q-tame}
if every structure map \(V_s\to V_t\) with \(s<t\) has finite rank.
Its \textbf{persistence diagram} records birth--death pairs \((b,d)\),
with \(b<d\) and multiplicities; depicting each pair as a \textbf{bar} gives its
\textbf{barcode} \(\caB(V_\bullet)\) \cite{CSGO16}.
Endpoint inclusion does not affect the distances below.

A partial matching pairs bars from two barcodes, using each at most
once. Matching \((b,d)\) to \((b',d')\) costs
\(\max\{|b-b'|,|d-d'|\}\); leaving a bar unmatched costs half its
length. The \textbf{bottleneck distance} \(\db\) is the infimum,
over partial matchings, of the supremum of these costs
\cite{CSGO16}, with supremum zero for an empty collection.
Equal infinite endpoints contribute zero; unequal endpoints involving
infinity and unmatched unbounded bars have infinite cost.

For q-tame modules, the \textbf{isometry theorem} \cite{CSGO16} gives
\begin{equation}
\label{eq:barcode-isometry}
\di(V_\bullet,W_\bullet)
=\db\bigl(\caB(V_\bullet),\caB(W_\bullet)\bigr).
\end{equation}
Thus, any partial matching bounds the interleaving distance from above.
Vietoris--Rips homology modules of compact metric spaces are q-tame
\cite[Proposition~5.1]{CSO14}.

\subsubsection*{Gromov–Hausdorff distance.}

Let \(X, Y\) be subsets of a metric space \(Z\). The \emph{Hausdorff distance} between \(X\) and \(Y\) in \(Z\) is given by
\[
d_\mathrm{H}^Z(X, Y) := \inf\left\{ r > 0 \mid X \subseteq B_r(Y;Z) \text{ and } Y \subseteq B_r(X;Z) \right\},
\]
where \(B_r(A;Z)\) denotes the open \(r\)-neighborhood of \(A\) in \(Z\).

\begin{definition}
The \textbf{Gromov–Hausdorff distance} between compact metric spaces \((X, d_X)\) and \((Y, d_Y)\) is defined by
\[
\dgh(X, Y) := \inf_{Z, \psi_X, \psi_Y} d_\mathrm{H}^Z(\psi_X(X), \psi_Y(Y)),
\]
where the infimum is taken over all metric spaces \(Z\) and isometric embeddings \(\psi_X: X \to Z\) and \(\psi_Y: Y \to Z\).
\end{definition}

The Gromov--Hausdorff distance has several equivalent formulations. Below, we recall a map-based one. For further background on the Gromov-Hausdorff distance, see \cite{edwards1975structure,gromov2007metric}.

For metric spaces $(X,d_X)$ and $(Y,d_Y)$, the \emph{distortion} of a map
$\varphi:X\to Y$ is defined by
\[
\dis(\varphi):=\sup_{x,x'\in X}
\bigl|d_X(x,x')-d_Y(\varphi(x),\varphi(x'))\bigr|.
\]
For maps $\varphi:X\to Y$ and $\psi:Y\to X$, their \emph{co-distortion} is defined by
\[
\codis(\varphi,\psi)
:=\sup_{x\in X,\,y\in Y}
\bigl|d_X(x,\psi(y))-d_Y(\varphi(x),y)\bigr|.
\]
The Gromov--Hausdorff distance satisfies (see \cite[Theorem~2.1]{kalton1999distances}):
\begin{equation}\label{eq:dgh-maps}
\dgh(X,Y)
=
\inf_{\substack{\varphi:X\to Y\\ \psi:Y\to X}}
\tfrac12
\max\bigl\{
\dis(\varphi),
\dis(\psi),
\codis(\varphi,\psi)
\bigr\}.
\end{equation}

For nonempty compact metric spaces \(X,Y\), the classical
diameter bounds give (see \cite[Chapter~7]{burago2001course})
\begin{equation}\label{eq:dgh-diameter-bounds}
\tfrac12\bigl|\diam(X)-\diam(Y)\bigr|
\leq \dgh(X,Y)
\leq \tfrac12\max\{\diam(X),\diam(Y)\}.
\end{equation}
The lower bound follows from
\(\lvert\diam(X)-\diam(Y)\rvert\leq2d_{\mathrm H}^Z(X,Y)\)
for any common isometric realization in \(Z\); the upper bound follows
by taking constant maps in \cref{eq:dgh-maps}.

\subsubsection*{Vietoris–Rips filtration.}

Let \((X, d_X)\) be a metric space and let \(t > 0\). The \emph{Vietoris–Rips complex} \(\VR_t(X)\) is the simplicial complex with vertex set \(X\), in which a finite subset \(\sigma \subset X\) spans a simplex if and only if \(\diam(\sigma) < t\). When there is no risk of ambiguity, we will use the same notation to refer to both a simplicial complex and its geometric realization.

For $t\leq0$, set $\VR_t(X)=\varnothing$, with the unique structure maps from the empty complex.
The family \(\{\VR_t(X)\}_{t\in\bbR}\), equipped with the natural simplicial inclusions \(\VR_t(X) \hookrightarrow \VR_s(X)\) for \(t \leq s\), defines a persistent space
denoted by \(\VR_\bullet(X)\). We refer to this as the \emph{Vietoris–Rips filtration} of \(X\). Throughout this paper, we adopt the open condition \(\diam(\sigma) < t\) for defining simplices. The results presented through \cref{sec:persistent invariants} remain valid under the closed condition \(\diam(\sigma) \leq t\) as well. In \cref{sec:examples-main}, all explicit computations use the open condition; some results may extend to the closed case, though we do not pursue this direction here.

Applying the \(k\)-th homology functor with field coefficients to the filtration \(\VR_\bullet(X)\) yields a persistence module, denoted \(\rmH_k(\VR_\bullet(X))\) and called the \emph{\(k\)-th persistent homology} of \(X\).

A fundamental property of persistent homology is its stability with respect
to perturbations of the underlying metric space.
In the setting of Vietoris--Rips filtrations, this stability can be stated
as follows.

\begin{theorem}[{\cite{chazal2009gromov,CSO14}}]
\label{thm:vr-ph-stability}
Let $X$ and $Y$ be compact metric spaces and let $k\ge0$.
Then,
\begin{equation}\label{eq:gh_inequality_intro_homology}
\di\bigl(
\rmH_k(\VR_\bullet(X)),
\rmH_k(\VR_\bullet(Y))
\bigr)
\le
2\,\dgh(X,Y).
\end{equation}
\end{theorem}

\subsection{Persistent Invariants and Stability}
\label{sec:persistent invariants}

Persistent invariants, broadly interpreted, capture the evolution of topological invariants across scale. The classical \emph{rank invariant} of persistent (co)homology vector spaces serves as a prominent example \cite{carlsson2007theory}. There has been ongoing work on developing general frameworks for studying more flexible types of persistent invariants, as seen in works such as \cite{puuska2017erosion, bergomi2019rank, giunti2021amplitudes, memoli2024PersistentCup}.
We review the formalism from \cite{memoli2024PersistentCup} for defining persistent invariants and proving their stability.

In particular, \cite[Theorem~2]{memoli2024PersistentCup} establishes
stability for categorical invariants that are invariant under pre- and
post-composition with weak homotopy equivalences. The result is formulated
in terms of the homotopy--interleaving distance \(\dhi\) introduced in
\cite{blumberg2023universality}. The argument used in
\cite{memoli2024PersistentCup} for persistent LS-category and cup-length
also applies to the persistent topological complexity and
zero-divisor-cup-length defined in \cref{sec:persistent_TC}, provided that
the persistent spaces take values in CW complexes. Indeed, by Whitehead's
theorem, weak homotopy equivalences between CW complexes are homotopy
equivalences, and the corresponding map-level invariants are unchanged
under pre- and post-composition with homotopy equivalences. Thus,
\cite[Theorem~2]{memoli2024PersistentCup} yields both \(\dhi\)-stability
and, for Vietoris--Rips filtrations of compact metric spaces, Gromov--Hausdorff stability.

To keep the present treatment self-contained, we instead give a
direct stability proof using explicit maps and homotopies.  This approach
remains within the formalism of categorical invariants and yields the erosion
and Gromov--Hausdorff stability estimates needed here.

\label{subset:persistent invariants}

For any category $\caC$, let $\mor(\caC)$ and $\ob(\caC)$ denote the collection of all morphisms and objects of $\caC$, respectively.

\begin{definition}[{\cite[Definition 2.5]{memoli2024PersistentCup}}]
\label{def:categorical invariant}
A \textbf{categorical invariant} of a category $\caC$ is a map $\invariant :\ob(\caC)\sqcup\mor(\caC)\to\bbN\cup\{\infty\}$ such that
\begin{itemize}
    \item [(i)] $\invariant (\id _X)=\invariant (X)$,~for all $X\in \ob(\caC)$, and
    \item [(ii)] for any maps $X\xrightarrow{f} Y\xrightarrow{g}Z\xrightarrow{h} W$ in $\caC$,
	we have \(\invariant(h\circ g\circ f)\leq \invariant(g).\)
\end{itemize}
Alternatively, condition (ii) can be replaced with the equivalent condition
\begin{itemize}
    \item [(ii')] for any maps $X\xrightarrow{f} Y\xrightarrow{g}Z$ in $\caC$,
	we have $\invariant(g\circ f)\leq \min\{\invariant(f), \invariant(g)\}$.
\end{itemize}
\end{definition}

We write \(\Int_\omega\) for the set of intervals of type
\(\omega\), where \(\omega\) denotes one of the four types:
open--open, open--closed, closed--open, or closed--closed.
The set \(\Int_\omega\) is partially ordered by inclusion,
and we regard it as a poset category where morphisms
correspond to inclusions of intervals.
For simplicity, we often use \(\Int_\omega\) to refer to
both the underlying set and the associated poset category.
Our results apply uniformly to all four types; for
conciseness, we state them only for closed--closed intervals
and omit explicit reference to \(\omega\) unless otherwise
specified.

Given any functor $X_\bullet : (\bbR,\leq)\to\caC$, there is an induced functor $\invariant(X_\bullet)$ \cite[Definition~2.7]{memoli2024PersistentCup} from the poset category $\Int$ of intervals to \((\bbN\cup\{\infty\},\geq)\), given by
\[
\invariant(X_\bullet)([t,s]) := \invariant(X_t \to X_s).
\]
If \([t,s]\subseteq[a,b]\), the factorization
\(X_a\to X_t\to X_s\to X_b\) and
\cref{def:categorical invariant}(ii) give
\(\invariant(X_\bullet)([t,s])\geq
\invariant(X_\bullet)([a,b])\), explaining the reversed order
in the target.
We refer to the induced functor as the \textbf{persistent invariant associated to $\invariant$ and $X_\bullet$}.

We now recall a notion of distance for comparing persistent invariants.

\begin{definition}[{\cite[Definition~3.2]{memoli2024PersistentCup}}]
\label{def:de}
    Let \( \invariant_1, \invariant_2: (\Int, \subseteq) \to (\mathbb{R}_{\geq 0}\cup\{\infty\}, \geq) \) be two functors, and let \( \delta \geq 0 \). We say that \( \invariant_1 \) and \( \invariant_2 \) are \textbf{\(\delta\)-eroded} if,
    for all \( [t,s] \in \Int \),
    \[
    \invariant_1([t,s]) \geq \invariant_2([t- \delta, s+ \delta]) \quad \text{and} \quad
    \invariant_2([t,s]) \geq \invariant_1([t- \delta, s+ \delta]).
    \]

    The \textbf{erosion distance} between \( \invariant_1 \) and \( \invariant_2 \) is defined as
    \[
    d_{\mathrm{E}}(\invariant_1, \invariant_2) := \inf \left\{ \delta \geq 0 \,\middle|\, \invariant_1,\invariant_2 \text{ are }\delta\text{-eroded}
    \right\},
    \]
    with the convention that \( d_{\mathrm{E}}(\invariant_1, \invariant_2) = \infty \) if no such \( \delta \) exists.
\end{definition}

\begin{theorem}[{\cite[Theorem 1]{memoli2024PersistentCup}}]
    \label{thm:stab-per-inv}
Let $\invariant :\ob(\caC)\sqcup\mor(\caC)\to\bbN\cup\{\infty\}$ be a categorical invariant.
The assignment $X_\bullet\mapsto\invariant(X_\bullet)$ is $1$-Lipschitz with respect to the interleaving and erosion distances: for any $X_\bullet ,Y_\bullet :(\bbR,\leq)\to\caC$,
$$d_{\mathrm{E}}(\invariant (X_\bullet ),\invariant (Y_\bullet ))\leq \di(X_\bullet ,Y_\bullet ).$$
\end{theorem}

To establish the Gromov--Hausdorff stability of persistent topological complexity and zero-divisor-cup-length, we introduce a notion of interleaving based on homotopy-commutative diagrams.

\begin{definition}
    \label{def:d_IH}
    Given two persistent spaces $X_\bullet ,Y_\bullet :(\bbR,\leq)\to\topo$, we say that they are \textbf{$\delta$-interleaved-up-to-homotopy} if there exist families of continuous maps
    \[
    \varphi_t:X_t\to Y_{t+\delta},
    \qquad
    \psi_t:Y_t\to X_{t+\delta},
    \qquad t\in\bbR,
    \]
    such that the following diagrams commute up to homotopy for all $t\leq s$ in $\bbR$:
	\begin{center}
		\begin{tikzcd}[column sep={6em,between origins}]
			X_t \ar[dr, "\varphi_t" below left]
			\ar[r,"f_t^s"]
			&
			X_s
			\ar[dr,"\varphi_s " ]&
			\\
			& Y_{t+\delta}
			\ar[r,"g_{t+\delta}^{s+\delta}" below]
			&
			 Y_{s+\delta}
		\end{tikzcd}
    \hspace{0.6cm}
	\begin{tikzcd}[column sep={6em,between origins}]
			&  X_{t+\delta}
			\ar[r,"f_{t+\delta}^{s+\delta}"]
			&
			 X_{s+\delta}
			\\
			Y_t
			\ar[ur, "\psi_t"]
			\ar[r,"g_t^s" below]
			&
			Y_{s}
			\ar[ur,"\psi_s" below right]&
		\end{tikzcd}
	\end{center}
    and
    \begin{center}
		\begin{tikzcd}
			X_t
			\ar[dr, "\varphi_t" below left ]
			\ar[rr,"f_{t}^{t+2\delta}"]
			&&
			X_{t+2\delta}
			\\
			&
            Y_{t+\delta}
			\ar[ur,"\psi_{t+\delta}" below right ]
			&
		\end{tikzcd}
    \hspace{0.6cm}
		\begin{tikzcd}
			&
            X_{t+\delta}
			\ar[dr,"\varphi_{t+\delta}"]
			&
			\\
			Y_t
			\ar[ur, "\psi_t"]
			\ar[rr,"g_{t}^{t+2\delta}" below]
			&&
			 Y_{t+2\delta}.
		\end{tikzcd}
	\end{center}

    We define the \textbf{interleaving-up-to-homotopy distance} between $X_\bullet$ and $Y_\bullet$ as the infimum of all $\delta\geq 0$ such that $X_\bullet$ and $Y_\bullet$ are $\delta$-interleaved-up-to-homotopy, denoted by $\dih(X_\bullet ,Y_\bullet )$.
\end{definition}

We establish the basic properties of $\dih$. Related constructions are
discussed in \cref{para:dIH related work}.

\begin{remark}
\label{rmk:dih_properties}
The quantity \(\dih\) is an extended pseudometric. Symmetry and
\(\dih(X_\bullet,X_\bullet)=0\) follow immediately from the definition.
For the triangle inequality, suppose that \((\varphi,\psi)\) is a
\(\delta\)-interleaving up to homotopy between \(X_\bullet\) and
\(Y_\bullet\), and that \((\alpha,\beta)\) is an
\(\varepsilon\)-interleaving up to homotopy between \(Y_\bullet\) and
\(Z_\bullet\). Then
\[
\Phi_t:=\alpha_{t+\delta}\circ\varphi_t,
\qquad
\Psi_t:=\psi_{t+\varepsilon}\circ\beta_t
\]
define a \((\delta+\varepsilon)\)-interleaving up to homotopy between
\(X_\bullet\) and \(Z_\bullet\). Consequently,
\[
\dih(X_\bullet,Z_\bullet)
\leq
\dih(X_\bullet,Y_\bullet)
+
\dih(Y_\bullet,Z_\bullet).
\]
\end{remark}

\begin{definition}
Let $\varphi, \psi:K\to L$ be two simplicial maps.
We say that $\varphi$ and $\psi$ are \textbf{contiguous} if for every simplex $\sigma$ of $K$,
the union $\varphi(\sigma)\cup \psi(\sigma)$ is a simplex of $L$.
\end{definition}

\begin{lemma}\label{lem:dih-dgh}
Let $X$ and $Y$ be compact metric spaces. Then the Vietoris--Rips filtrations
$\VR_\bullet(X)$ and $\VR_\bullet(Y)$ satisfy
\[
\dih(\VR_\bullet(X),\VR_\bullet(Y)) \le 2\, d_{\mathrm{GH}}(X,Y).
\]
\end{lemma}

\begin{proof}
Let $\delta > d_{\mathrm{GH}}(X,Y)$. Then, by \cref{eq:dgh-maps}, there exist maps
$\varphi:X\to Y$ and $\psi:Y\to X$ such that
\[
\dis(\varphi),\ \dis(\psi),\ \codis(\varphi,\psi) \le 2\delta .
\]

\medskip
\noindent
\textit{Step 1: Construction of simplicial maps.}
For $t\leq0$, define $\varphi_t$ and $\psi_t$ to be the unique maps from the empty complexes to $\VR_{t+2\delta}(Y)$ and $\VR_{t+2\delta}(X)$, respectively.
Every interleaving diagram with such a source commutes strictly, since its source is empty.
It therefore suffices to construct the maps and verify the interleaving diagrams for $t>0$.
If $\sigma$ is a simplex of $\VR_t(X)$ then
$\operatorname{diam}_X(\sigma)< t$. Hence for any vertices
$x,x'\in\sigma$,
\[
d_Y(\varphi(x),\varphi(x'))\le d_X(x,x')+\dis(\varphi)< t+2\delta.
\]
Thus $\varphi(\sigma)$ is a simplex of $\VR_{t+2\delta}(Y)$, and $\varphi$ induces a
simplicial map
\(
\varphi_t:\VR_t(X)\to\VR_{t+2\delta}(Y).
\)
Similarly, $\psi$ induces a simplicial map
\(
\psi_t:\VR_t(Y)\to\VR_{t+2\delta}(X).
\)
We will show that these maps induce a $2\delta$-interleaving-up-to-homotopy between the two Vietoris--Rips filtrations.

\medskip
\noindent
\textit{Step 2: Verification of square diagrams.}
Since the simplicial maps $\varphi_r$ and $\psi_r$ are defined vertexwise from the maps $\varphi$ and $\psi$, they are compatible with the inclusion maps in the Vietoris--Rips filtrations. In particular, for every $t\le s$ the following squares commute strictly:

\begin{center}
\begin{tikzcd}[column sep={7.5em,between origins}]
\VR_t(X)
\ar[dr,"\varphi_t" below left]
\ar[r, hook, "f_t^s"]
&
\VR_s(X)
\ar[dr,"\varphi_s"]
&
\\
&
\VR_{t+2\delta}(Y)
\ar[r, hook, "g_{t+2\delta}^{\,s+2\delta}" below]
&
\VR_{s+2\delta}(Y)
\end{tikzcd}
\hspace{0.5em}
\begin{tikzcd}[column sep={7.5em,between origins}]
&
\VR_{t+2\delta}(X)
\ar[r, hook, "f_{t+2\delta}^{\,s+2\delta}"]
&
\VR_{s+2\delta}(X)
\\
\VR_t(Y)
\ar[ur,"\psi_t"]
\ar[r, hook, "g_t^s" below]
&
\VR_s(Y)
\ar[ur,"\psi_s" below right]
&
\end{tikzcd}
\end{center}
where $\{f_t^s\}$ and $\{g_t^s\}$ denote the structure maps of the Vietoris--Rips filtrations.

\medskip
\noindent
\textit{Step 3: Verification of the triangular diagrams.}
We verify that the two triangular diagrams required in
\cref{def:d_IH} commute up to homotopy. We first show that
the simplicial maps
\( \psi_{t+2\delta}\circ \varphi_t \) and $f_t^{\,t+4\delta}$ are contiguous.
Let $\sigma$ be a simplex of $\VR_t(X)$. Since
$f_t^{\,t+4\delta}$ is the inclusion, it suffices to show that
\[
\sigma \cup (\psi\circ\varphi)(\sigma) \in \VR_{t+4\delta}(X).\]

Let $u,v$ be vertices of $\sigma \cup (\psi\circ\varphi)(\sigma)$.
We want to show $d_X(u,v)< t+4\delta$.
There are three possible cases.
If $u,v\in\sigma$, then $d_X(u,v)< t < t+4\delta$.
If $u,v\in (\psi\circ\varphi)(\sigma)$, then $u=(\psi\circ\varphi)(x)$ and $v=(\psi\circ\varphi)(x')$ for some
$x,x'\in\sigma$, and
\[
d_X(u,v)\le d_X(x,x') + \dis(\psi\circ\varphi)
\le d_X(x,x')+ \dis(\psi) + \dis(\varphi) < t+4\delta.
\]
Lastly, if $u\in\sigma$, $v\in (\psi\circ\varphi)(\sigma)$, then $v=(\psi\circ\varphi)(x)$ for some $x\in\sigma$, and
\begin{align*}
    d_X(u,v) = d_X(u,(\psi\circ\varphi)(x))
\leq & d_Y(\varphi(u), \varphi(x)) + \codis(\varphi, \psi) \\
\leq & d_X(u,x) +\dis(\varphi) + \codis(\varphi, \psi)
< t+4\delta .
\end{align*}

Therefore the maps $\psi_{t+2\delta}\circ\varphi_t$ and
$f_t^{\,t+4\delta}$ are contiguous, and hence their geometric
realizations are homotopic
\cite[Proposition~7(1)]{chowdhury2018functorial}.
By a similar argument, we can show that the maps $\varphi_{t+2\delta}\circ\psi_t$ and
$g_t^{\,t+4\delta}$ are contiguous, and hence homotopic.
Consequently the diagrams
\begin{center}
\begin{tikzcd}[column sep={5em,between origins}]
\VR_t(X)
\ar[dr,"\varphi_t" below left]
\ar[rr,"f_t^{\,t+4\delta}"]
&&
\VR_{t+4\delta}(X)
\\
&
\VR_{t+2\delta}(Y)
\ar[ur,"\psi_{t+2\delta}" below right]
&
\end{tikzcd}
\hspace{2em}
\begin{tikzcd}[column sep={5em,between origins}]
&
\VR_{t+2\delta}(X)
\ar[dr,"\varphi_{t+2\delta}"]
&
\\
\VR_t(Y)
\ar[ur,"\psi_t"]
\ar[rr,"g_t^{\,t+4\delta}" below]
&&
\VR_{t+4\delta}(Y)
\end{tikzcd}
\end{center}
commute up to homotopy.

\medskip
\noindent
\textit{Step 4: Conclusion.}
Step 2 and 3 together imply that $\VR_\bullet(X)$ and $\VR_\bullet(Y)$ are
$2\delta$-interleaved-up-to-homotopy. Letting $\delta \downarrow d_{\mathrm{GH}}(X,Y)$
proves the lemma.
\end{proof}

Correspondence and contiguity arguments of this type are
standard in stability proofs for Vietoris--Rips persistent homology;
see \cite[Proposition~3.3 and Lemma~4.3]{CSO14}.  For finite metric
spaces, M\'emoli and Singhal give a related map-based proof
\cite[Theorem~5.10]{memoli2019primer}.  Closely related naturality and
interleaving relations up to contiguity also appear for filtered
Dowker complexes
\cite[Proposition~7 and Lemma~8]{chowdhury2018functorial}.  We include
the details to make the stability argument self-contained.

We now apply \cref{lem:dih-dgh} to establish stability of persistent
invariants.

\stabilitytcnew*

\begin{proof}
    Given any $\delta > \dih(X_\bullet, Y_\bullet)$, there exist maps $\varphi_\bullet :X_\bullet \to Y_{\bullet +\delta}$ and $\psi_\bullet :Y_\bullet \to X_{\bullet +\delta}$ such that the diagrams in \cref{def:d_IH} commute up to homotopy.

    We claim that $\invariant (X_\bullet)$ and $\invariant (Y_\bullet)$ are $\delta$-eroded.
    For any interval $[t,s]\in \Int$, we have
    the following diagram commuting up to homotopy:
    \begin{center}
    \begin{tikzcd}[column sep={6em,between origins}]
        & X_t \ar[dr, "\varphi_t" below left]
        \ar[r,"f_t^s"]
        &
        X_s
        \ar[dr,"\varphi_s " ]&
        \\
        Y_{t-\delta}
        \ar[rr, "g_{t-\delta}^{t+\delta}" below]
        \ar[ur, "\psi_{t-\delta}"]
        & & Y_{t+\delta}
        \ar[r,"g_{t+\delta}^{s+\delta}" below]
        &
        Y_{s+\delta}
    \end{tikzcd}
    \end{center}
    Homotopy invariance of $\invariant$ and the categorical
composition inequality give:
    \begin{align*}
        \invariant (X_\bullet)([t,s])
        &= \invariant (f_t^s) \\
        &\geq \invariant (\varphi_s\circ f_t^s \circ \psi_{t-\delta}) & \text{(by \cref{def:categorical invariant} (ii))} \\
        &= \invariant ( g_{t+\delta}^{s+\delta} \circ g_{t-\delta}^{t+\delta} )
        & \text{(by commutative diagram up to homotopy)}\\
        &= \invariant (g_{t-\delta}^{s+\delta}) \\
        &= \invariant (Y_\bullet)([t-\delta,s+\delta]).
    \end{align*}
    Similarly, we have $\invariant (Y_\bullet)([t,s]) \geq \invariant (X_\bullet)([t-\delta,s+\delta])$.
    Thus, $\invariant (X_\bullet)$ and $\invariant (Y_\bullet)$ are $\delta$-eroded, implying \eqref{eq:stability1} by the arbitrariness of $\delta$.

    For the Vietoris-Rips filtrations of compact metric spaces, the statement follows from \cref{eq:stability1} and \cref{lem:dih-dgh}.
\end{proof}

\begin{remark}
\label{para:dIH related work}
Homotopy-invariant variants of the interleaving distance have been
studied extensively in recent years.

For persistent spaces taking values in cofibrant spaces, as is the
case for geometric realizations of filtered simplicial complexes,
$\dih$ agrees with the homotopy-commutative interleaving distance of
Blumberg and Lesnick \cite{blumberg2023universality}; for arbitrary
persistent spaces, the latter is bounded above by $\dih$. More
broadly, categorical approaches to distances in persistence have
been developed in
\cite{lesnick2012multidimensional,masaki2018persistent,
patrizio2019persistent}. Comparisons among homotopy-invariant
interleaving distances are given in
\cite{blumberg2023universality,lanari2023rectification}. For filtered
simplicial complexes, M\'emoli and Okutan introduce a related
construction based on iterated contiguity
\cite{memoli2021quantitative}.

\end{remark}

\subsection{Persistent Topological Complexity and Cohomological Lower Bound}
\label{subsec:def_persistent_TC}
\label{sec:persistent_TC}

In this section, we lift both topological complexity and its cohomological lower bound, the zero-divisor-cup-length, to persistent invariants and study their properties, including stability.

It follows from \cref{prop:TC of a map} \ref{prop:indentity} and \ref{prop:map composition} that the topological complexity of spaces and maps together satisfy the following conditions:
\begin{itemize}
    \item [(i)] $\TC (\id _X)= \TC (X)$, for any $X$ in $\topo$, and
    \item [(ii)] \(\TC(h\circ g\circ f)\leq \TC(g),\) for any maps $X\xrightarrow{f} Y\xrightarrow{g}Z\xrightarrow{h} W$.
\end{itemize}
Thus, $\TC$ defines a categorical invariant on $\topo$ (see \cref{def:categorical invariant}).
When applied to persistent spaces, this invariant gives rise to persistent invariants, according to \cref{subset:persistent invariants}.

More precisely, we introduce the following definition.

\begin{definition}[Persistent topological complexity]
\label{def:p_TC}
    Let $X_\bullet$ be a persistent space over $(\bbR, \leq)$, so that each structure map $f_t^s:X_t\to X_s$ is a map between spaces.
    The \textbf{persistent topological complexity} of $X_\bullet$ is the functor
    \begin{align*}
        \TC (X_\bullet):(\Int, \subseteq)&\to (\bbN\cup \{+\infty\},\, \geq)\\
        [t,s]&\mapsto \TC \left(f_t^s: X_t\to X_s\right).
    \end{align*}
\end{definition}

To define persistent zero-divisor-cup-length, we need to verify
that $\zcup$ is a categorical invariant.
% Analogously,  \cref{prop:zcup is categorical} shows that the zero-divisor-cup-length defines a categorical invariant on the category of topological spaces, so we can also lift it to a persistent invariant.

\begin{proposition} \label{prop:zcup is categorical}
The zero-divisor-cup-length satisfies the following conditions:
\begin{itemize}
    \item [(i)] $\zcup(\id_X) = \nil(\kernel(\Delta_X^*)) = \zcup(X)$, for any $X$ in $\topo$.
    \item [(ii)] \(\zcup(h\circ g\circ f)\leq \min\{\zcup(g),\zcup(h)\},\) for any maps $X\xrightarrow{f} Y\xrightarrow{g}Z\xrightarrow{h} W$ between topological spaces.
\end{itemize}
In particular, $\zcup$ is a categorical invariant on $\topo$.
\end{proposition}

\begin{proof}
The equality $\zcup(\id_X)=\zcup(X)$ follows directly from \cref{def:zcl of f}.
For the composition inequality, it suffices to prove
\begin{equation}\label{eq:zcl monotone}
    \zcup(g \circ f) \leq \min\{\zcup(f), \zcup(g)\}
\end{equation}
for any composable maps $f$ and $g$; applying \cref{eq:zcl monotone}
twice then gives the required inequality for $h\circ g\circ f$.
By \cref{rmk:zcl-map-formulas}, restrict the pullbacks to the diagonal kernels:
\[
\ker\Delta_Z^*\xrightarrow{(g\times g)^*}
\ker\Delta_Y^*\xrightarrow{(f\times f)^*}
\ker\Delta_X^*.
\]
The image of the composite is contained in the image of the second arrow, so its nilpotency is at most $\zcup(f)$.
It is also the homomorphic image, under the second arrow, of the image of the first arrow, so its nilpotency is at most $\zcup(g)$.
This proves \cref{eq:zcl monotone}.
\end{proof}

\begin{definition}[Persistent zero-divisor-cup-length]
\label{def:p_zcup}
    Let $X_\bullet$ be a persistent space over $(\bbR, \leq)$, and fix the coefficient field $\field$ used in cohomology.
    The \textbf{persistent zero-divisor-cup-length} of $X_\bullet$ is the functor
\begin{align*}
\zcup(X_\bullet) : (\Int,\subseteq)
    &\longrightarrow (\mathbb{N}\cup\{+\infty\},\geq) \\[4pt]
[t,s]
    &\longmapsto \zcup\left(f_t^s: X_t\to X_s\right) \\[4pt]
    &\phantom{\longmapsto}
      = \nil\Bigl(
        \operatorname{im}\bigl(
        (f_t^s\times f_t^s)^* :
        \ker(\Delta_{X_s}^*)
        \longrightarrow
        \ker(\Delta_{X_t}^*)
        \bigr)
      \Bigr).
\end{align*}
\end{definition}

The following properties of persistent topological complexity, which extend the properties of TC of spaces, are direct consequences of \cref{prop:TC of a map} \ref{prop:cup cat TC} and \ref{prop:cohomology estimate}.
\begin{proposition}\label{prop:persistent-tc-bounds}
Let $X_\bullet$ be a persistent space.
    \begin{itemize}
        \item $\cl(X_\bullet) \leq \cat(X_\bullet) \leq \TC(X_\bullet).$
        \item If $X_t$ is an ANR for each $t$, then $\zcup(X_\bullet) \leq \TC(X_\bullet)$.
    \end{itemize}
\end{proposition}

\label{subsec:stability}

\begin{corollary}\label{cor:stability-tc-zcl}
    Inequalities \eqref{eq:stability1} and \eqref{eq:dE-dGH} hold for the persistent topological complexity and the persistent zero-divisor-cup-length.
\end{corollary}

\begin{proof}
    For the persistent topological complexity, the claim follows from \cref{prop:TC of a map} \ref{prop:map homotopy invariance}, which verifies that $\TC$ satisfies the condition of \cref{thm:stability}.

    For the persistent zero-divisor-cup-length, use \cref{prop:zcup is categorical}. If $f,g:X\to Y$ are homotopic maps between topological spaces, then $(f\times f)^*=(g\times g)^*$ on cohomology with coefficients in $\field$, so the relevant images in $\ker(\Delta_X^*)$ coincide. Hence $\zcup(f)=\zcup(g)$, and $\zcup$ depends only on homotopy classes of maps. Thus $\zcup$ satisfies the hypothesis of \cref{thm:stability}, which gives the claimed erosion bound for persistent $\zcup$.
\end{proof}

\subsection{Persistent Topological Complexity of Disconnected Spaces}
\label{subsec:persistent-tc-disconnected}

Vietoris--Rips complexes need not be path-connected; for finite metric
spaces, they are discrete at sufficiently small positive scales.
Under the usual motion-planning definition in \cref{def:TC},
$\TC(X)=+\infty$ whenever $X$ has more than one path component.
We discuss alternative conventions for disconnected spaces and show
how Farber's componentwise convention extends to a stable persistent
invariant.

Farber's original definition assumes path-connectedness
\cite[Section~1 and Definition~2]{farber2003}, as does the
definition for maps of Murillo--Wu
\cite[Section~1]{murillo2021topological}.
Scott adopts the same assumption after discussing disconnected
spaces in \cite[Remark~3.2]{scott2022topological}.
For a map $f:X\to Y$, \cref{def:TC of f} gives
$\TC(f)=+\infty$ whenever $f(X)$ meets more than one path
component of $Y$.
A disconnected source alone does not imply infinite topological
complexity; for example, every constant map has $\TC(f)=0$.

A second convention, used by Farber, considers motion planning
within individual path components and assigns to $X$ the maximum
of their topological complexities
\cite[Section~8, pp.~169--170]{farber2008random}.
To allow an arbitrary collection of path components, we use
the supremum:
\[
\TC_{\max}(X)=\sup_{C\in\pi_0(X)}\TC(C),
\]
where $\pi_0(X)$ denotes the set of path components of $X$.

\begin{remark}\label{rmk:zcl-degree-zero}
By \cref{def:zcl of f},
$\zcup(f)=+\infty$ whenever $f(X)$ meets more than one path
component of $Y$.
In particular, for the discrete two-point space $X=\{x_1,x_2\}$, we have
$\TC(X)=\zcup(X)=+\infty$, whereas $\TC_{\max}(X)=0$.
Hence $\zcup$ is not a lower bound for $\TC_{\max}$.
This motivates a componentwise definition of zero-divisor-cup-length as well.
\end{remark}

We now extend the maximum convention to maps in the general setting
of categorical invariants. Let $\invariant$ be a homotopy-invariant
categorical invariant of nonempty path-connected spaces and maps.
For a map $f:X\to Y$ between arbitrary spaces and a path component
$C$ of $X$, let $Y_C$ be the path component of $Y$ containing
$f(C)$. Define
\begin{equation}\label{eq:componentwise-map-invariant}
\invariant_{\max}(f):=\sup_{C\in\pi_0(X)}\invariant(f|_C:C\to Y_C),
\qquad
\invariant_{\max}(X):=\sup_{C\in\pi_0(X)}\invariant(C),
\end{equation}
where the supremum of the empty set is $0$.
For $\invariant=\TC$, this recovers Farber's convention on spaces.
For a map $f$, $\TC_{\max}(f)$ measures motion planning separately
on each source component $C$, with paths joining image points
throughout the target component $Y_C$.

The same construction defines $\cat_{\max}$, $\cl_{\max}$,
and $\zcup_{\max}$.
For cup-length, we have $\cl_{\max}=\cl$, since $\cl_{\max}(f)=\cl(f)$ for every map $f$:
nonzero cup products are detected on path components, and
cohomology classes on a target component extend by zero
to the other components. 
In contrast, $\TC_{\max}$, $\zcup_{\max}$, and $\cat_{\max}$
can differ from their usual counterparts on disconnected spaces,
as illustrated in
\cref{ex:two-point-rips-consistency}.

\smallskip

We verify that the categorical properties and homotopy invariance
of an invariant $\invariant$ extend to $\invariant_{\max}$ on arbitrary spaces.

\begin{proposition}\label{prop:componentwise-categorical}
Let $\invariant$ be a homotopy-invariant categorical invariant
on the category of nonempty path-connected spaces.
Then $\invariant_{\max}$ is a homotopy-invariant categorical
invariant on $\topo$.
\end{proposition}

\begin{proof}
For each path component $C$ of $X$, we have
$\invariant(\id_C)=\invariant(C)$.
Taking suprema gives
$\invariant_{\max}(\id_X)=\invariant_{\max}(X)$.

Let $X\xrightarrow{f}Y\xrightarrow{g}Z$ be continuous maps.
For each path component $C$ of $X$, the restriction of $g\circ f$
factors as
\[
C\xrightarrow{f|_C}Y_C\xrightarrow{g|_{Y_C}}Z_C,
\]
where $Z_C$ is the path component of $Z$ containing $g(Y_C)$.
Since $\invariant$ is categorical,
\begin{align*}
\invariant((g\circ f)|_C)
&\leq \min\{\invariant(f|_C),\invariant(g|_{Y_C})\}\\
&\leq \min\{\invariant_{\max}(f),\invariant_{\max}(g)\}.
\end{align*}
Taking the supremum over $C$ proves the composition inequality
for $\invariant_{\max}$.

Finally, let $H:X\times[0,1]\to Y$ be a homotopy from $f_0$ to $f_1$.
For each path component $C$ of $X$, the set $C\times[0,1]$ is
path-connected, so its image under $H$ lies in the path component
$Y_C$ containing $f_0(C)$.
Thus $f_0|_C,f_1|_C:C\to Y_C$ are homotopic, and
$\invariant(f_0|_C)=\invariant(f_1|_C)$.
Taking suprema gives
$\invariant_{\max}(f_0)=\invariant_{\max}(f_1)$.
\end{proof}

The properties in
\cref{prop:TC of a map}\,\ref{prop:indentity}--\ref{prop:map composition left h.i.}
remain valid with $\TC$ replaced by $\TC_{\max}$.
Indeed, they follow from the categorical axioms (\cref{def:categorical invariant}) and homotopy
invariance in \cref{prop:componentwise-categorical}.
Moreover, applying
\cref{prop:TC of a map}\,\ref{prop:cup cat TC}
to each restriction $f|_C:C\to Y_C$ and taking suprema over
$C\in\pi_0(X)$ gives
\[
\cl_{\max}(f)\leq\cat_{\max}(f)\leq\TC_{\max}(f)
\leq\cat_{\max}(f\times f).
\]
For the upper bound, we also use that each $C\times C$ is a
path component of $X\times X$.
If $X$ is an ANR, so is each of its path components.
Applying \cref{prop:TC of a map}\,\ref{prop:cohomology estimate}
to these restrictions and again taking suprema gives
$\zcup_{\max}(f)\leq\TC_{\max}(f)$.

\subsubsection*{Persistence and stability.}
Let $X_\bullet$ be a persistent space with structure maps
$f_t^s:X_t\to X_s$. By \cref{prop:componentwise-categorical},
evaluating $\TC_{\max}$ on these maps gives a persistent invariant:
\begin{equation}\label{eq:persistent-componentwise-tc}
\TC_{\max}(X_\bullet)([t,s])
=\sup_{C\in\pi_0(X_t)}\TC(f_t^s|_C:C\to Y_C),
\end{equation}
where $Y_C$ is the path component of $X_s$ containing $f_t^s(C)$.
% Distinct components of $X_t$ may map into the same component of $X_s$.
Replacing $\TC$ by $\zcup$ in
\cref{eq:persistent-componentwise-tc} gives persistent $\zcup_{\max}$,
with $\zcup_{\max}(X_\bullet)\leq\TC_{\max}(X_\bullet)$
whenever each $X_t$ is an ANR.

\begin{corollary}\label{cor:stability of tc}
The stability inequalities in \cref{eq:stability1,eq:dE-dGH}
hold for the persistent invariants
$\TC_{\max}$, $\zcup_{\max}$, $\cat_{\max}$, and $\cl_{\max}$.
\end{corollary}

\begin{proof}
This follows from \cref{thm:stability} and
\cref{prop:componentwise-categorical}.
\end{proof}

\begin{example}[A two-point metric space]\label{ex:two-point-rips-consistency}
Let $X=\{x_1,x_2\}$ with $d(x_1,x_2)=a>0$.
Under the open Rips convention, $\VR_t(X)$ is empty for $t\leq0$,
consists of two isolated vertices for $0<t\leq a$, and is
a single edge for $t>a$.
Thus, for $\invariant\in\{\TC,\zcup,\cat\}$ and $t\leq s$,
\[
\begin{aligned}
\invariant(\VR_\bullet(X))([t,s])
&=
\begin{cases}
+\infty,
&\invariant\in\{\TC,\zcup\},\quad 0<t\leq s\leq a,\\
1,
&\invariant=\cat,\quad 0<t\leq s\leq a,\\
0,&\text{otherwise},
\end{cases}\\
\invariant_{\max}(\VR_\bullet(X))([t,s])&=0.
\end{aligned}
\]
For all three cases, $\invariant$ differs from $\invariant_{\max}$. However, $\cl(\VR_\bullet(X))=\cl_{\max}(\VR_\bullet(X))\equiv 0$. 
\end{example}

For the metric spaces in \cref{subsec:persistent-computations-main},
the Rips complexes are path-connected at every positive scale,
so persistent $\TC=\TC_{\max}$ and $\zcup=\zcup_{\max}$.
\section{Examples and Distinguishing Power of Persistent Invariants}

\label{sec:examples-main}

Having defined persistent versions of \(\TC\), \(\zcup\), \(\cat\), and
\(\cl\), and established their stability, we turn to their distinguishing
power. 
We begin with the classical comparisons in
\cref{subsec:classical-comparisons-main}, which guide the
choice of examples.
We then construct metric realizations and study their
Vietoris--Rips filtrations in \cref{subsec:vr-models-main},
and compute their persistent invariants in
\cref{subsec:persistent-computations-main}.
In \cref{subsec:distinguishing-main}, we compare the distinguishing
power of these persistent invariants using both explicit computations
and partial information about the Vietoris--Rips filtrations.

\subsection{Classical Examples and Comparisons}
\label{subsec:classical-comparisons-main}

Throughout \cref{sec:examples-main}, we write
\(\mathbb{M}:=\bbS^1\vee\bbS^2\vee\bbS^1\) for the Mickey Mouse space.
Let \(n\geq2\), set \(\wedgeS:=\bigvee_{i=1}^n\bbS^i\), and let
\(\beta_i\) denote the \(i\)-th Betti number.  For \(\zcup\), \(\cl\),
and the Betti numbers, we fix a coefficient field \(\field\).  
We specify the coefficient field when the results depend on its characteristic.
The relevant values are given in \cref{table:invariant-values-main}.

\begin{table}[H]
    \caption{Classical invariant values for selected spaces.
    The values hold over every field unless a coefficient restriction is stated.}
    \label{table:invariant-values-main}
    \centering
    \begin{tabular}{|c|ccccc|}
    \hline
     & \(\TC\) & \(\zcup\) & \(\cat\) & \(\cl\)
     & \((\beta_i)_{i\geq0}\) \\
    \hline
    \(\bbS^1\)
        & \(1\) & \(1\) & \(1\) & \(1\) & \((1,1,0,\ldots)\) \\
    \(\bbS^2\), \(\operatorname{char}(\field)=2\)
        & \(2\) & \(1\) & \(1\) & \(1\) & \((1,0,1,0,\ldots)\) \\
    \(\bbS^2\), \(\operatorname{char}(\field)\neq2\)
        & \(2\) & \(2\) & \(1\) & \(1\) & \((1,0,1,0,\ldots)\) \\
    \(\bbS^1\vee\bbS^1\)
        & \(2\) & \(2\) & \(1\) & \(1\) & \((1,2,0,\ldots)\) \\
    \(\mathbb{M}=\bbS^1\vee\bbS^2\vee\bbS^1\)
        & \(2\) & \(2\) & \(1\) & \(1\) & \((1,2,1,0,\ldots)\) \\
    \(\bbT^2=\bbS^1\times\bbS^1\)
        & \(2\) & \(2\) & \(2\) & \(2\) & \((1,2,1,0,\ldots)\) \\
    \hline
    \(\bbT^2\vee\bbT^2\)
        & \(4\) & \(4\) & \(2\) & \(2\)
        & \((1,4,2,0,\ldots)\) \\
    \(\bbT^2\vee \mathbb{M}\)
        & \(3\) & \(3\) & \(2\) & \(2\)
        & \((1,4,2,0,\ldots)\) \\
    \hline
    \(\rp^n\), \(\field=\bbF_2\)
        & \(>2\) & \(>2\) & \(n\) & \(n\)
        & \((\underbrace{1,\ldots,1}_{n+1},0,\ldots)\) \\
    \(\wedgeS\), \(\field=\bbF_2\)
        & \(2\) & \(2\) & \(1\) & \(1\)
        & \((\underbrace{1,\ldots,1}_{n+1},0,\ldots)\) \\
    \hline
    \end{tabular}
\end{table}

\Cref{table:invariant-values-main} yields the following comparisons:
\begin{enumerate}[label=(\roman*)]
    \item \label{comp:s1-s2-main}
    \(\bbS^1\) vs.\ \(\bbS^2\) (over \(\bbF_2\)):
    \(\zcup\), \(\cat\), and \(\cl\) agree, whereas \(\TC\) and the
    Betti numbers differ.
    \item \label{comp:s2-wedge11-main}
    \(\bbS^2\) vs.\ \(\bbS^1\vee\bbS^1\) (over \(\bbF_2\)):
    \(\TC\), \(\cat\), and \(\cl\) agree, whereas \(\zcup\) and the
    Betti numbers differ.
    \item \label{comp:wedge11-wedge121-main}
    \(\mathbb{M}\) vs.
    \(\bbS^1\vee\bbS^1\): all four numerical homotopy
    invariants agree, whereas the Betti numbers differ.
    \item \label{comp:wedge121-torus-main}
    \(\mathbb{M}\) vs.\
    \(\bbT^2\): the Betti numbers, \(\TC\), and \(\zcup\)
    agree, whereas \(\cat\) and \(\cl\) differ.
    \item \label{comp:tc-zcl-separation-main}
    \(\bbT^2\vee\bbT^2\) vs.\
    \(\bbT^2\vee\mathbb{M}\): the Betti numbers,
    \(\cat\), and \(\cl\) agree, whereas \(\TC\) and \(\zcup\) differ.
    \item \label{comp:rpn-wedge-main}
    \(\rp^n\) vs.\ \(\wedgeS\) (over \(\bbF_2\)): the Betti numbers
    agree in every degree, whereas \(\TC\), \(\zcup\), \(\cat\), and
    \(\cl\) differ.
\end{enumerate}

For  the spaces in \cref{comp:s1-s2-main,comp:s2-wedge11-main,comp:wedge11-wedge121-main,comp:wedge121-torus-main}, the \(\TC\)-values follow from
\cref{prop:tc of a space}\,\ref{prop-item:product of spheres} and
\cref{ex:tc of bouquets}, the \(\cat\)-values from
\cref{prop:properties of cat}\,\ref{prop-item:cat wedge} and
\ref{prop-item:cat-product-spheres}, and the sphere values of
\(\zcup\) from \cref{ex:zcup of Sn}.  The \(\cl\)-values follow from
the corresponding cohomology rings and \cref{prop:cl<=cat}.
For the remaining \(\zcup\)-values, choose classes supported on
distinct wedge summands for
\(\bbS^1\vee\bbS^1\) and
\(\bbS^1\vee\bbS^2\vee\bbS^1\), and choose the two degree-one
generators for \(\bbS^1\times\bbS^1\).  In each case, the associated
zero-divisors have nonzero product.  Thus \(\zcup\geq2\), and equality
follows from \cref{prop:zcup<=tc}.

For  the spaces in \cref{comp:s1-s2-main,comp:s2-wedge11-main,comp:wedge11-wedge121-main,comp:wedge121-torus-main}, the equality \(\cat=\cl\) is preserved under
finite products and wedge sums of positive-dimensional spheres. On these
examples, \(\cat\) and \(\cl\) therefore distinguish exactly the same
pairs. Examples for which the two invariants
behave differently are given in
\cite[Examples~2.30 and~2.34]{memoli2024PersistentCup}.

\Cref{ex:tc-wedges-tori} establishes the values in
\cref{table:invariant-values-main} for the pair \ref{comp:tc-zcl-separation-main} \(\bbT^2\vee\bbT^2\) and
\(\bbT^2\vee\mathbb{M}\).

\begin{example}\label{ex:tc-wedges-tori}
Consider
\(\bbT^2\vee\bbT^2\) and
\(\bbT^2\vee\mathbb{M}\), with basepoints at vertices.
The wedge formula in \cref{prop:properties of cat} and the
standard homology and cohomology formulas
\cite[Corollary~2.25 and Examples~3.13--3.14]{hatcher2002algebraic}
give category \(2\), cup-length \(2\), and Betti numbers
\((1,4,2,0,\ldots)\) for both spaces, over every field.

By \cref{prop:tc of a space,ex:tc of bouquets},
\(\TC(\bbT^2)=\TC(\mathbb{M})=2\), while
\(\cat(\bbT^2\times\bbT^2)=4\) by \cref{prop:properties of cat}.
Also, \(\cat(\bbT^2\times \mathbb{M})=3\): category subadditivity gives the upper bound,
and the product of two torus generators with a nonzero positive-degree class of
\(\mathbb{M}\) gives the cup-length lower bound.
Theorem~6 of \cite{dranishnikov2017freeproduct} therefore gives
\[
\begin{aligned}
\TC(\bbT^2\vee\bbT^2)&=\max\{2,2,4\}=4,\\
\TC(\bbT^2\vee\mathbb{M})
&=\max\{2,2,3\}=3.
\end{aligned}
\]
Both applications satisfy the theorem's dimension hypothesis:
the summands are connected CW complexes of dimension \(2\), and each maximum is at least \(3\).

For \(\zcup\), use the graded K\"unneth identification and
write \(\bar x=x\otimes1-1\otimes x\).
For \(\bbT^2\vee\bbT^2\), choose degree-one generators \(a,b\)
from the first torus and \(c,d\) from the second.
Products between distinct summands vanish, so
\[
\bar a\bar b\bar c\bar d=ab\otimes cd+cd\otimes ab\ne0.
\]
For \(\bbT^2\vee\bbS^1\vee\bbS^2\vee\bbS^1\), choose torus
generators \(a,b\) and a nonzero degree-two class \(v\) on the \(2\)-sphere.
Then
\[
\bar a\bar b\bar v=-ab\otimes v+v\otimes ab\ne0.
\]
The two tensors in each expression are linearly independent over
every field. Together with \cref{prop:zcup<=tc}, these products give
\(\zcup(\bbT^2\vee\bbT^2)=4\) and \(\zcup(\bbT^2\vee \mathbb{M})=3\).
The cohomology rings differ: products of degree-one classes span a
two-dimensional subspace in degree \(2\) for \(\bbT^2\vee\bbT^2\),
and a one-dimensional subspace for \(\bbT^2\vee \mathbb{M}\).
\end{example}

 For the pair \ref{comp:rpn-wedge-main} \(\rp^n\) and \(\wedgeS\), \cref{ex:tc of rpn} gives
\(\TC(\rp^n)\geq\zcup(\rp^n)>2\), and \cite[Theorem~3]{oprea2014applications} gives
\(\cat(\rp^n)=n\).  Moreover,
\[
\rmH^*(\rp^n;\bbF_2)\cong\bbF_2[a]/(a^{n+1}),
\qquad |a|=1,
\]
so \(\cl(\rp^n)=n\) and \(\beta_i(\rp^n)=1\) for
\(0\leq i\leq n\).  By \cref{ex:tc of bouquets},
\(\TC(\wedgeS)=2\), while \(\cat(\wedgeS)=\cl(\wedgeS)=1\).
Nonzero classes in degrees \(1\) and \(2\), supported on distinct
summands, give \(\zcup(\wedgeS)\geq2\); equality follows from
\cref{prop:zcup<=tc}.  Moreover,
\(\beta_i(\wedgeS)=1\) for \(0\leq i\leq n\).

\subsection{Examples of Vietoris--Rips Filtrations}
\label{subsec:vr-models-main}

We use two families of base metrics: \(\ell_\infty\) metrics on
cube boundaries and geodesic metrics on round spheres.
Throughout \cref{sec:examples-main}, products carry maximum
product metrics, wedges gluing metrics, and real projective
spaces antipodal quotient metrics, unless stated otherwise.
We specify the metric when introducing a space and omit it
from the notation thereafter when unambiguous.
Any rescaling or truncation is stated explicitly.

For every filtration considered here, all structure maps within
each displayed interval of constant homotopy type are homotopy
equivalences. We will not repeat this assertion in the individual
examples.

Let the Cartesian product \(X\times Y\) of two metric spaces be equipped with
the \textbf{maximum product metric}
\begin{equation}\label{eq:ell-infty-product-main}
d_{\max}((x,y),(x',y')):=\max\{d_X(x,x'),\,d_Y(y,y')\}.
\end{equation}
Let the wedge sum \(X\vee_{x_0\sim y_0}Y\) be equipped with the
\textbf{gluing metric} (cf.~\cite{burago2001course}; see also
\cite{adamaszek2020homotopy}) defined by
\begin{equation}\label{eq:gluing-main}
d_\vee(x,y):=d_X(x,x_0)+d_Y(y,y_0),
\qquad x\in X,\;y\in Y,
\end{equation}
while restricting to \(d_X\) and \(d_Y\) on \(X\times X\) and \(Y\times Y\), respectively.
The product and wedge constructions generalize to multiple spaces.
We often omit the base points in the wedge sums.
The behavior of Vietoris--Rips complexes under products and wedge sums was studied in \cite{adamaszek2017vietoris,adamaszek2020homotopy}, respectively.
We summarize them as follows.

\begin{proposition}[{\cite[Proposition 10.2]{adamaszek2017vietoris}; \cite[Proposition~1]{adamaszek2020homotopy}}]
\label{cor:VR-product-wedge-main}
Let \(X\) and \(Y\) be pointed metric spaces. For each \(t>0\) there are
basepoint-preserving homotopy equivalences
\[
\begin{gathered}
|\VR_t(X)| \times |\VR_t(Y)|
\;\simeq\;
|\VR_t(X\times Y)|,\\
|\VR_t(X)| \vee |\VR_t(Y)|
\;\simeq\;
|\VR_t(X\vee Y)|.
\end{gathered}
\]
Moreover, these equivalences are compatible with the inclusions
\(\VR_t\hookrightarrow \VR_{s}\) for \(0<t\le s\).
\end{proposition}

If \(X=-X \subset\bbR^{n+1}\setminus\{0\}\) is a metric space on which the antipodal map is an isometry, we
equip the orbit space \(X/(\pm1)\) with the \textbf{quotient metric}
\[
d_{X/(\pm1)}([x],[x'])
:=
\min\bigl\{d_X(x,x'),d_X(x,-x')\bigr\}.
\]

We also use rescaling and metric truncation. Rescaling the metric rescales the
filtration parameter: for \(\lambda>0\), the identity on the vertex set gives
\[
\VR_t(X,\lambda d)=\VR_{t/\lambda}(X,d).
\]
For \(a>0\), the \textbf{$a$-truncated metric} \cite[p.~88]{deza2016encyclopedia} is defined by
\[
d^{(a)}(x,x'):=\min\{d(x,x'),a\}.
\]
When combining truncation and rescaling, we place the scalar outside
the truncated metric, using
\[
(\lambda d)^{(a)}=\lambda d^{(a/\lambda)}
\qquad (\lambda,a>0).
\]
Truncation commutes with products and antipodal quotient metrics, but not
with gluing in general. Accordingly, wedge metrics are truncated only after
the gluing metric has been formed. 
The following lemma records the effect of
truncation on the Vietoris--Rips filtration.

\begin{lemma}
\label{lem:capped-metric-main}
Let \((X,d_X)\) be a nonempty metric space and let \(a>0\).  Then
\(d_X^{(a)}(x,x'):=\min\{d_X(x,x'),a\}\) is a metric inducing the same
topology as \(d_X\), and
\[
\VR_t\bigl(X,d_X^{(a)}\bigr)
=
\begin{cases}
\VR_t(X,d_X), & 0<t\leq a,\\[2mm]
\operatorname{Simp}(X), & t>a,
\end{cases}
\]
where \(\operatorname{Simp}(X)\) is the full simplex on \(X\).  
\end{lemma}

\begin{proof}
The function \(\lambda\mapsto\min\{\lambda,a\}\) is nondecreasing,
vanishes only at \(0\), and is subadditive, since
\[
\min\{s+t,a\}
\leq
\min\{s,a\}+\min\{t,a\}
\qquad (s,t\geq0).
\]
Thus \(d_X^{(a)}\) is a metric.  Its open balls of radius less than \(a\)
agree with those of \(d_X\).  For every nonempty finite
\(\sigma\subseteq X\),
\(
\diam_{d_X^{(a)}}(\sigma)
=
\min\{\diam_{d_X}(\sigma),a\}.
\)
\end{proof}

\subsubsection{\texorpdfstring{\(\ell_\infty\)}{l-infinity} models.}
\label{subsubsec:l_infty-main}

We begin with \(\ell_\infty\) metrics because they lead to particularly
tractable Vietoris--Rips filtrations.  Using neighborhoods in injective metric spaces
\cite{lim2024vietoris}, we can determine the relevant homotopy types and structure maps across the
entire filtration for simple spaces, thereby reducing the persistent computations to
classical computations of the underlying invariants.

Let $\mathbb{R}^n$ be equipped with the $\ell_\infty$ metric: for any $x=(x_1,\dots,x_n), x'=(x_1',\dots,x_n') \in \mathbb{R}^n$,
\[
d_\infty(x,x') := \max_{1 \le i \le n} |x_i - x_i'|.
\]

By Theorem~4.1 of \cite{lim2024vietoris}, given an injective metric space $E$ and a nonempty subset $X\subset E$, the Vietoris--Rips filtration (after rescaling) of $X$ is naturally homotopy equivalent to the filtration by open neighborhoods
\[
r \longmapsto B_r(X; E) \simeq \mathrm{VR}_{2r}(X).
\]
The precise definition of injective metric spaces is not central to this paper, so we refer interested readers to \cite{dress2012basic,land2013injective} for the foundations of injective spaces and to \cite{lim2024vietoris} for their connections to applied topology, including the results used here.
In particular, the metric space $(\mathbb{R}^n,d_\infty)$ is injective, and closed cubes are injective as well. This will be sufficient for the applications considered in this work.

\begin{example} \label{ex:cube-main}
    Let $\blacksquare^{{n+1}}:= [-1,1]^{n+1}$ be the cube in $\mathbb{R}^{n+1}$ and $\square^{n}$ be its boundary.
    Throughout, cubes $\blacksquare^{n+1}$ and their boundaries $\square^n$ carry the metric induced by the ambient $\ell_\infty$ norm; we omit the metric from their notation.
    By Lemma~7.16 and Corollary~7.17 of \cite{lim2024vietoris},
    \begin{equation}
        \mathrm{VR}_{2r}(\square^{n})
        \simeq
        B_r(\square^{n};\blacksquare^{{n+1}})
        =
        \begin{cases}
        \blacksquare^{{n+1}}\setminus [r-1, 1-r]^{n+1} \simeq \bbS^{n}, & 0<r\le 1,\\[6pt]
        \blacksquare^{{n+1}} \simeq \ast, & r>1.
        \end{cases}
    \end{equation}
    For \(0<s\leq t\leq2\), naturality of the equivalence with the neighborhood filtration and the
    radial deformation retractions imply that the structure map
    \(\VR_s(\square^n)\to\VR_t(\square^n)\) is a homotopy
    equivalence.
\end{example}

\begin{example} \label{ex:inf_wedge_product}
    Let $n_1, \dots, n_k$ be positive integers.
    By \cref{ex:cube-main} and \cref{cor:VR-product-wedge-main}, we have
    \begin{equation}
        \mathrm{VR}_{2r}\left(\square^{n_1}\vee \dots \vee \square^{n_k}\right)
        \simeq
        \begin{cases}
        \bigvee_{i=1}^k \bbS^{n_i}, & 0<r\le 1,\\[6pt]
        \ast, & r>1,
        \end{cases}
    \end{equation}
    and
    \begin{equation}
        \mathrm{VR}_{2r}\left(\square^{n_1}\times \dots \times \square^{n_k}\right)
        \simeq
        \begin{cases}
        \prod_{i=1}^k \bbS^{n_i}, & 0<r\le 1,\\[6pt]
        \ast, & r>1.
        \end{cases}
    \end{equation}
    % In both cases, all structure maps with \(0<s\leq t\leq2\) are homotopy
    % equivalences. %by \cref{cor:VR-product-wedge-main}.
\end{example}

\begin{example}\label{ex:tc-zcl-filtrations-main}
Use corners as basepoints of the cube boundaries and the
corresponding pairs of corners for their products.
By \cref{ex:cube-main,cor:VR-product-wedge-main},
\begin{equation}\label{eq:tc-zcl-separating-pair-main}
\begin{aligned}
\VR_t\bigl((\square^1\times\square^1)\vee
                   (\square^1\times\square^1)\bigr)
&\simeq\begin{cases}
\bbT^2\vee\bbT^2, & 0<t\leq2,\\[2mm]
\ast, & t>2,
\end{cases}\\[4mm]
\VR_t\bigl((\square^1\times\square^1)\vee
                   \square^1\vee\square^2\vee\square^1\bigr)
&\simeq\begin{cases}
\bbT^2\vee\mathbb{M}, & 0<t\leq2,\\[2mm]
\ast, & t>2.
\end{cases}
\end{aligned}
\end{equation}
Here \(\mathbb{M}=\bbS^1\vee\bbS^2\vee\bbS^1\) is the
Mickey Mouse space.
% In both cases, all structure maps with \(0<s\leq t\leq2\)
% are homotopy equivalences. %by \cref{cor:VR-product-wedge-main}.
\end{example}

Let \(n\geq2\). The group \(\{\pm1\}\) acts on
\(\bbR^{n+1}\setminus\{0\}\) by the \emph{antipodal action}
\(x\mapsto-x\), with orbits \([x]=\{x,-x\}\).
For \(X=-X\subset\bbR^{n+1}\setminus\{0\}\), we equip
\(X/(\pm1)\) with the antipodal quotient of the
\(\ell_\infty\) metric, as defined above.
Each line through the origin meets \(\square^n\) in a unique
antipodal pair, giving the identifications
\[
\rp^n=(\bbR^{n+1}\setminus\{0\})/\bbR^\times
\cong\bbS^n/(\pm1)\cong\square^n/(\pm1).
\]
We denote the resulting metric on \(\rp^n\) by \(d_\square\).
Both \(\square^n\) and \((\rp^n,d_\square)\) have diameter \(2\):
the classes of \((1,1,0,\ldots,0)\) and
\((1,-1,0,\ldots,0)\) remain at distance \(2\).

\begin{proposition}
\label{prop:VR-rpn-linfty-main}
For \(n\geq2\),
\[
\VR_t(\rp^n,d_\square)
\simeq
\begin{cases}
\rp^n, & 0<t\leq\tfrac23,\\[2mm]
\ast, & t>1.
\end{cases}
\]
% For \(0<s\leq t\leq\tfrac23\), the structure map
% \(
% \VR_s(\rp^n,d_\square)
% \longrightarrow
% \VR_t(\rp^n,d_\square)
% \)
% is a homotopy equivalence.
\end{proposition}

\begin{remark}
For a proper isometric action of a group \(G\) on a metric space \(X\),
\cite[Proposition~3.5]{adams2022metric} asserted that
\(\VR_t(X/G)\cong\VR_t(X)/G\) under its \(t\)-diameter hypothesis.  As
reported in \cite[Sections~4.4.2--4.4.3]{medina2025persistent}, Liam
Barham observed through personal communication that this hypothesis is
insufficient and proposed the stronger requirement that every simplex
of \(\VR_t(X/G)\) have a unique lift to \(\VR_t(X)\) after one
representative has been fixed.  Barham discussed this stronger condition
in \cite{barham2025group}.  For the purpose of giving a self-contained
proof, we verify the required lifting property
directly, without introducing the terminology of \(t\)-diameter actions.
\end{remark}

\begin{proof}[Proof of \cref{prop:VR-rpn-linfty-main}]
Suppose first that \(t>1\), and let \(e_1\) be the first standard
basis vector of \(\bbR^{n+1}\). For every \([x]\in\rp^n\), choosing
a representative \(x\in\square^n\) with \(x_1\geq0\) gives
\(d_\square([x],[e_1])\leq d_\infty(x,e_1)\leq1<t\).
Thus adjoining \([e_1]\) to any simplex of
\(\VR_t(\rp^n,d_\square)\) again gives a simplex, so the complex
is a cone with apex \([e_1]\) and hence contractible.

Now fix \(0<t\leq\tfrac23\). We will use the nerve lemma and
radial retraction to establish the homotopy equivalences
\begin{equation}\label{eq:rpn-linfty-homotopy-chain}
\VR_t(\rp^n,d_\square)
\cong \VR_t(\square^n)/(\pm1)
\simeq B_{t/2}(\square^n;\blacksquare^{n+1})/(\pm1)
\simeq \rp^n.
\end{equation}
For the third equivalence in \cref{eq:rpn-linfty-homotopy-chain},
note that every \(z\in B_{t/2}(\square^n;\blacksquare^{n+1})\)
satisfies \(\|z\|_\infty>1-t/2\geq\tfrac23>0\), so this
neighborhood excludes the origin.
The radial deformation retraction \(z\mapsto z/\|z\|_\infty\)
therefore induces the equivalence.

For the first homeomorphism in \cref{eq:rpn-linfty-homotopy-chain},
we choose representatives of each Rips simplex consistently. Let
\(\{[x_0],\ldots,[x_k]\}\) be a simplex of
\(\VR_t(\rp^n,d_\square)\), and fix a representative \(x_0\).
Set \(y_0=x_0\), and choose \(y_i\in\{x_i,-x_i\}\) such that
\(d_\infty(x_0,y_i)<t\). If \(d_\infty(y_i,y_j)\geq t\), the
quotient distance implies \(d_\infty(y_i,-y_j)<t\), so
\[
2=d_\infty(x_0,-x_0)
\leq d_\infty(x_0,y_i)+d_\infty(y_i,-y_j)
   +d_\infty(-y_j,-x_0)
<3t\leq2,
\]
a contradiction.
Thus \(d_\infty(y_i,y_j)<t\) for all \(i,j\), so the
\(y_i\) span a simplex of \(\VR_t(\square^n)\). The choices are unique,
since both \(y_i\) and \(-y_i\) lying within distance \(t\) of \(x_0\)
would give \(2=d_\infty(y_i,-y_i)<2t\leq\tfrac43\).
Thus each simplex of \(\VR_t(\rp^n,d_\square)\) has exactly two
antipodal lifts to \(\VR_t(\square^n)\). This correspondence respects
faces and gives the first homeomorphism.

For the second equivalence in \cref{eq:rpn-linfty-homotopy-chain}, let
\(q\colon \blacksquare^{n+1}\to \blacksquare^{n+1}/(\pm1)\)
be the quotient map. The sets
\(q(B_{t/2}(x;\blacksquare^{n+1}))\), indexed by \([x]\in\rp^n\),
are independent of the chosen representatives and form an open cover
of \(B_{t/2}(\square^n;\blacksquare^{n+1})/(\pm1)\).
We will show that its nerve is \(\VR_t(\rp^n,d_\square)\) and that
every nonempty finite intersection of its members is contractible.

We now identify the nerve of this cover. A point in
\(\bigcap_i q(B_{t/2}(x_i;\blacksquare^{n+1}))\) has a representative within
distance \(t/2\) of suitable representatives of all the \([x_i]\).
These representatives have pairwise distances less than \(t\), so the
classes span a Rips simplex. Conversely, lift a Rips simplex to
\(\{y_0,\ldots,y_k\}\). The point \(z\in \blacksquare^{n+1}\) with coordinates
\[
z_j:=\tfrac12\left(\min_i(y_i)_j+\max_i(y_i)_j\right),
\qquad 1\leq j\leq n+1,
\]
satisfies \(d_\infty(z,y_i)<t/2\) for every \(i\), so
\(q(z)\in\bigcap_i q(B_{t/2}(x_i;\blacksquare^{n+1}))\). This proves that the nerve is
\(\VR_t(\rp^n,d_\square)\).

To prove contractibility, fix a nonempty finite intersection
\(\bigcap_{i=0}^k q(B_{t/2}(x_i;\blacksquare^{n+1}))\) and lift its corresponding simplex
to \(\{y_0,\ldots,y_k\}\). Set
\(A:=\bigcap_{i=0}^k B_{t/2}(y_i;\blacksquare^{n+1})\),
which is nonempty by the coordinate midpoint construction.
Any point in \(q^{-1}(\bigcap_i q(B_{t/2}(x_i;\blacksquare^{n+1})))\) similarly determines a
lift of the simplex. Uniqueness after fixing one representative
implies that this lift is either \(\{y_i\}\) or \(\{-y_i\}\). Therefore
\[
q^{-1}\left(\bigcap_i q(B_{t/2}(x_i;\blacksquare^{n+1}))\right)=A\sqcup(-A).
\]
The union is disjoint because \(d_\infty(y_0,-y_0)=2>t\).
The open quotient map \(q\) restricts to a homeomorphism from \(A\) onto
\(\bigcap_i q(B_{t/2}(x_i;\blacksquare^{n+1}))\). The latter is therefore contractible, as
\(A\) is convex. Thus the cover is good.
The quotient \(B_{t/2}(\square^n;\blacksquare^{n+1})/(\pm1)\)
is metrizable, hence paracompact, so the functorial nerve lemma
\cite[Theorem~4.2]{lim2024vietoris}, together with the first homeomorphism,
gives the second equivalence in \cref{eq:rpn-linfty-homotopy-chain}.

\smallskip

% It remains to verify the assertion about structure maps.
We verify the structure-map assertion in \cref{subsec:vr-models-main}
for this filtration. For
\(0<s\leq t\leq\tfrac23\), the radial deformation retractions commute
with the inclusion
\(B_{s/2}(\square^n;\blacksquare^{n+1})\hookrightarrow B_{t/2}(\square^n;\blacksquare^{n+1})\),
so the induced quotient inclusion is a homotopy equivalence. Since
\(q(B_{s/2}(x;\blacksquare^{n+1}))\subseteq q(B_{t/2}(x;\blacksquare^{n+1}))\), naturality
in the functorial nerve lemma identifies this inclusion, up to homotopy,
with the Rips structure map, which is therefore a homotopy equivalence.
\end{proof}

The homotopy type of \(\VR_t(\rp^n,d_\square)\) is not completely
understood for \(\tfrac23<t\leq1\).  Applying
\cref{lem:capped-metric-main} with \(a=2/3\) gives a simpler filtration.

\begin{example}
\label{ex:rpn-linfty-main}
Equip \(\rp^n\) with the truncated quotient metric
\(d_\square^{(2/3)}\). Equivalently, this is the quotient metric
induced by \(d_\infty^{(2/3)}\) on \(\square^n\).  By
\cref{lem:capped-metric-main}, it induces the same topology as
\(d_\square\); hence \((\rp^n,d_\square^{(2/3)})\) is a compact
metric model of \(\rp^n\).  By \cref{prop:VR-rpn-linfty-main} and
\cref{lem:capped-metric-main},
\[
\VR_t(\rp^n,d_\square^{(2/3)})
\simeq
\begin{cases}
\rp^n, & 0<t\leq\tfrac23,\\[2mm]
\ast, & t>\tfrac23.
\end{cases}
\]
% The structure maps are homotopy equivalences for
% \(0<s\leq t\leq\tfrac23\).
\end{example}

\subsubsection{Geodesic models.}
\label{subsubsec:geodesic-main}

Next we summarize the known Vietoris--Rips homotopy types for round
geodesic spheres, their wedges and products, and real projective spaces
with the quotient metric.

These examples also underscore the simplifying role of the \(\ell_\infty\)
metrics used in \cref{subsubsec:l_infty-main}.  For a cube boundary, the
Vietoris--Rips filtration retains the homotopy type of the sphere up to a
single critical scale and is contractible thereafter.  In contrast, the
round geodesic spheres can produce substantially richer scale-dependent
homotopy behavior (see \cref{ex:sphere-geodesic-main}): even for \(\bbS^1\), the filtration passes through
infinitely many distinct homotopy types before becoming contractible.

\begin{example}
\label{ex:sphere-geodesic-main}
Let \(m\geq1\), and equip the unit sphere \(\bbS^m\) with its round
geodesic metric. Set
\[
\zeta_m:=\arccos\left(-\tfrac{1}{m+1}\right).
\]
By \cite[Theorem~7.1]{lim2024vietoris} and \cite[Section~1]{adams2025connectivity},
\[
\VR_t(\bbS^m)
\simeq
\begin{cases}
\bbS^m, & 0<t\leq\zeta_m,\\[2mm]
\ast, & t\geq\pi.
\end{cases}
\]
% where all structure maps with
% \(0<s\leq t\leq\zeta_m\) are homotopy equivalences.
We will not need the homotopy type in the intermediate range
\(\zeta_m<t<\pi\).  Additional results are known when \(m=1,2\);
the circle case is recalled separately below, while for \(m\geq3\)
the homotopy type after \(\zeta_m\) is not currently known in general
\cite[Section~7.1]{lim2024vietoris}.

When $m=1$, \cite[Theorem~7.4]{adamaszek2017vietoris} gives the complete homotopy type of the Vietoris--Rips filtration of \(\bbS^1\):
\[
\VR_t(\bbS^1)
\simeq
\begin{cases}
\bbS^{2\ell+1}, &
\displaystyle\tfrac{2\pi\ell}{2\ell+1}<t\leq
\tfrac{2\pi(\ell+1)}{2\ell+3},\quad \ell\geq0,\\[3mm]
\ast, & t\geq\pi.
\end{cases}
\]
% The structure maps
% within each band are homotopy equivalences.
\end{example}

\begin{example}
\label{ex:torus-geodesic-main}
Let the torus
\(
\bbT^2=\bbS^1\times\bbS^1
\)
carry the maximum product metric \(d_{\max}\) of the unit round geodesic metrics.  By
\cite[Proposition~10.2]{adamaszek2017vietoris} and
\cref{ex:sphere-geodesic-main},
\[
\VR_t(\bbT^2)
\simeq
\begin{cases}
\bbS^{2\ell+1}\times\bbS^{2\ell+1}, &
\displaystyle\tfrac{2\pi\ell}{2\ell+1}<t\leq
\tfrac{2\pi(\ell+1)}{2\ell+3},\quad \ell\geq0,\\[3mm]
\ast, & t\geq\pi.
\end{cases}
\]
% The structure maps
% within each band are homotopy equivalences.
\end{example}

\begin{example}
\label{ex:wedge-geodesic-main}
Equip \(\wedgeS=\bigvee_{i=1}^n\bbS^i\) with the gluing metric on unit
round spheres.  By \cite[Theorem~7.1]{lim2024vietoris}, \cite[Section~1]{adams2025connectivity}, and the natural
metric-gluing equivalence of
\cite[Proposition~1]{adamaszek2020homotopy}, together with the monotonicity
of \(\zeta_i\),
\[
\VR_t(\wedgeS)
\simeq
\begin{cases}
\wedgeS, & 0<t\leq\zeta_n,\\[2mm]
\ast, & t\geq\pi.
\end{cases}
\]
% All structure maps with \(0<s\leq t\leq\zeta_n\) are homotopy equivalences.
\end{example}

\begin{example}
\label{prop:VR-rpn-geodesic-main}
Rescale the unit round geodesic metric on \(\bbS^n\) by a factor of
\(2\), and equip \(\rp^n=\bbS^n/(\pm1)\) with the resulting
antipodal quotient metric, denoted \(d_{\rp}\).  
By \cite{adams2022metric},
\[
\VR_t(\rp^n,d_{\rp})
\simeq
\begin{cases}
\rp^n, & 0<t\leq\dfrac{2\pi}{3},\\[2mm]
\ast, & t>\pi.
\end{cases}
\]
% and all structure maps with
% \(0<s\leq t\leq2\pi/3\) are homotopy equivalences.
\end{example}

\begin{example}
\label{ex:capped-geodesic-main}
Let \(n\geq2\). Using the quotient metric \(d_{\rp}\) from
\cref{prop:VR-rpn-geodesic-main} and the gluing metric \(d_\vee\)
from \cref{ex:wedge-geodesic-main}, consider the truncated and rescaled metric spaces
\[
X=\bigl(\rp^n,d_{\rp}^{(2\pi/3)}\bigr),
\qquad
Y=\left(\wedgeS,
\tfrac{2\pi}{3\zeta_n}d_\vee^{(\zeta_n)}\right).
\]
Thus the quotient metric is truncated at \(2\pi/3\), while the wedge
metric is first truncated at \(\zeta_n\) and then multiplied by
\(2\pi/(3\zeta_n)\).
Combining \cref{prop:VR-rpn-geodesic-main,ex:wedge-geodesic-main}  with
\cref{lem:capped-metric-main}, we obtain
\[
\VR_t(X)\simeq
\begin{cases}
\rp^n, & 0<t\leq\tfrac{2\pi}{3},\\[2mm]
\ast, & t>\tfrac{2\pi}{3},
\end{cases}
\qquad
\VR_t(Y)\simeq
\begin{cases}
\wedgeS, & 0<t\leq \tfrac{2\pi}{3},\\[2mm]
\ast, & t>\tfrac{2\pi}{3}.
\end{cases}
\]
% For both filtrations, all structure maps with
% \(0<s\leq t\leq2\pi/3\) are homotopy equivalences.
% The gluing metric is formed before truncation and rescaling.
\end{example}

\subsubsection{Summary tables of Vietoris--Rips filtrations.}

\Cref{table:vr-linfty-main,table:vr-geodesic-main}
summarize the homotopy types and scale ranges established in
\cref{subsubsec:l_infty-main,subsubsec:geodesic-main}.
Within each listed range of constant homotopy type, all structure
maps are homotopy equivalences.

\begin{table}[H]
    \caption{Vietoris--Rips homotopy types for the \(\ell_\infty\) models.
    The cube-boundary, product, and wedge entries follow
    from \cref{ex:cube-main,ex:inf_wedge_product,ex:tc-zcl-filtrations-main};
    the projective-space entries follow from
    \cref{prop:VR-rpn-linfty-main,ex:rpn-linfty-main}.}
    \label{table:vr-linfty-main}
    \centering
    \small
    \renewcommand{\arraystretch}{1.35}
    \setlength{\tabcolsep}{5pt}
    \begin{tabular}{|>{\raggedright\arraybackslash}m{0.24\textwidth}|>{\centering\arraybackslash}m{0.17\textwidth}|>{\centering\arraybackslash}m{0.15\textwidth}|>{\raggedright\arraybackslash}m{0.32\textwidth}|}
    \hline
    Metric space \(X\) & Initial type & Initial scale range
        & Behavior at larger scales \\
    \hline
    \(\square^m\) & \(\bbS^m\) & \(0<t\leq2\)
        & contractible for \(t>2\) \\
    \hline
    \(\bigvee_{i=1}^k\square^{n_i}\) & \(\bigvee_{i=1}^k\bbS^{n_i}\) & \(0<t\leq2\)
        & contractible for \(t>2\) \\
    \hline
    \(\prod_{i=1}^k\square^{n_i}\) & \(\prod_{i=1}^k\bbS^{n_i}\) & \(0<t\leq2\)
        & contractible for \(t>2\) \\
    \hline
    \(\begin{gathered}\bigl((\square^1\times\square^1)\vee\\
        (\square^1\times\square^1),d_\vee\bigr)\end{gathered}\)
        & \(\bbT^2\vee\bbT^2\) & \(0<t\leq2\)
        & contractible for \(t>2\) \\
    \hline
    \(\begin{gathered}\bigl((\square^1\times\square^1)\vee\\
        \square^1\vee\square^2\vee\square^1,d_\vee\bigr)\end{gathered}\)
        & \(\bbT^2\vee\mathbb{M}\) & \(0<t\leq2\)
        & contractible for \(t>2\) \\
    \hline
    \((\rp^n,d_\square)\) & \(\rp^n\) & \(0<t\leq2/3\)
        & undetermined for \(2/3<t\leq1\)\newline contractible for \(t>1\) \\
    \hline
    \((\rp^n,d_\square^{(2/3)})\) & \(\rp^n\) & \(0<t\leq2/3\)
        & contractible for \(t>2/3\) \\
    \hline
    \end{tabular}
\end{table}

\begin{table}[!htbp]
    \caption{Vietoris--Rips homotopy types for the geodesic models and their
    truncations, where \(\zeta_m=\arccos(-1/(m+1))\).
    Rows~1--4 follow from
    \cref{ex:sphere-geodesic-main,ex:torus-geodesic-main,ex:wedge-geodesic-main,prop:VR-rpn-geodesic-main},
    respectively, and rows~5--6 from \cref{ex:capped-geodesic-main}.}
    \label{table:vr-geodesic-main}
    \centering
    \small
    \renewcommand{\arraystretch}{1.35}
    \setlength{\tabcolsep}{5pt}
    \begin{tabular}{|>{\raggedright\arraybackslash}m{0.185\textwidth}|>{\centering\arraybackslash}m{0.14\textwidth}|>{\centering\arraybackslash}m{0.22\textwidth}|>{\raggedright\arraybackslash}m{0.335\textwidth}|}
    \hline
    Metric space \(X\) & Initial type & Initial scale range
        & Behavior at larger scales \\
    \hline
    \(\bbS^m\) & \(\bbS^m\) & \(0<t\leq\zeta_m\)
        & see \cref{ex:sphere-geodesic-main} for \(\zeta_m<t<\pi\)\newline contractible for \(t\geq\pi\) \\
    \hline
    \(\bbT^2\) & \(\bbT^2\) & \(0<t\leq2\pi/3\)
        & see \cref{ex:torus-geodesic-main} for \(2\pi/3<t<\pi\)\newline contractible for \(t\geq\pi\) \\
    \hline
    \((\wedgeS,d_\vee)\) & \(\wedgeS\) & \(0<t\leq\zeta_n\)
        & unspecified for \(\zeta_n<t<\pi\)\newline contractible for \(t\geq\pi\) \\
    \hline
    \((\rp^n,d_{\rp})\) & \(\rp^n\) & \(0<t\leq2\pi/3\)
        & unspecified for \(2\pi/3<t\leq\pi\)\newline contractible for \(t>\pi\) \\
    \hline
    \((\rp^n,d_{\rp}^{(2\pi/3)})\) & \(\rp^n\) & \(0<t\leq2\pi/3\)
        & contractible for \(t>2\pi/3\) \\
    \hline
    \(\left(\wedgeS,\tfrac{2\pi}{3\zeta_n}d_\vee^{(\zeta_n)}\right)\) & \(\wedgeS\) & \(0<t\leq2\pi/3\)
        & contractible for \(t>2\pi/3\) \\
    \hline
    \end{tabular}
\end{table}

\subsection{Examples of Persistent Invariants}
\label{subsec:persistent-computations-main}

We combine the classical invariant values in
\cref{table:invariant-values-main} with the Vietoris--Rips computations
in \cref{subsec:vr-models-main} to obtain the following examples.
Throughout, let \(J=[s,t]\) with \(0<s\leq t\), and let
\(\invariant=\TC,\zcup,\cat\), or \(\cl\). We fix a coefficient field
\(\field\) for the cohomological invariants and use \(\field=\bbF_2\)
for projective spaces.
All four invariants vanish on intervals with nonpositive left endpoint,
since the source Rips complex is empty.

On intervals contained in a displayed range of constant homotopy
type in \cref{subsec:vr-models-main}, the persistent invariants equal
their classical values on that homotopy type, because the structure maps
are homotopy equivalences. All four invariants vanish when the target space is contractible.

% For general products and wedges of spheres, we also use
% the following classical values.
% For a product of \(j\geq1\) positive-dimensional spheres,
% \(\cat=\cl=j\), while \(\TC\) equals \(j\) plus the number of
% even-dimensional factors. The same formula holds for \(\zcup\) when
% \(\operatorname{char}(\field)\neq2\); in characteristic \(2\),
% \(\zcup=j\).
% For a wedge of at least two positive-dimensional spheres,
% \(\TC=\zcup=2\) and \(\cat=\cl=1\).
% The \(\TC\)- and \(\cat\)-values follow from
% \cref{prop:tc of a space,prop:properties of cat,ex:tc of bouquets}.
% The cohomological values follow from \cref{ex:zcup of Sn}, the
% tensor-product cohomology ring for products, and classes supported on
% distinct summands for wedges, as in the discussion following
% \cref{table:invariant-values-main}.

\subsubsection{\(\ell_\infty\) models.}

\begin{example}[Cube boundaries, products, and wedges]
\label{prop:linfty-persistent-values-main}
Let \(n_1,\ldots,n_k\) be positive integers, and equip the products and
wedges with the metrics of \cref{ex:inf_wedge_product}. Then
\[
\invariant\left(
\VR_\bullet\Big(\prod_{i=1}^k\square^{n_i},d_{\max}\Big)
\right)(J)=
\begin{cases}
\invariant\left(\prod_{i=1}^k\bbS^{n_i}\right), & J\subset(0,2],\\
0, & \text{otherwise},
\end{cases}
\]
and
\[
\invariant\left(
\VR_\bullet\Big(\bigvee_{i=1}^k\square^{n_i},d_\vee\Big)
\right)(J)=
\begin{cases}
\invariant\left(\bigvee_{i=1}^k\bbS^{n_i}\right), & J\subset(0,2],\\
0, & \text{otherwise}.
\end{cases}
\]
% Indeed, \cref{ex:cube-main,ex:inf_wedge_product} identify the homotopy
% types and show that all structure maps with \(0<s\leq t\leq2\) are
% homotopy equivalences; for \(t>2\), their targets are contractible.
The formulas follow from \cref{ex:cube-main,ex:inf_wedge_product}.

The torus and Mickey Mouse space with the \(\ell_\infty\)
metrics are special cases of the product and wedge constructions,
respectively.
By \cref{table:invariant-values-main}, their persistent
\(\cat\) and \(\cl\) equal \(2\) and \(1\), respectively,
on \(J\subset(0,2]\); see
\cref{fig:cubic-comparison-intro}(a)--(b).
Both have persistent \(\TC=\zcup=2\) on intervals \(J\subset(0,2]\).
\end{example}

\begin{example}\label{ex:tc-zcl-persistent-main}
The filtration descriptions in \cref{ex:tc-zcl-filtrations-main}
and the classical values in
\cref{ex:tc-wedges-tori} give, over every field and for every interval \(J\),
\[
\begin{aligned}
\invariant\left(\VR_\bullet\bigl(
(\square^1\times\square^1)\vee(\square^1\times\square^1),d_\vee
\bigr)\right)(J)
&=\begin{cases}4,&J\subset(0,2],\\0,&\text{otherwise},\end{cases}\\
\invariant\left(\VR_\bullet\bigl(
(\square^1\times\square^1)\vee\square^1\vee\square^2\vee\square^1,d_\vee
\bigr)\right)(J)
&=\begin{cases}3,&J\subset(0,2],\\0,&\text{otherwise},\end{cases}
\end{aligned}
\]
where \(\invariant\) is either \(\TC\) or \(\zcup\).

If \(\invariant\) is \(\cat\) or \(\cl\), then both spaces have the same persistent invariant that is \(2\) on \(J\subset(0,2]\) and \(0\) otherwise.
Both persistent homology modules have four copies of \((0,2]\)
in degree \(1\), two copies in degree \(2\), and no bars in higher degrees.
In degree \(0\), both persistent homology modules have the single interval \((0,\infty)\).
Thus their entire persistent homology modules agree as well.
The comparison is illustrated in \cref{fig:cubic-comparison-intro}(c)--(d).
\end{example}

\begin{example}[Projective space and truncation]
Let \(n\geq2\). For the \(\ell_\infty\) quotient metric in
\cref{prop:VR-rpn-linfty-main},
\[
\invariant\left(\VR_\bullet(\rp^n,d_\square)\right)(J)
=\begin{cases}
\invariant(\rp^n), & J\subset(0,\tfrac23],\\
0, & J\not\subset(0,1],\\
\text{undetermined}, & \text{otherwise}.
\end{cases}
\]
For the truncated metric in \cref{ex:rpn-linfty-main},
\[
\invariant\left(\VR_\bullet(\rp^n,d_\square^{(2/3)})\right)(J)
=
\begin{cases}
\invariant(\rp^n), & J\subset(0,\tfrac23],\\
0, & \text{otherwise}.
\end{cases}
\]
% Both computations follow from the structure maps being homotopy equivalences in the cited results and contractibility above the stated scales.
The classical invariant values of \(\rp^n\) can be read from \cref{table:invariant-values-main}.
% For the rescaled metric \(3d_\square^{(2/3)}\), the cutoff is
% \(2\) instead of \(2/3\).
\end{example}

\subsubsection{Geodesic models.}

\begin{example}[Spheres and wedge sums]
\label{ex:geodesic-sphere-invariants-main}
Let \(m\geq1\), and equip \(\bbS^m\) with the unit round geodesic
metric. The scale ranges in \cref{ex:sphere-geodesic-main} give
\[
\invariant\left(\VR_\bullet(\bbS^m)\right)(J)=
\begin{cases}
\invariant(\bbS^m), & J\subset(0,\zeta_m],\\
0, & J\not\subset(0,\pi),\\
\text{undetermined}, & \text{otherwise}.
\end{cases}
\]
An undetermined case denotes an interval not covered by this computation.

For \(k\geq2\) and positive integers \(n_1\leq\cdots\leq n_k\),
\cref{cor:VR-product-wedge-main} gives
\[
\invariant\left(\VR_\bullet\Big(\bigvee_{i=1}^k\bbS^{n_i}\Big)\right)(J)=
\begin{cases}
\invariant\left(\bigvee_{i=1}^k\bbS^{n_i}\right), & J\subset(0,\zeta_{n_k}],\\
0, & J\not\subset(0,\pi),\\
\text{undetermined}, & \text{otherwise}.
\end{cases}
\]
% The Mickey Mouse space \(\mathbb{M}\) and
% \(\wedgeS=\bigvee_{i=1}^n\bbS^i\) are special cases.
\end{example}

\begin{example}[Circle and torus]
\label{ex:geodesic-circle-torus-invariants-main}
For the unit round circle, the complete filtration description in
\cref{ex:sphere-geodesic-main} gives
\[
\invariant\left(\VR_\bullet(\bbS^1)\right)(J)=
\begin{cases}
1, & J\subset\left(\tfrac{2\pi\ell}{2\ell+1},
\tfrac{2\pi(\ell+1)}{2\ell+3}\right]\text{ for some }\ell\geq0,\\[2mm]
0, & \text{otherwise}.
\end{cases}
\]
Here, if \(J\subset(0,\pi)\) but
\(J\not\subset\left(\tfrac{2\pi\ell}{2\ell+1},
\tfrac{2\pi(\ell+1)}{2\ell+3}\right]\) for every \(\ell\geq0\),
then the source and target are homotopy equivalent to
a lower-dimensional sphere and a higher-dimensional sphere,
respectively. The structure map is therefore nullhomotopic,
so all four persistent invariants equal \(0\).

For the torus in \cref{ex:torus-geodesic-main}, the classical values in
\cref{prop:tc of a space,prop:properties of cat,ex:zcup of Sn} give
\[
\invariant\left(\VR_\bullet(\bbT^2)\right)(J)=
\begin{cases}
2, & J\subset\left(\tfrac{2\pi\ell}{2\ell+1},
\tfrac{2\pi(\ell+1)}{2\ell+3}\right]\text{ for some }\ell\geq0,\\[2mm]
0, & \text{otherwise}.
\end{cases}
\]
Here, if \(J\subset(0,\pi)\) but
\(J\not\subset\left(\tfrac{2\pi\ell}{2\ell+1},
\tfrac{2\pi(\ell+1)}{2\ell+3}\right]\) for every \(\ell\geq0\),
then \cref{cor:VR-product-wedge-main} identifies
\(\VR_s(\bbT^2)\to\VR_t(\bbT^2)\), up to homotopy, with the
product of two copies of the nullhomotopic map
\(\VR_s(\bbS^1)\to\VR_t(\bbS^1)\).
Hence all four persistent invariants equal \(0\).
\end{example}

\begin{example}[Projective space and truncation]
\label{lem:TC-zcup-RPn-wedge-main}
Let \(n\geq2\). For the geodesic quotient metric in
\cref{prop:VR-rpn-geodesic-main},
\[
\invariant\left(\VR_\bullet(\rp^n,d_{\rp})\right)(J)
=\begin{cases}
\invariant(\rp^n), & J\subset(0,\tfrac{2\pi}{3}],\\
0, & J\not\subset(0,\pi],\\
\text{undetermined}, & \text{otherwise}.
\end{cases}
\]
The persistent invariant of \((\wedgeS,d_\vee)\) is obtained by setting
\(n_i=i\) and \(k=n\) in \cref{ex:geodesic-sphere-invariants-main}.
\cref{fig:tc_zcl} illustrates the persistent \(\TC\) and \(\zcup\)
of \((\rp^n,d_{\rp})\) and \((\wedgeS,d_\vee)\), including the intervals
on which the values remain undetermined.

\label{ex:truncated-geodesic-invariants-main}
For the metrics defined in
\cref{ex:capped-geodesic-main}, truncation of the quotient metric
and truncation and rescaling of the gluing metric give
\begin{equation}
\label{eq:capped-geodesic-profiles-main}
\begin{aligned}
\invariant\left(\VR_\bullet(\rp^n,d_{\rp}^{(2\pi/3)})\right)(J)
&=\begin{cases}
\invariant(\rp^n), & J\subset(0,2\pi/3],\\
0, & \text{otherwise},
\end{cases}\\[2mm]
\invariant\left(\VR_\bullet(\wedgeS,\tfrac{2\pi}{3\zeta_n}d_\vee^{(\zeta_n)})\right)(J)
&=\begin{cases}
\invariant(\wedgeS), & J\subset(0,2\pi/3],\\
0, & \text{otherwise}.
\end{cases}
\end{aligned}
\end{equation}
% These computations follow from the structure maps being homotopy equivalences
% in \cref{prop:VR-rpn-geodesic-main,ex:capped-geodesic-main}
% and contractibility above the stated scales.
% The formulas follow from \cref{prop:VR-rpn-geodesic-main,ex:capped-geodesic-main}. The classical values are
% given in \cref{table:invariant-values-main}.

For the truncated metrics, \cref{ex:capped-geodesic-main} also shows that
the positive-degree \(\bbF_2\)-persistent homology modules agree:
in each degree \(1,\ldots,n\), the barcode consists of one interval
\((0,2\pi/3]\).
\end{example}

\subsection{Distinguishing Power of Persistent Invariants}
\label{subsec:distinguishing-main}

We compare the distinguishing power of persistent
\(\TC\), \(\zcup\), \(\cat\), \(\cl\), and positive-degree persistent
homology through the Gromov--Hausdorff lower bounds they provide.
For metric spaces \(X,Y\) from \cref{subsec:vr-models-main},
the stability estimates in \cref{thm:stability,thm:vr-ph-stability}
give
\begin{align*}
\de\left(\invariant(\VR_\bullet(X)),\invariant(\VR_\bullet(Y))\right)
&\leq2\dgh(X,Y),\\
\di\left(\rmH_k(\VR_\bullet(X)),\rmH_k(\VR_\bullet(Y))\right)
&\leq2\dgh(X,Y),
\end{align*}
where \(\invariant\in\{\TC,\zcup,\cat,\cl\}\) and \(k\geq1\).
We follow the coefficient choices of
\cref{subsec:persistent-computations-main}: homology and cohomology
are taken over a fixed field \(\field\), with \(\field=\bbF_2\)
for comparisons involving projective spaces.

The comparisons are organized according to the information available
about the persistent invariants.
In \cref{subsubsec:explicit-comparisons-main}, explicit formulas
for the persistent invariants allow us to compute their erosion
distances exactly.
We first treat cube boundaries and their products and wedges, then
projective spaces and wedges of spheres with truncated metrics.

In \cref{subsubsec:partial-comparisons-main}, we treat the projective-space
and wedge comparison with untruncated metrics arising from cube boundaries
and geodesic spheres.
Partial information about the persistent invariants suffices to show
that, for suitable rescalings of the projective-space metric, each of the persistent invariants
\(\TC\), \(\zcup\), \(\cat\), and \(\cl\) gives a
Gromov--Hausdorff lower bound stronger than that obtained from
persistent homology.

% In \cref{subsubsec:partial-comparisons-main}, we treat the
% projective-space and wedge comparison with untruncated metrics
% arising from cube boundaries and geodesic spheres, using partial
% information about their persistent invariants.

\subsubsection{Comparisons from explicit persistent invariants}
\label{subsubsec:explicit-comparisons-main}
\label{ex:tc-zcl-cat-cl-main}
\label{subsubsec:truncated-comparisons-main}

\Cref{table:persistent-distinguishing-main} gives the exact distances
between persistent invariants for metric realizations of the six
classical comparisons in \cref{subsec:classical-comparisons-main}.
For each pair of metric spaces \(X,Y\), we use the abbreviations
\[
\begin{aligned}
\de(\invariant)&:=
\de\left(\invariant(\VR_\bullet(X)),\invariant(\VR_\bullet(Y))\right),\\
\di(\rmH_*)&:=\max_{k\geq1}
\di\left(\rmH_k(\VR_\bullet(X)),\rmH_k(\VR_\bullet(Y))\right).
\end{aligned}
\]
We also include a column for \(|\diam(X)-\diam(Y)|\), twice
the diameter lower bound for \(\dgh(X,Y)\), as a point of reference.
But the emphasis remains on
comparing the distinguishing power of the persistent invariants,
rather than finding the best possible Gromov--Hausdorff lower
bound for each pair.

\begin{table}[!htbp]
    \caption{
        Metric realizations of the six pairs in \cref{subsec:classical-comparisons-main}, and exact distances between their
    persistent invariants.
    Rows~\ref{comp:s1-s2-main}--\ref{comp:wedge121-torus-main} use cube boundaries with
    the \(\ell_\infty\) metric, as in
    \cref{ex:cube-main,ex:inf_wedge_product}.
    Comparison \labelcref{comp:tc-zcl-separation-main} uses the
    corner basepoints specified in \cref{ex:tc-zcl-filtrations-main}.
    The \(\zcup\)-entries in rows~\ref{comp:s1-s2-main}
    and~\ref{comp:s2-wedge11-main} assume a
    coefficient field of characteristic \(2\).
    For \(n\geq2\), the two rows for comparison \ref{comp:rpn-wedge-main}
    give the \(\ell_\infty\) and geodesic realizations, respectively,
    with coefficients in \(\bbF_2\).
    All other entries hold over
    any field.
    Blank entries denote zero. Half of each numerical entry gives a lower bound for the
    Gromov--Hausdorff distance.}
    \label{table:persistent-distinguishing-main}
    \centering
    \normalsize
    \setlength{\tabcolsep}{1.1pt}
    \renewcommand{\arraystretch}{1.45}
    \begin{tabular}{|c|>{\centering\arraybackslash}m{.23\textwidth}|>{\centering\arraybackslash}m{.25\textwidth}|ccccc|c|}
    \hline
    Pair & Metric space \(X\) & Metric space \(Y\)
        & \(\begin{gathered}\de\\(\TC)\end{gathered}\) & \(\begin{gathered}\de\\(\zcup)\end{gathered}\) & \(\begin{gathered}\de\\(\cat)\end{gathered}\) & \(\begin{gathered}\de\\(\cl)\end{gathered}\)
        & \(\begin{gathered}\di\\(\rmH_*)\end{gathered}\)
        & {\scriptsize\(\begin{gathered}|\diam(X)\\[-2pt]{}-\diam(Y)|\end{gathered}\)} \\
    \hline
    \labelcref{comp:s1-s2-main}
        & \((\square^1,d_\infty)\) & \((\square^2,d_\infty)\)
        & \(1\) &  &  &  & \(1\) &  \\
    \labelcref{comp:s2-wedge11-main}
        & \((\square^2,d_\infty)\) & \((\square^1\vee\square^1,d_\vee)\)
        &  & \(1\) &  &  & \(1\) & \(2\) \\
    \labelcref{comp:wedge11-wedge121-main}
        & \((\square^1\vee\square^2\vee\square^1,d_\vee)\)
        & \((\square^1\vee\square^1,d_\vee)\)
        &  &  &  &  & \(1\) &  \\
    \labelcref{comp:wedge121-torus-main}
        & \((\square^1\vee\square^2\vee\square^1,d_\vee)\)
        & \((\square^1\times\square^1,d_{\max})\)
        &  &  & \(1\) & \(1\) &  & \(2\) \\
    \hline
    \labelcref{comp:tc-zcl-separation-main}
        & \(\begin{gathered}
          \bigl((\square^1\times\square^1)\vee\\
          (\square^1\times\square^1),d_\vee\bigr)
          \end{gathered}\)
        & \(\begin{gathered}
          \bigl((\square^1\times\square^1)\vee\\
          \square^1\vee\square^2\vee\square^1,d_\vee\bigr)
          \end{gathered}\)
        & \(1\) & \(1\) &
        &  &  &  \\
    \hline
    \labelcref{comp:rpn-wedge-main}
        & \((\rp^n,3d_\square^{(2/3)})\) & \((\bigvee_{i=1}^n\square^i,d_\vee)\)
        & \(1\) & \(1\) & \(1\) & \(1\) &  & \(2\) \\
    \labelcref{comp:rpn-wedge-main}
        & \((\rp^n,d_{\rp}^{(2\pi/3)})\) & \((\wedgeS,\tfrac{2\pi}{3\zeta_n}d_\vee^{(\zeta_n)})\)
        & \(\pi/3\) & \(\pi/3\) & \(\pi/3\) & \(\pi/3\) &  &  \\
    \hline
    \end{tabular}
\end{table}

Rows~\ref{comp:s1-s2-main} and~\ref{comp:s2-wedge11-main} of
\cref{table:persistent-distinguishing-main} show that neither
persistent \(\TC\) nor persistent \(\zcup\) has uniformly greater
distinguishing power.
With coefficients in a field of characteristic \(2\),
persistent \(\TC\) distinguishes \(\square^1\) and \(\square^2\),
whereas persistent \(\zcup\) does not; persistent \(\zcup\)
distinguishes \(\square^2\) and \(\square^1\vee\square^1\),
whereas persistent \(\TC\) does not.

Row~\ref{comp:wedge11-wedge121-main} shows that persistent
homology can distinguish spaces for which persistent \(\TC\),
\(\zcup\), \(\cat\), and \(\cl\) agree.
Here, \(\square^1\vee\square^2\vee\square^1\) and
\(\square^1\vee\square^1\) have different degree-\(2\) persistent
homology over any coefficient field.

Row~\ref{comp:wedge121-torus-main} shows that persistent
\(\cat\) and \(\cl\) can distinguish spaces for which persistent
homology, \(\TC\), and \(\zcup\) agree.
Here, the \(\ell_\infty\) realizations of the torus and Mickey Mouse
space give a directly computable counterpart to the geodesic
torus--Mickey Mouse comparison in \cite{memoli2024PersistentCup}.

Row~\ref{comp:tc-zcl-separation-main} shows that persistent
\(\TC\) and \(\zcup\) can distinguish spaces for which persistent
homology, \(\cat\), and \(\cl\) agree.
For the pair \(T\vee T,T\vee M\) in
\cref{fig:cubic-comparison-intro}(c)--(d), persistent \(\TC\) and \(\zcup\) each give the Gromov--Hausdorff
lower bound \(1/2\), whereas persistent homology, \(\cat\), and
\(\cl\) give zero.

The two rows labeled~\ref{comp:rpn-wedge-main} in
\cref{table:persistent-distinguishing-main} show that persistent
\(\TC\), \(\zcup\), \(\cat\), and \(\cl\) can all distinguish
spaces with isomorphic persistent homology modules.
For these realizations of \(\rp^n\) and \(\wedgeS\), with
coefficients in \(\bbF_2\), each invariant gives the
Gromov--Hausdorff lower bound \(1/2\) in the \(\ell_\infty\)
case and \(\pi/6\) in the geodesic case.

\begin{remark}
Persistent \(\TC\) and \(\zcup\) give stronger
Gromov--Hausdorff lower bounds than both persistent homology and
the diameter bound for \(T\vee T\) and \(T\vee M\).
The same holds for all four invariants in the truncated geodesic
comparison of \(\rp^n\) and \(\wedgeS\).
In both comparisons, the persistent homology and diameter bounds
are zero.
% For the \(\ell_\infty\) comparison of \(\rp^n\) and \(\wedgeS\),
% however, the diameter bound \(1\) exceeds the bound \(1/2\)
% from each of the four invariants.
\end{remark}

We now justify the entries in \cref{table:persistent-distinguishing-main}.
For the numerical invariants, the zero erosion distances follow from
the identical persistent profiles in
\cref{subsec:persistent-computations-main}.
To compute the nonzero erosion distances, we use
\cref{prop:de-value-main}, which generalizes the
erosion-distance argument in the proof of
\cite[Proposition~3.7]{memoli2024PersistentCup}.
We also apply this proposition in
\cref{subsubsec:partial-comparisons-main}.

\begin{proposition}
\label{prop:de-value-main}
Let \(\invariant_1,\invariant_2\colon(\Int,\subseteq)\to
(\bbR_{\geq0}\cup\{\infty\},\geq)\) be functors satisfying
\[
\invariant_1([t,s])=c_1\quad\text{for }0<t\leq s\leq a_1,
\qquad
\invariant_2([t,s])=c_2\quad\text{for }0<t\leq s\leq a_2,
\]
where \(0<c_2<c_1\) and \(0<a_1/2<a_2\leq a_1<\infty\). Then
\begin{equation}
\label{eq:de-profile-lower-main}
\de(\invariant_1,\invariant_2)\geq\tfrac{a_1}{2}.
\end{equation}
If both functors vanish on every interval \([t,s]\) with \(s-t>a_1\),
then equality holds.
\end{proposition}

\begin{proof}
Fix \(0\leq\delta<a_1/2\). Choose
\[
0<\epsilon<\min\{a_2-a_1/2,a_1/2-\delta\},
\qquad
t=a_1/2-\epsilon,\quad s=a_1/2+\epsilon.
\]
Then \(0<t\leq s\leq a_2\) and
\(0<t-\delta\leq s+\delta\leq a_1\). Hence
\[
\invariant_2([t,s])=c_2<c_1=\invariant_1([t-\delta,s+\delta]),
\]
so \(\invariant_1\) and \(\invariant_2\) are not \(\delta\)-eroded.
This proves \cref{eq:de-profile-lower-main}.
Under the additional vanishing assumption, for any \(\delta>a_1/2\) and
\([t,s]\in\Int\), the interval \([t-\delta,s+\delta]\) has length
\(s-t+2\delta>a_1\). Both functors therefore vanish on this interval,
so they are \(\delta\)-eroded. It follows that
\(\de(\invariant_1,\invariant_2)=a_1/2\).

For the zero extensions used in the applications, the erosion inequalities
are automatic when \(t\leq\delta\): both values on the expanded interval
\([t-\delta,s+\delta]\) are zero, and the values on \([t,s]\) are
nonnegative. It therefore suffices to check the inequalities for
\(\delta<t\leq s\).
\end{proof}

The erosion distances in \cref{table:persistent-distinguishing-main}
follow from \cref{prop:de-value-main} and the persistent invariant
values in \cref{subsec:persistent-computations-main}.
We give the details for row~\ref{comp:tc-zcl-separation-main};
the erosion distances in
rows~\ref{comp:s1-s2-main}--\ref{comp:wedge121-torus-main}
follow by the same argument.
Let \(X=\bigl((\square^1\times\square^1)\vee(\square^1\times\square^1),d_\vee\bigr)\)
and \(Y=\bigl((\square^1\times\square^1)\vee\square^1\vee\square^2\vee\square^1,d_\vee\bigr)\).
For \(\invariant\in\{\TC,\zcup\}\),
\cref{ex:tc-zcl-persistent-main} gives
\[
\invariant(\VR_\bullet(X))(J)=4,
\qquad
\invariant(\VR_\bullet(Y))(J)=3
\qquad\text{for }J\subset(0,2],
\]
and both invariants vanish on all other intervals.
Applying \cref{prop:de-value-main} with \(c_1=4\), \(c_2=3\),
and \(a_1=a_2=2\), we obtain
\[
\de(\TC)=\de(\zcup)=1.
\]
These spaces have identical persistent \(\cat\) and \(\cl\)
and isomorphic persistent homology modules, so the corresponding
distances are zero.

For the two rows labeled~\ref{comp:rpn-wedge-main},
\cref{table:invariant-values-main} gives
\(\invariant(\rp^n)>\invariant(\wedgeS)>0\) for all four invariants.
The persistent invariant values in
\cref{subsec:persistent-computations-main} satisfy the hypotheses
of \cref{prop:de-value-main}, with \(a_1=a_2=2\) in the
\(\ell_\infty\) case and \(a_1=a_2=2\pi/3\) in the geodesic case.

\smallskip

For persistent homology, the entries follow by comparing barcodes. In
rows~\ref{comp:s1-s2-main}--\ref{comp:wedge11-wedge121-main},
each positive-degree barcode consists of copies of
\((0,2]\), and the multiplicities differ in at least one degree.
For example, the degree-one barcode of \(\square^1\) contains one
such interval, whereas that of \(\square^2\) is empty. Their
interleaving distance is half the interval length, namely \(1\).
The same argument gives \(\di(\rmH_*)=1\) in
rows~\ref{comp:s1-s2-main}--\ref{comp:wedge11-wedge121-main}.
In rows~\ref{comp:wedge121-torus-main}--\ref{comp:rpn-wedge-main}, the barcodes agree in every positive degree, so
\(\di(\rmH_*)=0\). For row~\ref{comp:wedge121-torus-main},
both spaces have barcodes consisting of two copies of \((0,2]\)
in degree \(1\) and one copy in degree \(2\).
For row~\ref{comp:tc-zcl-separation-main}, the common barcodes
are given in \cref{ex:tc-zcl-persistent-main}.
In each of the two rows labeled~\ref{comp:rpn-wedge-main},
the barcodes of both spaces have one bar in each degree \(1,\ldots,n\) and
no bars in degrees \(k>n\).
These bars are \((0,2]\) in the \(\ell_\infty\) case and
\((0,2\pi/3]\) in the geodesic case.

\subsubsection{Comparisons from partial information}
\label{subsubsec:partial-comparisons-main}
\label{subsubsec:untruncated-comparisons-main}

We return to comparison~\ref{comp:rpn-wedge-main} in
\cref{subsec:classical-comparisons-main}, between projective space
and a wedge of spheres, now with untruncated metrics.
Homology and cohomology are taken with coefficients in \(\bbF_2\).
% For both the \(\ell_\infty\) and geodesic models, each of the persistent
% invariants \(\TC\), \(\zcup\), \(\cat\), and \(\cl\) can yield
% Gromov--Hausdorff lower bounds strictly stronger than those from
% persistent homology. For suitable rescalings of the projective-space
% metric, these bounds also improve on the diameter lower bound in
% \cref{eq:dgh-diameter-bounds}.

\subsubsection*{\(\ell_\infty\) models.}
Let \(n\geq2\) and \(\lambda>0\), and consider the metric spaces
\begin{equation}
\label{eq:untruncated-cubical-pair}
X=(\rp^n,\lambda d_\square),
\qquad
Y=\left(\bigvee_{i=1}^n\square^i,d_\vee\right).
\end{equation}
The parameter \(\lambda\) rescales the projective-space metric while
the wedge metric remains fixed. The following estimates hold for every
\(\lambda>0\).

\begin{theorem}
\label{prop:untruncated-cubical-comparison}
For the metric spaces \(X,Y\) in \cref{eq:untruncated-cubical-pair},
with homology and cohomology over \(\bbF_2\), we have, for every \(k\geq1\),
\begin{equation}
\label{eq:cubical-scaled-ph}
\di\left(\rmH_k(\VR_\bullet(X)),\rmH_k(\VR_\bullet(Y))\right)
\leq
\min\left\{
\max\left\{2-\tfrac{2\lambda}{3},\lambda-2\right\},
\max\left\{\tfrac\lambda2,1\right\}
\right\}.
\end{equation}
For each \(\invariant=\TC,\zcup,\cat,\cl\),
\begin{equation}
\label{eq:cubical-scaled-erosion}
\de\left(\invariant(\VR_\bullet(X)),\invariant(\VR_\bullet(Y))\right)
\geq\max\left\{\tfrac\lambda3,
\min\{1,\max\{0,2-\lambda\}\}\right\}.
\end{equation}
Consequently, for \(2<\lambda<3\),
\[
\de\left(\invariant(\VR_\bullet(X)),\invariant(\VR_\bullet(Y))\right)
\geq\tfrac\lambda3>
\sup_{k\geq1}\di\left(\rmH_k(\VR_\bullet(X)),\rmH_k(\VR_\bullet(Y))\right).
\]
\end{theorem}

\Cref{fig:linfty-scale-bounds} illustrates the resulting comparison
of Gromov--Hausdorff lower bounds.
The figure shows one half of the lower bound for \(\de(\invariant)\)
(solid blue) and one half of the upper bound for \(\di(\rmH_*)\)
(dashed orange); stability gives the factor \(1/2\).
The diameter lower bound is \(|\lambda-2|\) (solid green), since
\(\diam(X)=2\lambda\) and \(\diam(Y)=4\). In particular:
\begin{itemize}
\item For \(2<\lambda<3\), each of persistent \(\TC,\zcup,\cat,\cl\)
gives the lower bound \(\lambda/6\), strictly larger than the bound
from persistent homology, which is at most
\(\max\{1-\lambda/3,\lambda/2-1\}\).
\item For \(2<\lambda<12/5\), the lower bound \(\lambda/6\)
also exceeds the diameter lower bound \(|\lambda-2|\).
\end{itemize}

\begin{figure}[!htbp]
\centering
\begin{tikzpicture}
\def\noncriticalopacity{0.3}
\begin{axis}[
  width=0.704\linewidth, height=4.9cm,
  tick label style={font=\normalsize},
  label style={font=\normalsize},
  xmin=0, xmax=4, ymin=0, ymax=2.1,
  xlabel={\(\lambda\)}, ylabel={Bound},
  xtick={0,1,2,2.4,3,4},
  xticklabels={\(0\),\(1\),\(2\),\(12/5\),\(3\),\(4\)},
  ytick={0,0.5,1,1.5,2},
  grid=major, grid style={gray!20},
  legend style={at={(0.5,-0.34)},anchor=north,draw=none,
    font=\normalsize,cells={anchor=west},
    /tikz/every even column/.append style={column sep=1em}},
  legend columns=1,
]
\addplot[line width=1.2pt,blue,
  opacity=\noncriticalopacity,forget plot] coordinates
  {(0.01,0.5) (1,0.5) (1.5,0.25) (4,{4/6})};
\addplot[line width=1.2pt,blue] coordinates
  {(2,{1/3}) (3,0.5)};
\addlegendentry{Lower bound for \(\tfrac12\de(\invariant)\)}
\addplot[line width=1.2pt,orange!80!black,
  dash pattern=on 5pt off 5pt,
  opacity=\noncriticalopacity,forget plot] coordinates
  {(0.01,0.5) (1.5,0.5) (2,{1/3})};
\addplot[line width=1.2pt,orange!80!black,
  dash pattern=on 5pt off 5pt,
  opacity=\noncriticalopacity,forget plot] coordinates
  {(3,0.5) (4,1)};
\addplot[line width=1.2pt,orange!80!black,
  dash pattern=on 5pt off 5pt] coordinates
  {(2,{1/3}) (2.4,0.2) (3,0.5)};
\addlegendentry{Upper bound for \(\tfrac12\di(\rmH_*)\)}
\addplot[thick,green!40!black,
  opacity=\noncriticalopacity,forget plot] coordinates
  {(0.01,1.99) (2,0) (4,2)};
\addplot[thick,green!40!black] coordinates
  {(2,0) (3,1)};
\addlegendentry{\(\lvert\lambda-2\rvert\) diameter lower bound}
\end{axis}
\end{tikzpicture}
\caption{Bounds as functions of \(\lambda>0\) for the metric spaces
in \protect\cref{eq:untruncated-cubical-pair}.
The blue and dashed orange curves show half the bounds in
\protect\cref{eq:cubical-scaled-erosion} and
\protect\cref{eq:cubical-scaled-ph}, respectively, where
\(\invariant\in\{\TC,\zcup,\cat,\cl\}\) and
\(\di(\rmH_*)=\sup_{k\geq1}\di(\rmH_k)\).
The blue and green curves give Gromov--Hausdorff lower bounds;
the dashed orange curve bounds the persistent-homology lower
bound from above.}
\label{fig:linfty-scale-bounds}
\end{figure}

\begin{proof}[Proof of \cref{prop:untruncated-cubical-comparison}]
By \cref{prop:VR-rpn-linfty-main} and \cref{ex:inf_wedge_product},
\begin{equation}\label{eq:VR X and Y}
\VR_t(X)\simeq
\begin{cases}
\rp^n,&0<t\leq2\lambda/3,\\
\ast,&t>\lambda,
\end{cases}
\qquad
\VR_t(Y)\simeq
\begin{cases}
\wedgeS,&0<t\leq2,\\
\ast,&t>2.
\end{cases}
\end{equation}
The structure maps are homotopy equivalences for
\(0<s\leq t\leq2\lambda/3\) in the filtration of \(X\), and for
\(0<s\leq t\leq2\) in that of \(Y\).

We first prove the upper bound for the interleaving distance in
\cref{eq:cubical-scaled-ph}. Fix \(k\geq1\). By
\cref{eq:barcode-isometry}, it suffices to construct partial matchings
between the barcodes of \(\rmH_k(\VR_\bullet(X))\) and
\(\rmH_k(\VR_\bullet(Y))\).

For \(k>n\), the barcode for \(Y\) is empty, and
every bar in the barcode for \(X\) has length at most \(\lambda/3\).
Leaving all bars of \(X\) unmatched therefore gives
\(\di\leq\lambda/6\).
For \(1\leq k\leq n\), the Betti numbers in
\cref{table:invariant-values-main} and the filtration descriptions in
\cref{eq:VR X and Y}, together with the homotopy equivalences of
structure maps, give:
\begin{itemize}
    \item The barcode of \(\rmH_k(\VR_\bullet(Y))\) consists of the single
    bar \((0,2]\).
    \item The barcode of \(\rmH_k(\VR_\bullet(X))\) contains exactly one bar born at zero; denote its death time by
    \(b_k\in[2\lambda/3,\lambda]\). Every remaining bar in the barcode for \(X\) is born at or after
    \(2\lambda/3\) and dies by \(\lambda\), so its length is at most
    \(\lambda/3\).
\end{itemize}
Match the bar \((0,b_k]\)
in the barcode for \(X\) to the bar \((0,2]\) in the barcode for
\(Y\), leaving all other bars of \(X\) unmatched.
Each unmatched bar has length at most \(\lambda/3\), so leaving it
unmatched costs at most \(\lambda/6\).
The cost of the matched bars is
\[
|b_k-2|\leq\max\{2-2\lambda/3,\lambda-2\}.
\]
Since
\(
\max\{2-2\lambda/3,\lambda-2\}
\geq\tfrac12\bigl((2-2\lambda/3)+(\lambda-2)\bigr)=\lambda/6
\),
we obtain the upper bound
\(\max\{2-2\lambda/3,\lambda-2\}\) for \(\di\).
A second matching leaves every bar of both barcodes unmatched, giving the upper bound \(\max\{\lambda/2,1\}\) for \(\di\).
Both upper bounds are at least \(\lambda/6\), so they also
apply when \(k>n\). Taking their minimum proves
\cref{eq:cubical-scaled-ph}.

\smallskip

We next prove the lower bound for the erosion distance in
\cref{eq:cubical-scaled-erosion}.
For each \(\invariant\in\{\TC,\zcup,\cat,\cl\}\),
\cref{table:invariant-values-main} gives
\(\invariant(\rp^n)>\invariant(\wedgeS)>0\).
Using \cref{eq:VR X and Y} and the homotopy equivalences of structure
maps, we obtain:
\begin{itemize}
    \item The persistent invariant for \(X\) equals \(\invariant(\rp^n)\)
    on intervals contained in \((0,2\lambda/3]\) and vanishes on
    intervals not contained in \((0,\lambda]\).
    \item The persistent invariant for \(Y\) equals \(\invariant(\wedgeS)\)
    on intervals contained in \((0,2]\) and vanishes on all other intervals.
\end{itemize}
To prove \(\de\geq\lambda/3\), we rule out every
\(0\leq\delta<\lambda/3\) as an admissible erosion parameter.
Set \(t=\lambda/3\). Then
\([t-\delta,t+\delta]\subset(0,2\lambda/3]\), so
\[
\begin{aligned}
\invariant(\VR_\bullet(X))([t-\delta,t+\delta])
=\invariant(\rp^n)>\invariant(\wedgeS)
\geq\invariant(\VR_\bullet(Y))([t,t]).
\end{aligned}
\]
Thus the erosion inequality fails.

We next prove
\(\de\geq\min\{1,2-\lambda\}\) when \(\lambda<2\).
Let \(0\leq\delta<\min\{1,2-\lambda\}\), and choose
\(\max\{\lambda,\delta\}<t<2-\delta\).
Such a \(t\) exists because \(\delta<1\) and \(\delta<2-\lambda\).
The invariant for \(X\) is zero on \([t,t]\), while that for \(Y\)
is positive on \([t-\delta,t+\delta]\). This violates the erosion
inequality in the reverse direction and proves
\(\de\geq\min\{1,2-\lambda\}\).
For \(\lambda\geq2\), the second term in
\cref{eq:cubical-scaled-erosion} is zero. Combining these lower
bounds proves \cref{eq:cubical-scaled-erosion}.

\smallskip

Finally, for \(2<\lambda<3\),
\(\max\{2-2\lambda/3,\lambda-2\}<\lambda/3\). Combining
\cref{eq:cubical-scaled-ph,eq:cubical-scaled-erosion} therefore gives
the strict comparison between erosion and interleaving distances.
\end{proof}

\subsubsection*{Geodesic models.}
Let \(n\geq2\) and \(\lambda>0\), and consider the metric spaces
\begin{equation}
\label{eq:untruncated-geodesic-pair-main}
X=(\rp^n,\lambda d_{\rp}),
\qquad
Y=(\wedgeS,d_\vee),
\end{equation}
with the quotient and gluing metrics of
\cref{prop:VR-rpn-geodesic-main,ex:wedge-geodesic-main}.
As above, \(\lambda\) rescales only the projective-space metric.
Recall that \(\zeta_n=\arccos(-1/(n+1))\).

\begin{theorem}
\label{prop:scaled-geodesic-comparison}
For the metric spaces \(X,Y\) in \cref{eq:untruncated-geodesic-pair-main},
with homology and cohomology over \(\bbF_2\), set
\[
B_n(\lambda):=\max\left\{
\tfrac{2\pi}{3}|\lambda-1|,
\left|\tfrac{2\lambda\pi}{3}-\zeta_n\right|,
\tfrac{\lambda\pi}{6},\tfrac{\pi-\zeta_n}{2}
\right\}.
\]
For every \(k\geq1\),
\begin{equation}
\label{eq:geodesic-scaled-ph}
\di\left(\rmH_k(\VR_\bullet(X)),\rmH_k(\VR_\bullet(Y))\right)
\leq\min\left\{B_n(\lambda),\tfrac\pi3\max\{\lambda,1\}\right\}.
\end{equation}
For each \(\invariant=\TC,\zcup,\cat,\cl\),
\begin{equation}
\label{eq:geodesic-scaled-erosion}
\de\left(\invariant(\VR_\bullet(X)),\invariant(\VR_\bullet(Y))\right)
\geq\min\left\{\tfrac{\lambda\pi}{3},\zeta_n\right\}.
\end{equation}
If \(\max\{\tfrac23,\tfrac{3(\pi-\zeta_n)}{2\pi}\}<\lambda<3\zeta_n/\pi\), then
\[
\de\left(\invariant(\VR_\bullet(X)),\invariant(\VR_\bullet(Y))\right)
\geq\tfrac{\lambda\pi}{3}>
\sup_{k\geq1}\di\left(\rmH_k(\VR_\bullet(X)),\rmH_k(\VR_\bullet(Y))\right).
\]
\end{theorem}

\Cref{fig:geodesic-scale-bounds} illustrates the resulting comparison
of Gromov--Hausdorff lower bounds for \(n=2\).
The figure shows one half of the lower bound for \(\de(\invariant)\)
(solid blue) and one half of the upper bound for \(\di(\rmH_*)\)
(dashed orange); stability gives the factor \(1/2\).
The diameter lower bound is \(\tfrac\pi2|\lambda-2|\) (solid green), since
\(\diam(X)=\lambda\pi\) and \(\diam(Y)=2\pi\). In particular:
\begin{itemize}
\item For \(\max\{\tfrac23,\tfrac{3(\pi-\zeta_n)}{2\pi}\}<\lambda<3\zeta_n/\pi\), each of persistent
\(\TC,\zcup,\cat,\cl\) gives the lower bound \(\lambda\pi/6\),
strictly larger than the bound from persistent homology.
In particular, \(\lambda=1\) recovers the bound \(\pi/6\),
while persistent homology gives at most
\((\pi-\zeta_n)/4<\pi/8\).
\item For \(3/2<\lambda<3\zeta_n/\pi\), the lower bound
\(\lambda\pi/6\) also exceeds the diameter lower bound
\(\tfrac\pi2|\lambda-2|\). This interval is nonempty for every
\(n\geq2\), since \(\zeta_n>\pi/2\).
\end{itemize}

\begin{figure}[!htbp]
\centering
\begin{tikzpicture}
\def\geodesicdimension{2}
\def\noncriticalopacity{0.3}
\pgfmathsetmacro{\zetaratio}{acos(-1/(\geodesicdimension+1))/180}
\pgfmathsetmacro{\criticalstart}{max(2/3,3*(1-\zetaratio)/2)}
\pgfmathsetmacro{\criticalend}{3*\zetaratio}
\begin{axis}[
  width=0.704\linewidth, height=4.9cm,
  tick label style={font=\normalsize},
  label style={font=\normalsize},
  xmin=0, xmax=4, ymin=0, ymax=3.3,
  xlabel={\(\lambda\)}, ylabel={Bound},
  xtick={0,1,1.5,2,3,4}, ytick={0,1,2,3},
  grid=major, grid style={gray!20},
  legend style={at={(0.5,-0.34)},anchor=north,draw=none,
    font=\normalsize,cells={anchor=west},
    /tikz/every even column/.append style={column sep=1em}},
  legend columns=1,
  declare function={
    erosion(\x)=pi*min(\x/6,\zetaratio/2);
    homology(\x)=pi*min(max(max(abs(2*(\x-1)/3),abs(2*\x/3-\zetaratio)),max(\x/6,(1-\zetaratio)/2)),max(\x,1)/3)/2;
  },
]
\addplot[line width=1.2pt,blue,opacity=\noncriticalopacity,
  forget plot,domain=0.01:4,samples=301] {erosion(x)};
\addplot[line width=1.2pt,blue,
  domain=\criticalstart:\criticalend,samples=2] {erosion(x)};
\addlegendentry{Lower bound for \(\tfrac12\de(\invariant)\)}
\addplot[line width=1.2pt,orange!80!black,dash pattern=on 5pt off 5pt,
  opacity=\noncriticalopacity,forget plot,
  domain=0.01:\criticalstart,samples=101] {homology(x)};
\addplot[line width=1.2pt,orange!80!black,dash pattern=on 5pt off 5pt,
  opacity=\noncriticalopacity,forget plot,
  domain=\criticalend:4,samples=201] {homology(x)};
\addplot[line width=1.2pt,orange!80!black,dash pattern=on 5pt off 5pt,
  domain=\criticalstart:\criticalend,samples=201] {homology(x)};
\addlegendentry{Upper bound for \(\tfrac12\di(\rmH_*)\)}
\addplot[thick,green!40!black,opacity=\noncriticalopacity,forget plot]
  coordinates {(0.01,{0.995*pi}) (2,0) (4,pi)};
\addplot[thick,green!40!black,domain=\criticalstart:\criticalend,samples=2]
  {pi*abs(x-2)/2};
\addlegendentry{\(\tfrac\pi2|\lambda-2|\) diameter lower bound}
\end{axis}
\end{tikzpicture}
\caption{Bounds as functions of \(\lambda>0\) for the metric spaces
in \protect\cref{eq:untruncated-geodesic-pair-main}, with \(n=2\).
The blue and dashed orange curves show half the bounds in
\protect\cref{eq:geodesic-scaled-erosion} and
\protect\cref{eq:geodesic-scaled-ph}, respectively, where
\(\invariant\in\{\TC,\zcup,\cat,\cl\}\) and
\(\di(\rmH_*)=\sup_{k\geq1}\di(\rmH_k)\).
The blue and green curves give Gromov--Hausdorff lower bounds;
the dashed orange curve bounds the persistent-homology lower
bound from above.
Full opacity marks \(\max\{\tfrac23,\tfrac{3(\pi-\zeta_2)}{2\pi}\}<\lambda<3\zeta_2/\pi\), where each of
persistent \(\TC,\zcup,\cat,\cl\) gives a stronger
Gromov--Hausdorff lower bound than persistent homology.}
\label{fig:geodesic-scale-bounds}
\end{figure}

\begin{proof}[Proof of \cref{prop:scaled-geodesic-comparison}]
We first prove \cref{eq:geodesic-scaled-ph} using the two barcode
matchings from the proof of \cref{prop:untruncated-cubical-comparison}.
Fix \(k\geq1\). The barcode computations in
\cite[Sections~4.3 and~4.6]{medina2025persistent} and rescaling show
that, for \(k>n\), every bar has length at most \(\lambda\pi/3\)
for \(X\) and \(\pi-\zeta_n\) for \(Y\).
Leaving all bars unmatched therefore gives
\(\di\leq\max\{\lambda\pi/6,(\pi-\zeta_n)/2\}\).
For \(1\leq k\leq n\), the same computations give:
\begin{itemize}
    \item The barcode of \(\rmH_k(\VR_\bullet(Y))\) contains exactly
    one bar born at zero, with death \(\zeta_k\). Every remaining
    bar has length at most \(\pi-\zeta_n\).
    \item The barcode of \(\rmH_k(\VR_\bullet(X))\) contains exactly
    one bar born at zero, with death \(2\lambda\pi/3\). Every remaining
    bar has length at most \(\lambda\pi/3\).
\end{itemize}
Matching the bars born at zero and leaving all other bars unmatched
costs at most
\[
\max\left\{\left|\tfrac{2\lambda\pi}{3}-\zeta_k\right|,
\tfrac{\lambda\pi}{6},\tfrac{\pi-\zeta_n}{2}\right\}
\leq B_n(\lambda),
\]
since \(\zeta_n\leq\zeta_k\leq\zeta_1=2\pi/3\).
Leaving every bar unmatched costs at most
\(\tfrac\pi3\max\{\lambda,1\}\), since bars of \(X\) have length
at most \(2\lambda\pi/3\) and those of \(Y\) at most \(2\pi/3\).
Both bounds also apply when \(k>n\). Taking their minimum and using
\cref{eq:barcode-isometry} proves \cref{eq:geodesic-scaled-ph}.

\smallskip

We next prove the erosion bound in \cref{eq:geodesic-scaled-erosion}.
Fix \(\invariant\in\{\TC,\zcup,\cat,\cl\}\).
By \cref{lem:TC-zcup-RPn-wedge-main,ex:geodesic-sphere-invariants-main}
and rescaling, the persistent invariants for \(X\) and \(Y\) equal
\(\invariant(\rp^n)\) and \(\invariant(\wedgeS)\) on intervals
contained in \((0,2\lambda\pi/3]\) and \((0,\zeta_n]\), respectively.

Set \(t=\min\{\lambda\pi/3,\zeta_n\}\). For \(0\leq\delta<t\),
we have \([t-\delta,t+\delta]\subset(0,2\lambda\pi/3]\) and
\(0<t\leq\zeta_n\), so
\[
\invariant(\VR_\bullet(X))([t-\delta,t+\delta])
=\invariant(\rp^n)
>\invariant(\wedgeS)
=\invariant(\VR_\bullet(Y))([t,t]).
\]
The strict inequality follows from \cref{table:invariant-values-main}.
Thus no \(\delta<t\) satisfies the erosion inequalities, proving
\cref{eq:geodesic-scaled-erosion}.

\smallskip

Finally, if \(\max\{\tfrac23,\tfrac{3(\pi-\zeta_n)}{2\pi}\}<\lambda<3\zeta_n/\pi\), each term defining
\(B_n(\lambda)\) is strictly less than \(\lambda\pi/3\):
the first uses \(2/3<\lambda<2\), the second uses
\(\zeta_n/\pi<\lambda<3\zeta_n/\pi\), and the last uses
\(\lambda>3(\pi-\zeta_n)/(2\pi)\). Combining
\cref{eq:geodesic-scaled-ph,eq:geodesic-scaled-erosion} therefore gives
the strict comparison between erosion and interleaving distances.
\end{proof}

\rpnvswedge*

\begin{proof}
Set \(\lambda=1\) in \cref{eq:geodesic-scaled-erosion}.
Since \(\zeta_n>\pi/2\), this gives the lower bound
\(\pi/3\) in \cref{eq:erosion-comp}. The upper bound follows
from \cref{thm:stability}.
\end{proof}

\begin{remark}\label{rmk:exact-values-main}
The estimate in \cref{prop:de-value-main} depends on the strict
inequality \(c_1>c_2\) and the scale ranges, but not on the
difference \(c_1-c_2\).
Thus, more precise classical values of \(\TC(\rp^n)\) or
\(\zcup(\rp^n)\) do not sharpen the lower bounds obtained
by this argument.
\end{remark}

\section{Concluding remarks}
\label{sec:concluding-remarks}

The lower bounds in
\cref{eq:cubical-scaled-erosion,eq:geodesic-scaled-erosion}
use partial information about the Vietoris--Rips filtrations.
Further information about the persistent invariants on
intervals not covered by our computations may yield sharper
lower bounds for the erosion distances and, consequently,
stronger Gromov--Hausdorff lower bounds for these pairs
of metric spaces.
\printbibliography

\end{document}